\pdfoutput=1
\documentclass[11pt,a4paper,draft]{amsart} 

\usepackage[paper=a4paper,
            marginparwidth=30.5mm,    
            marginparsep=1.5mm,       
            margin=25mm,              
			marginparwidth=0mm,
			marginparsep=2mm,
			top=20mm, left=25mm, right=25mm, bottom=15mm,
          includemp]{geometry}

\usepackage[T1]{fontenc}
\usepackage[english]{babel}
\usepackage[dvips]{graphics}
\usepackage[pdftex]{graphicx}
\usepackage{xcolor}
\usepackage{url}
\usepackage{stmaryrd}
\usepackage{rotating}
\usepackage{longtable}
\usepackage{multicol}
\usepackage{etoolbox}
\usepackage{tikz}
\usepackage{tikz-cd}
\usepackage{tikz-3dplot}
\usepackage{ifthen}

\usepackage{amsmath}
\usepackage{amsthm}
\usepackage{amssymb}
\usepackage{enumerate}
\usepackage[shortlabels]{enumitem} \setlist{leftmargin=7mm} \setlist[enumerate,1]{label={\rm (\alph*)}}
\usepackage[all,poly,necula]{xy} 
\usepackage{mathrsfs}
\usepackage{mathtools}
\usepackage[reftex]{theoremref}

\usepackage{wasysym}

\usepackage{chngcntr}

\usepackage{adjustbox}
\usepackage{bbm}

\theoremstyle{plain}
\newtheorem{theorem}{Theorem}[section]
\renewcommand{\thetheorem}{\arabic{section}.\arabic{theorem}}
\newtheorem{theoremi}{Theorem}[section]

\newtheorem{proposition}[theorem]{Proposition}
\newtheorem{lemma}[theorem]{Lemma}
\newtheorem{corollary}[theorem]{Corollary}

\theoremstyle{definition}

\newtheorem{example}[theorem]{Example}

\theoremstyle{remark}
\newtheorem{remark}[theorem]{Remark}

\newtheoremstyle{boldNote}{}{}{}{}{\bfseries}{.}{5pt plus 1pt minus 1pt} 
  {\thmname{#1}\thmnumber{ #2}\thmnote{ (#3)}}  
\theoremstyle{boldNote}
\newtheorem{definition}[theorem]{Definition}

\counterwithin{equation}{theorem}
\newcommand{\theloctheorem}{\thetheorem}
\newcommand{\eqloc}[1]{\renewcommand{\theloctheorem}{\ref{#1}}}
\AtEndEnvironment{proof}{\renewcommand{\theloctheorem}{\thetheorem}}

\newtheoremstyle{clm}{5pt}{1pt}{\itshape}{.2cm}{\itshape\sc}{.}{ }{}{}
\theoremstyle{clm}
\newtheorem{claim}{Claim}[section]
\makeatletter
\newenvironment{cproof}[1][Proof of claim]{\par
\pushQED{\qed}%
\normalfont \topsep1\p@\@plus1\p@\relax
\trivlist\item\relax{\hspace{.2cm}\itshape#1\@addpunct{.}}\hspace\labelsep\ignorespaces}{%
\popQED\endtrivlist\@endpefalse\par\addvspace{5pt}
}
\AtBeginEnvironment{proof}{\setcounter{equation}{0}}
\makeatother

\AtBeginEnvironment{figure}{\stepcounter{theorem}}

 \DeclareMathOperator{\Id}{\mathsf{Id}}
\let\im\relax
\DeclareMathOperator{\im}{\mathsf{im}}
\DeclareMathOperator{\coker}{\mathsf{coker}}
\let\ker\relax
\DeclareMathOperator{\ker}{\mathsf{ker}}

\DeclareMathOperator{\Hom}{\mathsf{Hom}}

\DeclareMathOperator{\Ext}{\mathsf{Ext}}

\DeclareMathOperator{\End}{\mathsf{End}}

\DeclareMathOperator{\irr}{\mathsf{irr}}

\DeclareMathOperator{\Fac}{\mathsf{Fac}}

\DeclareMathOperator{\Mod}{\mathsf{Mod}}
\let\mod\relax
\DeclareMathOperator{\mod}{\mathsf{mod}}

\DeclareMathOperator{\fl}{\mathsf{f.\!l.}}

\DeclareMathOperator{\tors}{\mathsf{tors}}

\newcommand{\un}[2]{{{}_{#1} \mathbf{1}_{#2}}}
\newcommand{\st}[1]{\underline{\smash{#1}}}
\newcommand{\maxstr}[1][]{\mathop{\mathsf{maxNC}_{#1}}(Q, R)}
\newcommand{\torsQ}{\mathop{\mathsf{tors}}(Q, R)}

\DeclareMathOperator{\TTT}{\mathsf{T}}
\DeclareMathOperator{\Filt}{\mathsf{Filt}}

\DeclareMathOperator{\brick}{\mathsf{brick}}

\newcommand{\us}{\mathbf{s}}
\newcommand{\ut}{\mathbf{t}}

\newcommand{\xto}{\xrightarrow}
\newcommand{\Tors}{\TTT}

\newcommand{\extr}[1]{[ #1 \rangle}
\newcommand{\extir}[1]{[ #1 \langle}
\newcommand{\extl}[1]{\rangle #1 ]}
\newcommand{\extil}[1]{\langle #1 ]}

\renewcommand{\leq}{\leqslant}
\renewcommand{\geq}{\geqslant}
\newcommand{\Acal}{\mathcal{A}}
\newcommand{\Ar}{\mathscr{A}}

\newcommand{\Mr}{\mathscr{M}}

\newcommand{\Tr}{\mathscr{T}}

\newcommand{\Sr}{\mathscr{S}}

\newcommand{\Xr}{\mathscr{X}}

\newcommand{\Ur}{\mathscr{U}}

\newcommand{\Z}{\mathbb{Z}}

\renewcommand{\epsilon}{\varepsilon}

\newcommand{\tens}{\otimes}

\renewcommand{\phi}{\varphi}

\renewcommand{\bar}[1]{\overline{#1}}

\newcommand{\comment}[1]{}

\newcommand{\join}{\vee}

\newcommand{\inj}{\hookrightarrow}

\renewcommand{\tilde}[1]{\smash{\widetilde{#1}}}

\newcommand{\forloop}[5][1] { \setcounter{#2}{#3} \ifthenelse{#4} { #5 \addtocounter{#2}{#1} \forloop[#1]{#2}{\value{#2}}{#4}{#5} }{ } } 
\newcounter{i}
\newcounter{j}
\newcommand{\dets}[4]{
    \displaystyle \left|\;
    \begin{matrix}
      \forloop{i}{1}{\value{i} < 7}{
       \forloop{j}{1}{\value{j} < 7}{
        \ifthenelse{\value{i} < #1 \or \value{i} > #2 \or \value{j} < #3 \or \value{j} > #4}{\ifthenelse{\value{i}<\value{j}}{\cdot}{\ifthenelse{\value{i} = \value{j}}{1}{}}}{\bullet} \ifthenelse{\value{j}<6}{&}{}
       } \ifthenelse{\value{i}<6}{\\}{}
      }
    \end{matrix}
    \;\right|^{\vphantom{M^{M^M}}}_{\vphantom{M_{M_M}}}}

\makeatletter
\def\overarrowb@#1#2#3{\vbox{\vspace*{-.3ex}\ialign{##\crcr\vspace*{-.3ex}#1#2\crcr
 \noalign{\nointerlineskip}$\m@th\hfil#2#3\hfil$\crcr}}}
\newcommand{\amsvect}{%
  \mathpalette{\overarrowb@\rightarrowfill@}}
\makeatother

\renewcommand{\mod}{\operatorname{mod}}

\newsavebox\locboxinminipage
\newlength\locboxinminipagel

\def\newboxedcommand#1#2 
{%
 \def\newboxedcommandlocala##1##2.{##2}%
 \edef\newboxedcommandlocalb{\expandafter\newboxedcommandlocala\string#1.}%
 \expandafter\newsavebox\csname\newboxedcommandlocalb savebox\endcsname%
 \expandafter\sbox\csname\newboxedcommandlocalb savebox\endcsname{#2}%
 \expandafter\newlength\csname\newboxedcommandlocalb largeurbox\endcsname%
 \expandafter\settowidth\csname\newboxedcommandlocalb largeurbox\endcsname{\usebox{\csname\newboxedcommandlocalb savebox\endcsname}}%
 \edef#1{\noexpand\begin{minipage}{\csname\newboxedcommandlocalb largeurbox\endcsname}\usebox{\csname\newboxedcommandlocalb savebox\endcsname}\noexpand\end{minipage}}%
}

\usepackage{color}
\newcommand{\margincolor}{red}      
\definecolor{darkgreen}{rgb}{0,0.7,0}
      
\newcounter{margincounter}
\newcommand{\marginnum}{
\ifnum\value{margincounter}<10
\textcolor{\margincolor}{\begin{picture}(0,0)\put(2.2,2.4){\circle{9}}\end{picture}\footnotesize\arabic{margincounter}}
\else\ifnum\value{margincounter}<100
\textcolor{\margincolor}{\begin{picture}(0,0)\put(4.256,2.5){\circle{11}}\end{picture}\footnotesize\arabic{margincounter}}
\else
\textcolor{\margincolor}{\begin{picture}(0,0)\put(6.8,2.5){\circle{14}}\end{picture}\footnotesize\arabic{margincounter}}
\fi\fi
}

\newcommand{\cross}[5]
{
 \left. \begin{array}{r} {#1} \\ {#4} \end{array} \hspace{-3pt} \right\rangle {\ifx\hfuzz#2\hfuzz \hspace{-6pt} \else \hspace{-2pt} #2 \hspace{-2pt} \fi} \left\langle \hspace{-3pt} \begin{array}{l} {#3} \\ {#5} \end{array} \right.
}

\newcommand{\crosswn}[7]
{
 \left. \begin{array}{rr} {#1} = \hspace{-5pt} & {#2} \\ {#5} = \hspace{-5pt} & {#6} \end{array} \hspace{-3pt} \right\rangle {\ifx\hfuzz#3\hfuzz \hspace{-6pt} \else \hspace{-2pt} #3 \hspace{-2pt} \fi}  \left\langle \hspace{-3pt}  \begin{array}{l} {#4} \\ {#7} \end{array} \right.
}

\def\clap#1{\hbox to 0pt{\hss#1\hss}}

\numberwithin{equation}{subsection}

\definecolor{darkblue}{rgb}{0,0,0.7} 
\newcommand{\defn}[1]{\textsl{\color{darkblue} #1}}

\newcommand{\fin}{\mathsf{fin}}
\newcommand{\blambda}{\pmb{\lambda}}
\newcommand{\nstr}{\mathsf{nstr}}
\newcommand{\mt}{\mathtt}

\usetikzlibrary{shapes,arrows}
\usetikzlibrary{patterns}
\usetikzlibrary{decorations.pathmorphing}
\newcommand{\penta}[2]{
\draw[thick] (#1) circle (2);
\draw[shift={(#1)}, #2] (0:2) -- (144:2) (0:2) -- (-144:2);
\draw (#1) ++(0:2) node{$\bullet$};
\draw (#1) ++(72:2) node{$\bullet$};
\draw (#1) ++(144:2) node{$\bullet$};
\draw (#1) ++(-144:2) node{$\bullet$};
\draw (#1) ++(-72:2) node{$\bullet$};
}

\newcommand{\cyl}[1]{
\draw[thick] (#1) circle (2);
\filldraw [pattern=north east lines, thick] (#1) circle (0.5);
\draw (#1)+(0,0.5) node {$\bullet$};
\draw (#1)+(0,-2) node {$\bullet$}; 
}

\newcommand{\hrt}[1]{
\draw[gray, dotted, shift={(0,-2)}] (#1) .. controls +(1,0.9) and +(0,-1.5) .. +(1.5,2.5);
\draw[gray, dotted, shift={(0,0.5)}] (#1) .. controls +(0.3,1.5) and +(0,1) .. +(1.5,0);
\draw[gray, dotted, shift={(0,-2)}] (#1) .. controls +(-1,0.9) and +(0,-1.5) .. +(-1.5,2.5);
\draw[gray, dotted, shift={(0,0.5)}] (#1) .. controls +(-0.3,1.5) and +(0,1) .. +(-1.5,0);
}

\newcommand{\strPa}[2]{
\spiral[#2](#1)(275:160)(1.3:1)[1];
\draw[#2, shift={(#1)}] (-0.935,0.36) .. controls (-0.95,0.1) and (-0.9,-0.1) .. (-0.5,-0.1);
\draw[#2, shift={(#1)}] (0.115,-1.295) .. controls (-0.4,-1.38) and (-1.1,-1.4) .. (-1.417,-1.4);
}
\newcommand{\strPb}[2]{
\draw[#2, shift={(#1)}] (-0.4,-0.3) -- +(-1.56,0);
}
\newcommand{\strPba}[2]{
\spiral[#2](#1)(275:140)(1.5:0.8)[2];
\draw[#2, shift={(#1)}] (-0.61,0.52) .. controls (-0.8,0.2) and (-0.8,0) .. (-0.5,0);
\draw[#2, shift={(#1)}] (0.17,-1.487) .. controls +(-0.3,-0.04) and (-0.8,-1.6) .. (-1.2,-1.6);
}

\newcommand{\strPab}[2]{
\spiral[#2](#1)(275:120)(1.7:0.65)[3];
\draw[#2, shift={(#1)}] (-0.3,0.58) .. controls (-0.55,0.42) and +(-0.29,0.05) .. (-0.5,0.1);
\draw[#2, shift={(#1)}] (0.2,-1.685) .. controls +(-0.3,-0.05) and +(0.3,0) .. (-0.9,-1.78);
}

\newcommand{\strInftyOuter}[2]{
\spiral[#2](#1)(280:120)(1.75:1.43)[3];
\spiral[dashed, #2](#1)(122:170)(1.43:1.41)[0];
\draw[#2, shift={(#1)}] (0.32,-1.72) .. controls +(-0.4,-0.05) and +(0.2,0.25) .. (-0.58,-1.91);
}

\newcommand{\strInftyInner}[2]{
\spiral[#2](#1)(250:100)(1.1:0.78)[3];
\spiral[dashed, #2](#1)(245:180)(1.1:1.12)[0];
\draw[#2, shift={(#1)}] (-0.1,0.78) .. controls (-0.6,0.6) and (-0.4,0.3) .. (-0.4,0.3);
}

\newcommand{\cylNullStr}[2]{
\draw[#2, rotate=-90] (#1)+(8.5:2) arc (115:245:0.34);
\draw[#2, rotate=90] (#1)+(-39:0.5) arc (-90:94:0.32);
}

\newcommand\spiral{} 

\def\spiral[#1](#2)(#3:#4)(#5:#6)[#7]{
\pgfmathsetmacro{\domain}{#4+#7*360}
\pgfmathsetmacro{\growth}{180*(#6-#5)/(pi*(\domain-#3))}
\draw [#1,
       shift={(#2)},
       domain=#3*pi/180:\domain*pi/180,
       variable=\t,
       smooth,
       samples=int(\domain/5)] plot ({\t r}: {#5+\growth*\t-\growth*#3*pi/180})
}

\title{Classifying torsion classes of gentle algebras}

\author{Aaron Chan and Laurent Demonet}

\begin{document}

\begin{abstract}
For a finite-dimensional gentle algebra, it is already known that the functorially finite torsion classes of its category of finite-dimensional modules can be classified using a combinatorial interpretation, called maximal non-crossing sets of strings, of the corresponding support $\tau$-tilting module (or equivalently, two-term silting complexes).  In the topological interpretation of gentle algebras via marked surfaces, such a set can be interpreted as a dissection (or partial triangulation), or equivalently, a lamination that does not contain a closed curve.  We will refine this combinatorics, which gives us a classification of torsion classes in the category of finite length modules over a (possibly infinite-dimensional) gentle algebra.  As a consequence, our result also unifies the functorially finite torsion class classification of finite-dimensional gentle algebras with certain classes of special biserial algebras - such as Brauer graph algebras.
\end{abstract}

\maketitle

\section{Introduction}

Finite-dimensional gentle algebras has always been one of the central themes subject in the representation theory of finite-dimensional algebras.  They constitute to one of the most interesting classes of algebras as their representation theory are made up from type $\mathbb{A}$ quiver - called \emph{string modules} - and type $\widetilde{\mathbb{A}}$ quiver  - called \emph{bands modules}; see \cite{WW} or Section \ref{sec:gentlealg}.  


A gentle algebra $\Lambda=\Lambda_{Q,R}$ is uniquely described by a gentle quiver $(Q,R)$, where $R$ is the generating set of relations, all of which are quadratic monomial; see Definition \ref{def:gentle} for details.  A recent trend, coming from developments in cluster theory \cite{FST, ABCP} as well as symplectic topology \cite{HKK}, is to identify a gentle quiver with a \emph{dissection of marked surface}.  For simplicity, we use in this introduction the topological model from cluster theory.  That is, the vertices of $Q$ are given by the (non-boundary) arcs of a triangulation (without self-folded triangles) on a marked bordered (compact orientable real) surface without punctures $\Sigma$, and the arrows are in bijections with the (oriented) angles between arcs.  In Figure \ref{fig:A2,Kron0}, we show the case of the $\mathbb{A}_2$-quiver on the left, and the case of the Kronekcer quiver (i.e. $\widetilde{\mathbb{A}}_1$-quiver) on the right.

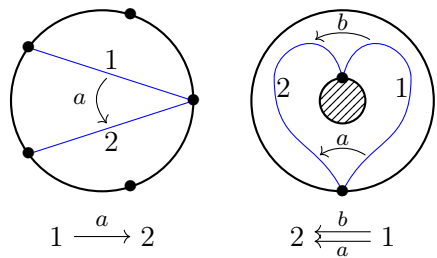
\begin{figure}[!ht]
\centering
\begin{tikzpicture}[scale=0.6]
\penta{0,0}{blue};
\node at (0.2,0.9) {1}; \node at (0.2,-0.9) {2};
\draw[->] (0.1,0.5) .. controls (-0.2,0.2) and (-0.2,-0.1) .. (0.1,-0.5) node [midway, left] {\footnotesize $a$};
\node (u1) at (-1,-3) {1}; \node (u2) at (1,-3) {2};
\draw[->] (u1) -- (u2) node [midway, above] {\footnotesize $a$};

\begin{scope}[xshift=150] 
\draw[blue, shift={(0,-2)}] (0,0) .. controls +(0.6,1.2) and +(0,-1.5) .. +(1.5,2.5);
\draw[blue, shift={(0,0.5)}] (0,0) .. controls +(0.3,1.3) and +(0,0.7) .. +(1.5,0);
\node at (1.3,0.3) {1};
\draw[blue, shift={(0,-2)}] (0,0) .. controls +(-0.6,1.2) and +(0,-1.5) .. +(-1.5,2.5);
\draw[blue, shift={(0,0.5)}] (0,0) .. controls +(-0.3,1.3) and +(0,0.7) .. +(-1.5,0);
\node at (-1.3,0.3) {2};
\cyl{0,0};
\draw[->] (0.6,1.35) .. controls (0.3,1.55) and (-0.3,1.55) .. (-0.6,1.35) node [inner sep=1, midway,above] {\footnotesize $b$};
\draw[->] (0.5,-1.2) .. controls (0.2,-1) and (-0.2,-1) .. (-0.5,-1.2) node [inner sep=1, midway,above] {\footnotesize $a$};
\node (v1) at (1,-3) {1}; \node (v2) at (-1,-3) {2};
\draw[->] ([yshift=3]v1.west) -- node [inner sep=2, midway,above] {\scriptsize $b$} ([yshift=3]v2.east);
\draw[->] ([yshift=-3]v1.west) -- node [inner sep=2, midway,below] {\scriptsize $a$} ([yshift=-3]v2.east);
\end{scope}
\end{tikzpicture}
\caption{Dissection and gentle quiver}\label{fig:A2,Kron0}
\end{figure}

Under this topological model, a string module corresponds to an arc (i.e. curve connecting marked points), and a band module corresponds to a closed curve equipped with some power of an irreducible Laurent polynomial\footnote{Specifying a `power of an irreducible Laurent polynomial' is equivalent to specifying a Jordan block, when the underlying field is algebraically closed}.  This has become a very powerful tool in understanding homological properties of gentle algebras.  For example, adding some geometric ingredients to a variant of this topological model yields a complete classification of derived equivalences classes as well as silting complexes \cite{APS}.

The homological behaviour that we are interested in this article is \emph{torsion classes} of the category of finite-dimensional $\Lambda$-modules.  In the topological model of the cluster theory setting, it is known that one can classify \emph{functorially finite} torsion classes by triangulations on the associated surface via the following sequence of bijections.
\[
\xymatrix@R=8pt@C=0pt{
*[l]{\{\text{triangulations of }\Sigma\}\;\;} \ar@{<->}[r]^{\mbox{\small\cite{FST}}} & *[r]{\;\;\{\text{clusters of }\mathcal{A}_{\Sigma}\}} &\\
\ar@{<->}[r]^{\mbox{\cite{BZ}}} & *[r]{\;\;\{\text{cluster tilting objects in } \mathcal{C}_{\Sigma}\}} \\
\ar@{<->}[r]^{\mbox{\small\cite{AIR}}} & *[r]{\;\;\{\text{functorially finite torsion classes}\}.}
}\]
The left-hand side of Figure \ref{fig:A2,Kron1} shows some examples for the $\mathbb{A}_2$ quiver and the Kronecker quiver.  Classification of functorially finite torsion classes for all finite-dimensional gentle algebras are known in \cite[Cor 5.7]{EJR}, \cite[Thm 5.1]{BDMTY},  \cite[Thm 2.6]{PPP17} using various different combinatorial model.  We remark that the combinatorics used in these cited works can be regarded as a subset of ours.  

The aim of this article is to classify \emph{all} torsion classes for \emph{any} gentle algebras, regardless of finite-dimensionality, by replacing triangulations with something more general.

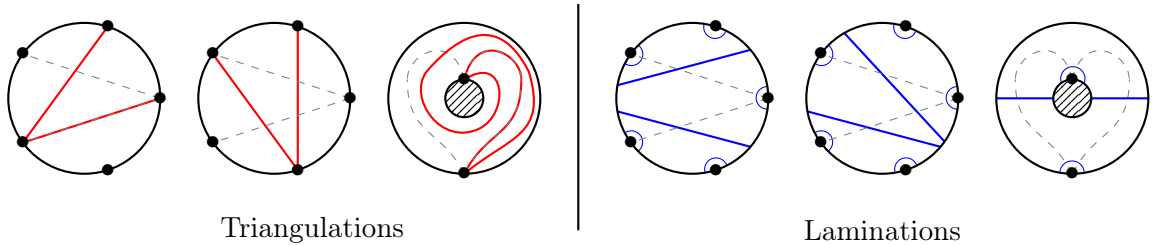
\begin{figure}[!ht]
\centering
\begin{tikzpicture}[scale=0.5]
\begin{scope}[shift={(-16,0)}]
\begin{scope}[shift={(10,0)}]
\draw[dashed, gray, shift={(0,-2)}] (0,0) .. controls +(-0.6,1.2) and +(0,-1.5) .. +(-1.5,2.5);
\draw[dashed, gray, shift={(0,0.5)}] (0,0) .. controls +(-0.3,1.3) and +(0,0.7) .. +(-1.5,0);
\draw[red, thick, shift={(0,-2)}] (0,0) .. controls +(0.6,1.2) and +(0,-1.5) .. +(1.5,2.5);
\draw[red, thick, shift={(0,0.5)}] (0,0) .. controls +(0.3,1.3) and +(0,0.7) .. +(1.5,0);
\draw[dashed, gray, shift={(0,-2)}] (0,0) .. controls +(0.6,1.2) and +(0,-1.5) .. +(1.5,2.5);
\draw[dashed, gray, shift={(0,0.5)}] (0,0) .. controls +(0.3,1.3) and +(0,0.7) .. +(1.5,0);
\draw[thick,red] (0,0.5) .. controls (1.1,1.2) and (1.2,-0.4) .. (0.4,-0.8) .. controls (-0.6,-1.3) and (-1.5,0) .. (-1,0.7) .. controls (-0.4,1.4) and (0.8,2.5) .. (1.7,0.7) .. controls (2.2,-1) and (1.2,-1) .. (0,-2);
\cyl{0,0};
\end{scope}
\begin{scope}[shift={(0,0)}]
\draw[thick, red] (72:2) -- (-144:2) (0:2) -- (-144:2);
\penta{0,0}{dashed, gray};
\end{scope}
\begin{scope}[shift={(5,0)}]
\draw[thick, red] (72:2) -- (-72:2) (-72:2) -- (144:2);
\penta{0,0}{dashed, gray};
\end{scope}
\end{scope}



\begin{scope}[shift={(0,0)}]
\draw[thick, blue] (-40:2) -- (-170:2) (40:2) -- (170:2);
\draw[blue] (9:2) arc (91:269:0.3);
\draw[blue, rotate=72] (9:2) arc (91:269:0.3);\draw[blue, rotate=-72] (9:2) arc (91:269:0.3);
\draw[blue, rotate=144] (8.5:2) arc (91:269:0.3);\draw[blue, rotate=-144] (8.5:2) arc (91:269:0.3);
\penta{0,0}{dashed, gray};
\end{scope}
\begin{scope}[shift={(5,0)}]
\draw[thick, blue] (-40:2) -- (-170:2) (-35:2) -- (120:2);
\draw[blue] (9:2) arc (91:269:0.3);
\draw[blue, rotate=72] (9:2) arc (91:269:0.3);\draw[blue, rotate=-72] (9:2) arc (91:269:0.3);
\draw[blue, rotate=144] (8.5:2) arc (91:269:0.3);\draw[blue, rotate=-144] (8.5:2) arc (91:269:0.3);
\penta{0,0}{dashed, gray};
\end{scope}
\begin{scope}[shift={(10,0)}]
\draw[dashed, gray, shift={(0,-2)}] (0,0) .. controls +(-0.6,1.2) and +(0,-1.5) .. +(-1.5,2.5);
\draw[dashed, gray, shift={(0,0.5)}] (0,0) .. controls +(-0.3,1.3) and +(0,0.7) .. +(-1.5,0);
\draw[dashed, gray, shift={(0,-2)}] (0,0) .. controls +(0.6,1.2) and +(0,-1.5) .. +(1.5,2.5);
\draw[dashed, gray, shift={(0,0.5)}] (0,0) .. controls +(0.3,1.3) and +(0,0.7) .. +(1.5,0);
\draw[thick, blue] (0.5,0) -- (2,0) (-0.5,0) -- (-2,0);
\draw[blue, rotate=-90] (9:2) arc (91:269:0.32);
\draw[blue, rotate=90] (35:0.5) arc (110:-110:0.32);
\cyl{0,0};
\end{scope}

\draw[thick] (-3,2.5) -- (-3,-3.5);
\node at (-10,-3.5) {Triangulations};
\node at (5,-3.5) {Laminations};
\end{tikzpicture}
\caption{Examples of triangulations and laminations}\label{fig:A2,Kron1}
\end{figure}

In finding deeper connection between the geometry of (triangulated) surfaces and cluster theory and in particular, tools that work regardless of existence of punctures, classical triangulation is somewhat inconvenient to work with.  Based on Fock and Goncharov's work \cite{FG}, Fomin and Thurston considered (Fock-Goncharov's reinterpretation of) \emph{laminations} in place of triangulations in the sequel article \cite{FT} of \cite{FST}.

A lamination is a maximal collection of pairwise non-intersecting (self-non-intersecting) curves (up to isotopy) such that each curve is of one of the following form:
\begin{itemize}[leftmargin=2cm]
\item[(L1)] a closed curve non-isotopic to a point;

\item[(L2)] a non-closed curve such that each of its ends 
\begin{itemize}
\item[either] (L2.a) terminates at an \emph{unmarked} point on the boundary,
\item[or] (L2.b) winds around a punctured indefinitely.
\end{itemize}
\end{itemize}
The three configurations on the right-hand side of Figure \ref{fig:A2,Kron1} shows the laminations corresponding to the three triangulations on the left.

Consider relaxing the condition (L2.b) to 
\begin{itemize}[leftmargin=2cm]
\item[(L2.b')] never terminates.
\end{itemize}
Then we call a maximal collection of pairwise non-intersecting curves satisfying (L0), (L1), (L2.a), (L2.b') a `\emph{refined lamination}'.
For instance, in the case of the Kronecker quiver, the two small curve around the marked points in Figure \ref{fig:A2,Kron1} along with the unique simple closed curve of the annulus form a lamination - this is shown in the first configuration of Figure \ref{fig:lam-eg}.  This lamination admits four distinct refinements that are shown on the right-hand side of Figure \ref{fig:kron2}.  In each of these refinements, there are two curves with one end that terminates on the boundary (satisfying (L2.a)) and the other end winds along the simple closed curve indefinitely (satisfying (L2.b')).  Note that every refined lamination of the Kronecker quiver belongs to one of these four, or one of the (ordinary) lamination; see Figure \ref{fig:kron2}.

\begin{figure}[!ht]
\centering
\begin{tikzpicture}[scale=0.6]
\draw[thick, black!60!green] (0,0) circle (1.25);
\cylNullStr{0,0}{orange};
\cyl{0,0}; 

\begin{scope}[shift={(5,0)}]
\draw[black!60!green] (0,0) circle (1.25);
\strInftyInner{0,0}{thick, magenta};
\strInftyOuter{0,0}{thick, blue}; 
\cylNullStr{0,0}{orange};\cyl{0,0}; 
\end{scope}

\begin{scope}[shift={(10,0)}]
\draw[black!60!green] (0,0) circle (1.25);
\strInftyOuter{0,0}{thick, blue}; 
\begin{scope}[xscale=-1]\strInftyInner{0,0}{thick, magenta};\end{scope}
\cylNullStr{0,0}{orange};\cyl{0,0}; 
\end{scope}

\begin{scope}[shift={(15,0)}]
\draw[black!60!green] (0,0) circle (1.25);
\strInftyInner{0,0}{thick, magenta}; 
\begin{scope}[xscale=-1]\strInftyOuter{0,0}{thick, blue};\end{scope}
\cylNullStr{0,0}{orange};\cyl{0,0}; 
\end{scope}

\begin{scope}[shift={(20,0)}]
\draw[black!60!green] (0,0) circle (1.25);
\begin{scope}[xscale=-1]\strInftyInner{0,0}{thick, magenta};\end{scope}
\begin{scope}[xscale=-1]\strInftyOuter{0,0}{thick, blue};\end{scope}
\cylNullStr{0,0}{orange};\cyl{0,0}; 
\end{scope}

\draw[thick] (2.5,2) -- (2.5,-2.5);
\node at (0,-2.5) {Lamination};
\node at (12.5,-2.5) {Refined lamination};
\end{tikzpicture}
\caption{Refinements of a lamination}\label{fig:lam-eg}
\end{figure}
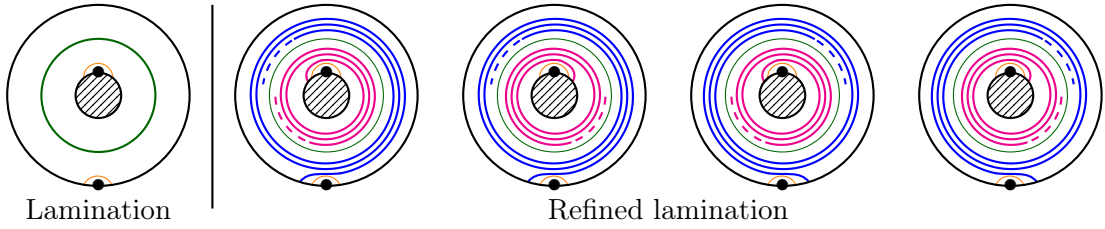

Roughly speaking, the main result of this article shows that `refined laminations', together with some extra data (namely, Laurent polynomials attached to simple closed curves), do classify torsion classes of gentle algebras.  However, nowhere in the actual proof rely on \emph{any} topological input; this is because all manipulation of curves in the topological model can only be described using the combinatorics of strings (and bands).  Therefore, we will not give any further details to the topological model from now on, and instead refer the reader to \cite{PPP19}.  We will, however, try to explain some of the intuitions of the proof using brief topological picture from time to time.  The precise classification result is the following.

\begin{theoremi}\label{thma}{\rm (Theorem \ref{classtors})}
Let $(Q,R)$ be a gentle quiver and $\Lambda$ be the associated gentle algebra defined over a field $k$.
Then there is a bijection
 \begin{align*}
\begin{Bmatrix}
 \mbox{$k$-parametrised maximal non-crossing set}\\ 
 \text{of infinite strings of ($Q, R$)} 
\end{Bmatrix} \leftrightarrow \begin{Bmatrix}
\text{torsion classes in the category of} \\
\text{fintie-dimensional $\Lambda$-modules}
\end{Bmatrix}.
 \end{align*}
Moreover, the inclusion relation between torsion classes is reflected by the positive crossings of strings.
\end{theoremi}

Let us gives a brief explanation to the words appearing in the statement. 
\begin{itemize}
\item {\it Infinite strings}:  This is the combinatorics that says that a curve satisfying (L1) and (L2) in the refined lamination setting.

\item {\it Non-crossing}:  This is the combinatorial language of curves in a set being self- and pairwise non-intersecting.

\item {\it $k$-parametrised}:  This is the property analogous to a band module corresponds to a simple closed curve equipped with a Laurent polynomial.  Precisely, a  non-crossing set of infinite strings is $k$-parametrised if every simple closed curve (i.e. periodic infinite string) in the set is attached with a (possibly empty) set of irreducible Laurent polynomial.  c.f. classification of torsion classes for Kronecker algebra \cite[Example 3.6]{DIRRT}.

\item {\it Positive crossing}:  Topologically, this is a certain choice of orienting an intersection by utilising the orientation of the surface and the (partial) triangulation defining the gentle algebra.
\end{itemize}

For the algebraically-inclined reader, let us remark on the meaning of positive crossing in the following.  One should think of an infinite string as a complex concentrated in homological degree $-1$ and $0$.  Suppose $U,V$ are two such complex corresponding to two infinite strings $\st{\gamma}, \st{\delta}$.  Then the existence of a \emph{positive crossing from $\st{\gamma}$ to $\st{\delta}$} says that $\Hom(V,U[1])\neq 0$.  For reader familiar with silting complexes, we note that this is precisely how the partial order on (two-term) silting complexes is defined.


Most of our proof is spent on understand the string combinatorics involved.  The key idea in our strategy is to replace torsion classes of module categories by the set of strings associated to the indecomposable module in there.  This translate most of the problem to a combinatorial one, namely, to show the set on the left-hand side of Theorem \ref{thma} is in correspondence with `combinatorial torsion classes'.  Another key ingredient to the proof is the {\it lattice property of the set of torsion classes} ordered by inclusion; see \cite{DIRRT} and subsection \ref{subsec:lattice}.

Finally, a well-versed reader may suspect that a $k$-parametrised maximal non-crossing set is just a combinatorial description of some `big (co)silting modules' of \cite{AHMV}.  We expect this is true - indeed, the case of an annulus without interior marked point is already shown in \cite{BL}.

This article is structured as follows.  In Section \ref{sec:prelim}, we recall various basic results concerning complete lattice and torsion classes.  In Section \ref{sec:basicstr}, we give definitions of various notions in string combinatorics that we need for our later exposition.  Section \ref{sec:torset} is devoted to showing how a certain combinatorial version of torsion classes over a gentle quiver $(Q,R)$ correspond to maximal non-crossing sets of strings, as well as the combinatorial meaning of inclusion of torsion classes in terms of the corresponding maximal non-crossing sets.  In Section \ref{sec:gentlealg}, we recall classification of indecomposable modules and canonical basis for morphisms between them.  The proof of Theorem \ref{thma} will occupy all of Section \ref{sec:torcl}.  In the last Section \ref{sec:BGA}, we explain how our correspondence can be used to obtain the classification of torsion classes for Brauer graph algebras.

\section*{Acknowledgement}
AC is supported by JSPS International Research Fellowship and JSPS Grant-in-Aid for Research Activity Start-up program 19K23401.  LD is supported by JSPS Grant-in-Aid for Young Scientists (B) 17K14160.  We thank Osamu Iyama, Toshiya Yurikusa, and Sota Asai for numerous fruitful discussions related to our work.


\section{Preliminaries}\label{sec:prelim}
\subsection{Lattice theory}\label{subsec:lattice}

Let us review some basic definition and properties about lattices.

\begin{definition}[Join, meet, (semi)lattice, completeness]
Let $L$ be a partially ordered set and $x,y\in L$.
The \defn{join} of $x$ and $y$, denoted by $x\vee y$, is an element $z\in L$ satisfying $z\geq x$, $z\geq y$, and is minimal with respect to this property.  If the join of $x,y$ exists, then it is the unique least upper bound of $x,y$.
Dually, the \defn{meet} of $x$ and $y$, denoted by $x\wedge y$, is the unique greatest lower bound of $x,y$.

We call $L$ a \defn{join-semilattice} if $x\vee y$ exists for all $x,y\in L$.  In which case, $\vee$ is an associative and commutative operation on $L$.  We define a \defn{meet-semilattice} analogously.

For an infinite subset $I\subseteq L$, there need not exists a unique least upper bound (resp. greatest lower bound) in $L$.  A join-semilattice (resp. meet-semilattice) $L$ is \defn{complete} if the unique least upper bound $\bigvee_{x\in I} x$ (resp. greatest lower bound $\bigwedge_{x\in I}x$) exists for all subset $I\subseteq L$.

$L$ is called a \defn{lattice} if it is both a join- and meet-semilattice with respect to the same partial order.  Likewise, a \defn{complete lattice} is a poset that is simultaneously compelete join-semilattice and complete meet-semilattice with respect to the same partial order.
\end{definition}

\begin{definition}[Morphism]
Suppose $L, L'$ are join-semilattices.
A map $f:L\to L'$ is a \defn{morphism of join-semilattices} if it is order-preserving and $f(x\vee y) = f(x)\vee f(y)$.
Likewise, $f:L\to L'$ is a \defn{morphism of complete join-semilattices} if $f\left(\bigvee_{x\in I} x\right) = \bigvee_{x\in I} f(x)$ for all $I\subset L$.

We define morphism of meet-semilattices and morphism of complete meet-semilattices similarly.  A \defn{morphism of lattices} is an order-preserving map that is also a morphism of join-semilattice and a morphism of meet-semilattice; similarly for \defn{morphism of complete lattices}.
\end{definition}
\begin{remark}\label{rem:lattice}
In practice, to check whether a map defines an isomorphism of complete lattices, it suffices to show that it is bijective and it is a morphism of complete join-semilattice (or meet-semilattice).
\end{remark}

It is not difficult to see that (as we show below for completeness) that if a morphism of join-semilattice maps preserves the strictness of finite chains, then it is injective.

\begin{lemma}\label{injmor}
Suppose $f:L\to L'$ is a morphism of join-semilattices.
If $f(x)\gneq f(y)$ holds for all $x\gneq y$ in $L$, then $f$ is injective.
\end{lemma}
\begin{proof}
If $f(x)=f(x')$, then we have
\begin{align*}
f(x\vee x')=f(x)\vee f(x) = f(x)\vee f(x) = f(x).
\end{align*}
Since $x\vee x'\geq x$ by the definition of join, the assumption of lemma implies that $x\vee x' = x$, and hence $x\geq x'$.
Arguing symmetrically yields $x'\geq x$, and so $x=x'$.
\end{proof}

\begin{definition}[Join-irreducible]
Let $L$ be a lattice.  An element $x\in L$ is \defn{join-irreducible} if there is no finite subset $I\subset L$ such that $x=\bigvee_{s\in I}s$ (or equivalently, $x=y\vee z$ implies $x=y$ or $x=z$).  Likewise, if $L$ is moreover complete, then we say that an element $x\in L$ is \defn{completely join-irreducible} when there is no subset $I\subseteq L$ such that $x = \bigvee_{s\in I} s$.
\end{definition}
While we will not use this characterization, one may find it helpful that if $L$ has ``enough arrows" in its Hasse quiver (recall that an arrow $x\to y$ exists if $x>y$ and $x\geq z\geq y\Rightarrow z=x$ or $z=y$) such as the complete lattice of torsion classes (see Proposition \ref{torclass1}), then an element is completely join-irreducible if it has precisely one outgoing arrow in the Hasse quiver.

There is also the dual notion of meet-irreducible and completely meet-irreducible that we will not use.  For us, the importance of join-irreducible is the following easy result.
\begin{lemma}\label{bricksurj}
Let $f:L\to L'$ be a morphism of complete join-semilattice such that any completely join-irreducible element of $L'$ is in the image of $f$.
If any element of $L'$ can be written as a join of completely join-irreducible element, then $f$ is surjective.
\end{lemma}
\begin{proof}
Take $x\in L'$, then by the condition we have some set $I$ of completely join-irreducible elements so that $x=\bigvee_{y\in I} y$.
But the assumption says that $f$ maps onto $I$, so the claim follows from $f$ being commutative with the join operation.
\end{proof}

\subsection{Torsion classes}\label{subsec:torclass}

Throughout, fix an abelian length category $\Ar$, i.e. an abelian category consisting only of finite (composition) length objects.  We summarize in below some results about the poset formed by torsion classes in $\Ar$ under the inclusion relation.  In the setting of $\Ar$ being the category of finite-dimensional modules over a finite-dimensional algebra, detailed proofs for (often much stronger version of) these statements can be found in \cite{DIRRT}; see also \cite{IRTT, GM}.  The proofs in the general setting of abelian length are usually in verbatim, so here we only include the most essential arguments for completeness.

\begin{definition}
A full subcategory $\Tr$ of an abelian length category $\Ar$ is called a \emph{torsion class} if it is closed under extensions and taking quotients.  In other words, for any short exact sequence $0\to L\to M\to N\to 0$ in $\Ar$, the terms $L,N\in \Tr$ implies so is $M$, and also $M\in \Tr$ implies so is $N$.  Denote by $\tors\Ar$ the collection of torsion classes in $\Ar$, which is also a poset where the partial order is given by inclusion of subcategories.
\end{definition}

For any class $\Mr$ of objects, we denote by
\begin{itemize}
\item $\Tors(\Mr)$ the smallest torsion class containing $\Mr$;
\item $\Fac(\Mr)$ the class of objects $X$ such that there is an epimorphism $M\twoheadrightarrow X$ with $M\in \Mr$;
\item $\Filt(\Mr)$ the class of objects $X$ that admit finite filtrations $(X_i)_{1\leq i\leq n_X}$ satisfying $X_i/X_{i+1}\cong M_i \in \Mr$ for all $1 \leq i<n_X$.
\item $\brick(\Mr)$ the class of \defn{bricks} in $\Mr$, i.e. objects  whose endomorphisms form a division ring.
\end{itemize}
We avoid overusing brackets, we will remove the curly brackets when $\Mr=\{X\}$ for some $X\in \Acal$ when applying any of the operations above.

\begin{lemma}\label{torclass0}
The following hold for an abelian length category $\Acal$.
\begin{enumerate}[\rm (i)]
\item {\rm \cite[Lemma 3.10]{DIRRT}}A torsion class $\Tr$ in $\Acal$ is uniquely determined by $\brick(\Tr)$, namely, $\Tr=\Filt(\brick(\Tr))$.

\item For any class $\Mr$ of objects in $\Acal$, we have $\Tors(\Mr) = \Filt(\Fac(\Mr))$.
\end{enumerate}
\end{lemma}
Due to heavy usage, from now on, we will omit the bracket between $\Filt$ and $\Fac$.
\begin{proof}
(i) $\Tr\supseteq \Filt(\brick(\Tr))$ is clear as $\Tr$ is extension closed and contains $\brick(\Tr)$.  For the inclusion, we prove by induction on length of object of $X\in \Tr$ (which is possible as $\Acal$ consists only objects of finite length).

Pick $X\in \Tr$ of minimal length that is not in $\Filt(\brick(\Tr))$.
Then $X$ cannot be a brick of $\Tr$, and so there is a morphism $f\in \End_\Acal(X)$ that is not invertible.  This yields a short exact sequence
\[ 0\to \im(f) \to X\to \coker(f)\to 0, \]
with $\im(f), \coker(f)$ both being of smaller length than $X$.
Since $f$ factors through $\im(f)$, both $\im(f)$ are quotients of $X$ and so they are both in $\Tr$.  By minimality of $X$, these two objects have to be in $\Filt(\brick(\Tr))$, but this would mean that $X\in \Filt(\brick(\Tr))$, too; a contradiction.

(ii) The definition of torsion class says that any torsion class containing $\Mr$ must contains all the objects in $\Filt\Fac(\Mr)$.  Conversely, it is easy to see that $\Filt\Fac(\Mr)$ is a torsion class and so it must contain $\Tors(\Mr)$.
\end{proof}

\begin{proposition}\label{torclass1}
Let $\Acal$ be an abelian length category and $\tors\Acal$ the poset of torsion classes in $\Acal$ ordered under inclusion.
Then the following hold.
\begin{enumerate}[\rm(i)]
\item $\tors\mathcal{A}$ is a complete lattice, with maximum $\Acal$ and minimum $\{0\}$, such that for a family $\{\Tr_i\}_{i\in I}$ of torsion classes, we have
\[\bigwedge_{i \in I} \Tr_i = \bigcap_{i \in I} \Tr_i,
\quad \text{ and }\quad 
\bigvee_{i \in I} \Tr_i = \bigwedge_{\Tr'\supseteq \Tr_i\;\forall i\in I} \Tr'.
\]
In particular, for any family $\{\Mr_i\}_{i\in I}$ of sets of modules, we have \begin{align}\label{veecup}
\bigvee_{i\in I} \Tors(\Mr_i) =  \Tors\left(\bigcup_{i\in I}\Mr_i\right).
\end{align}

\item {\rm \cite[Theorem 3.4(a)]{DIRRT}} $\Tr = \bigvee_{S\in\brick(\Tr)} \Tors(S)$ for any $\Tr\in\tors\Acal$.

\item {\rm \cite[Theorem 3.4(a)]{DIRRT}} A torsion class is completely join-irreducible if and only if it is the smallest torsion class $\Tors(S)$ containing a brick $S$.
\end{enumerate}
\end{proposition}
\begin{proof}
(i) Clear.

(ii) By \eqref{veecup}, we have $\bigvee_{S\in \brick(\Tr)} \Tors(S) = \Tors(\brick(\Tr))$.  So it follows from Lemma \ref{torclass0} (i) that this is equal to $\Filt\Fac(\brick(\Tr))=\Tr$.

(iii)  If $\Tr\in\tors\Acal$ is completely join-irreducible, then it follows from (ii) that it is of the form $\Tors(S)$ for some brick $S$.

Conversely, if $\Tors(S)=\bigvee_{i\in I}\Tr_i$ for some family $\{\Tr_i\}_{i\in I}$ of torsion classes in $\Acal$, then Lemma \ref{torclass0} (i) implies that there is an inclusion $M\inj S$ for some $M\in \brick(\Tr_i)$ for some $i\in I$.
On the other hand, as we have $\brick(\Tr_i)\subseteq \Tr_i\subseteq\Tors(S)$ for each $i$, by Lemma \ref{torclass0} (ii), there must exist a non-zero morphism $S\to M$.  

Composing the two morphisms yields a non-zero endomorphism of $S$.  Hence, $S$ being a brick implies that $S\cong M \in \brick(\Tr_i)$.  In particular, the minimality of $\Tors(S)$ implies that $\Tr_i=\Tors(S)$.
\end{proof}


\section{Basics of string combinatorics}\label{sec:basicstr}

We give the necessary combinatorial setup for the main result.  We compose arrows from left to right, i.e. $pq$ a a path for arrows $p, q$ with $t(p)=s(q)$.

\begin{definition}[Gentle quiver]\label{def:gentle}
A \defn{gentle quiver} is a tuple $(Q,R)$ consisting of a finite quiver $Q=(Q_0, Q_1, s, t)$ and a set $R\subset Q_2$ of (generating) \defn{relations} consisting of length-two paths such that the following conditions are satisfied.
\begin{enumerate}[(i)]
\item Any $i\in Q_0$ has at most two incoming and two outgoing arrows.
\item For any $q\in Q_1$, there is at most one $r\in Q_1$ such that $t(q)=s(r)$ and $qr\notin R$.
\item For any $q\in Q_1$, there is at most one $r\in Q_1$ such that $t(q)=s(r)$ and $qr\in R$.
\item For any $q\in Q_1$, there is at most one $r\in Q_1$ such that $t(r)=s(q)$ and $rq\notin R$.
\item For any $q\in Q_1$, there is at most one $r\in Q_1$ such that $t(r)=s(q)$ and $rq\in R$.
\end{enumerate}
\end{definition}

Our gentle quivers are the locally gentle bound quivers in \cite{PPP19}; in particular, the associated (completed) bounded path algebra is \emph{not} necessarily finite dimensional.

\begin{example}\label{eg:first}
We give three examples.  These will be the main running examples of throughout.
\begin{enumerate}[(1)]
\item $\vec{A}_n$-quiver: $Q=\vec{A}_n = (1\to 2 \to 3 \to \cdots \to n)$ and $R=\emptyset$.
\item Kronecker quiver: $Q=\xymatrix@1{1\ar@<2pt>[r]\ar@<-2pt>[r] & 2}$ and $R=\emptyset$.
\item Markov quiver: 
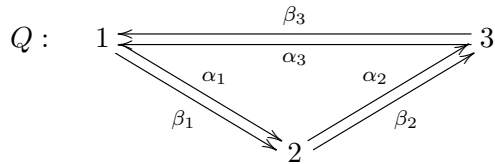
\[
Q: \quad \xymatrix@C=60pt@R=30pt{ 1 \ar@<2pt>[rd]^{\alpha_1}\ar@<-2pt>[rd]_{\beta_1} & & 3 \ar@<2pt>[ll]^{\alpha_3}\ar@<-2pt>[ll]_{\beta_3} \\ & 2 \ar@<2pt>[ru]^{\alpha_2}\ar@<-2pt>[ru]_{\beta_2} & }
\]
and $R=\{\alpha_i\beta_{i+1}, \beta_i\alpha_{i+1}\mid i\in\mathbb{Z}/3\mathbb{Z}\}$.
\end{enumerate}
\end{example}

\begin{definition}[Blossoming]
A \defn{blossoming} of a gentle quiver $(Q,R)$ is another gentle quiver $(Q', R')$ satisfying the following conditions.
\begin{enumerate}[(i)]
\item $Q_0 \subseteq Q'_0$, $Q_1 \subseteq Q'_1$, and $R'\cap Q_2 = R$.

\item For any $q \in Q'_1 \setminus Q_1$, exactly one of $s(q)$ and $t(q)$ is in $Q'_0$.

\item For any $i \in Q_0$, there are exactly two arrows $q \in Q'_1$ such that $s(q) = i$.

\item For any $i \in Q_0$, there are exactly two arrows $q \in Q'_1$ such that $t(q) = i$.

\item For any pair of arrows $q$ and $r$ satisfying $t(q) = s(r) \in Q'_0 \setminus Q_0$, $qr\in R'$.
\end{enumerate}
The \defn{classical blossoming} $(Q^\star, R^\star)$ is a blossoming where $\#Q_1^\star\setminus Q_1 = \#Q_0^\star\setminus Q_0$.
\end{definition}
\begin{remark}\label{rem:blossom}
(1) Two different blossom have the same set of arrows.  Everything we need from blossom quivers are the new arrows only, i.e. \emph{everything we do is independent of the choice of blossoming}.  

(2) Blossomed gentle quiver can be identify with a (suitably defined) \emph{dissection of marked surface}.  This topological interpretation plays no role in any of our arguments and so we will not give details about their construction and refer the reader to the `green-dissection' in \cite[Section 3]{PPP19}.  Roughly speaking, vertices of the blossomed quiver are arcs in the marked surface where those not in the original gentle quiver are boundary maps.  The arrows of the blossomed quiver corresponding to the angles between consecutive arcs along the orientation of the surface.  The polygons in the dissection correspond to maximal paths with consecutive arrows belonging to $R$.  Different choices of blossoming result in different topological models, but the difference is only in the number of marked points on the boundary.  See Example \ref{eg:blossom}.

(3) The conditions of blossoming imply that if $i \in Q'_0 \setminus Q_0$, at most one $q \in Q'_1$ satisfy $s(q) = i$ and at most one $q \in Q'_1$ satisfy $t(q) = i$.

(4) A blossoming is called fringing in \cite{BDMTY}.
\end{remark}

\begin{example}\label{eg:blossom}
\begin{enumerate}[(1)]
\item For the $\vec{A}_2$-quiver $(Q=\vec{\mathbb{A}}_2, R=0)$, we give two possible blossomings.
\begin{enumerate}[(a)]
\item The classical blossoming $Q^\star$ is given by
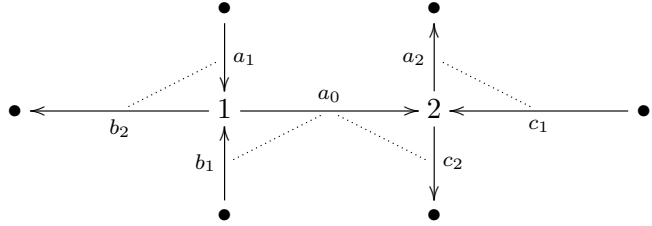
\[
\xymatrix@R=10pt@C=30pt{ & & \bullet\ar[dd]^{a_1} & & \bullet & &\\
 & & \ar@{.}[ld] & &\ar@{.}[rd] & & \\
\bullet & & 1\ar[rr]^{a_0}\ar[ll]^{b_2} & \ar@{.}[rd] & 2\ar[dd]^{c_2}\ar[uu]^{a_2} & & \bullet\ar[ll]^{c_1}  \\
 &&\ar@{.}[ru] && && \\
  & & \bullet \ar[uu]^{b_1}  & & \bullet & & }
\]
where all $\bullet$-points are distinct vertices, 
and $R^\star$ is given by all the paths $p_iq_j$ for $p\neq q\in\{a, b, c\}$; i.e. the composition given by the dotted lines.  The topological model (the green dissection model in \cite{PPP19}) is shown in the left-hand side of Figure \ref{fig:A2-top}.

\item $Q'$ is given by gluing the two new vertices of $Q^\star$ in the bottom row, i.e.
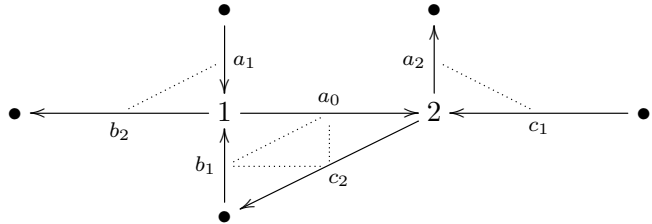
\[
\xymatrix@R=10pt@C=30pt{ & & \bullet\ar[dd]^{a_1} & & \bullet & &\\
 & & \ar@{.}[ld] & &\ar@{.}[rd] & & \\
\bullet & & 1\ar[rr]^{a_0}\ar[ll]^{b_2} & \ar@{.}[d] & 
2\ar[ddll]^{c_2}\ar[uu]^{a_2} & & \bullet\ar[ll]^{c_1}  \\
 &&\ar@{.}[ru]\ar@{.}[r] && && \\
  & & \bullet \ar[uu]^{b_1}  & & & & }
\]
$R'=R^\star\cup \{2\to\circ\to 1\}$.
One can find this blossoming in the context of triangulation of a pentagon, as shown in the right-hand side of Figure \ref{fig:A2-top}.

\begin{figure}[!htbp]\centering
\begin{tikzpicture}[scale=0.7]
\begin{scope}[shift={(0,0)}, rotate=90]
\penta{0,0}{blue};
\draw (0.2,0.6) node {1} (0.2,-0.6) node {2};
\draw[->] (0.2,0.4) -- (0.2,-0.4) node [midway, below] {\footnotesize $a_0$};
\draw[->, inner sep=2pt] (1.8,0.6) -- (0.6,0.6) node [midway, left] {\footnotesize $a_1$};
\draw[->, inner sep=2pt] (0.6,-0.6) -- (1.8,-0.6) node [midway, right] {\footnotesize $a_2$};
\draw[->, inner sep=2pt] (0.2,0.8) -- (-0.2,1.8) node [midway, below] {\footnotesize $b_2$};
\draw[->, inner sep=2pt] (-1.8,0.2) -- (-0.2,0.6) node [midway, left] {\footnotesize $b_1$};
\draw[->, inner sep=2pt] (-0.2,-0.6) -- (-1.8,-0.2) node [midway, right] {\footnotesize $c_2$};
\draw[->, inner sep=2pt] (-0.2,-1.8) -- (0.2,-0.8) node [midway, below] {\footnotesize $c_1$};
\end{scope}
\begin{scope}[shift={(-5,0)}, rotate=90]
\draw[thick] (0,0) circle (2);
\draw[shift={(0,0)}, blue] (0:2) -- (120:2) (0:2) -- (-120:2);
\draw (0,0) ++(0:2) node{$\bullet$};
\draw (0,0) ++(60:2) node{$\bullet$};
\draw (0,0) ++(120:2) node{$\bullet$};
\draw (0,0) ++(180:2) node{$\bullet$};
\draw (0,0) ++(-120:2) node{$\bullet$};
\draw (0,0) ++(-60:2) node{$\bullet$};
\draw (0.4,0.8) node {1} (0.4,-0.8) node {2};
\draw[->] (0.4,0.6) -- (0.4,-0.6) node [midway, below] {\footnotesize $a_0$};
\draw[->, inner sep=2pt] (1.6,0.8) -- (0.8,0.8) node [midway, left] {\footnotesize $a_1$};
\draw[->, inner sep=2pt] (0.8,-0.8) -- (1.6,-0.8) node [midway, right] {\footnotesize $a_2$};
\draw[->, inner sep=2pt] (0.4,1.2) -- (0.4,1.8) node [midway, below] {\footnotesize $b_2$};
\draw[->, inner sep=2pt] (-1.6,0.8) -- (0,0.8) node [midway, right] {\footnotesize $b_1$};
\draw[->, inner sep=2pt] (0,-0.8) -- (-1.6,-0.8) node [midway, left] {\footnotesize $c_2$};
\draw[->, inner sep=2pt] (0.4,-1.8) -- (0.4,-1.2) node [midway, below] {\footnotesize $c_1$};
\end{scope}
\end{tikzpicture}
\caption{Topological models for different blossomings of the $\vec{A}_2$-quiver}\label{fig:A2-top}
\end{figure}
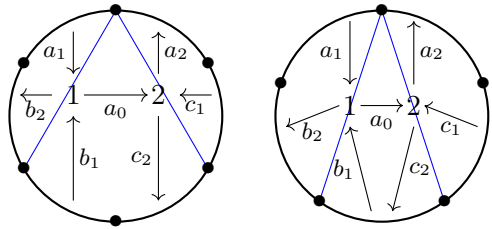
\end{enumerate}

\item For the Kronecker quiver $K_2$, we can choose blossoming $K_2'$ as follows.
\begin{align*}
K_2&=(\xymatrix{1\ar@<2pt>[r]^{a_1}\ar@<-2pt>[r]_{b_1} & 2}\;\;,\;\;\emptyset) \\ K_2'&=\Big(\vcenter{\vbox{\xymatrix@R=8pt{  & \circ \ar[dr]^{a_0}& \\ 2\ar[ur]^{a_2}\ar[dr]_{b_2} & & 1\ar@<2pt>[ll]_{a_1}\ar@<-2pt>[ll]^{b_1} \\ & \bullet\ar[ur]_{b_0} &  }}} \;\;,\;\; \{a_ib_{i+1}, b_ia_{i+1}, a_2a_0, b_2b_0\mid i=0,1\} \Big)
\end{align*}
This choice of blossoming is the one used in triangulated surface, see Figure \ref{fig:kron1}.

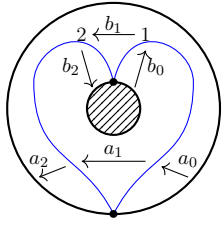
\begin{figure}[!ht]
\centering

\begin{tikzpicture}[scale=0.7,every node/.style={transform shape}]
\draw[blue, shift={(0,-2)}] (0,0) .. controls +(0.6,1.2) and +(0,-1.5) .. +(1.5,2.5);
\draw[blue, shift={(0,0.5)}] (0,0) .. controls +(0.3,1.3) and +(0,0.7) .. +(1.5,0);
\draw[blue, shift={(0,-2)}] (0,0) .. controls +(-0.6,1.2) and +(0,-1.5) .. +(-1.5,2.5);
\draw[blue, shift={(0,0.5)}] (0,0) .. controls +(-0.3,1.3) and +(0,0.7) .. +(-1.5,0);
\cyl{0,0};
\draw[->] (0.4,1.4) -- (-0.4,1.4); \draw[->] (0.6,-1) -- (-0.6,-1);
\node at (0.6,1.4) {1};
\node at (-0.6,1.4) {2};
\draw[->] (0.4,0.4) -- (0.6,1.1); \draw[->] (-0.6,1.1) -- (-0.4, 0.4);
\node at (0.8,0.8) {$b_0$};
\node at (0,1.6) {$b_1$};
\node at (-0.8,0.8) {$b_2$};
\draw[->] (1.4,-1.3) -- (0.9,-1.1); \draw[->] (-0.9,-1.1) -- (-1.4,-1.3);
\node at (1.4,-1) {$a_0$};
\node at (0,-0.8) {$a_1$};
\node at (-1.4,-1) {$a_2$};
\end{tikzpicture}
\caption{A blossoming of the Kronecker quiver}\label{fig:kron1}
\end{figure}

\item For the Markov quiver (or any 2-regular gentle quiver in general), it only admits one blossoming which is itself.  The topological model for the Markov quiver is the triangulation formed by 3 arcs enclosing two triangles in a once-punctured torus.
\end{enumerate}
\end{example}

Until the end of this section, we fix a gentle quiver $(Q,R)$ and a blossoming $(Q',R')$.
For convenience, for arrows $q,r$ in $Q$ or $Q'$, we write $qr=0$ if $qr\in R$ or $R'$; $qr\neq 0$, otherwise.

\begin{definition}[Letters, walks, and strings]\label{def:walks}
For an arrow $q\in Q_1$, the (formal) inverse of $q$ is denoted by $q^-$.
A \defn{letter} of $Q$ (respectively, $Q'$) is an arrow of $Q$ or the (formal) inverse of an arrow of $Q$ (respectively, $Q'$).
We denote the set of letters of $Q$ and $Q'$ by $Q_1^{\pm}$ and $Q_1'^{\pm}$ respectively.

For $q \in Q'_1$, we denote $s(q^{-}) = t(q)$ and $t(q^{-}) = s(q)$.
The vanishing relations on the letters of $Q$ and $Q'$ are defined as follows.
\begin{enumerate}[(i)]
\item For $q, r\in Q_1'$, $r^{-}q^{-}=0$ if and only if $qr=0$.
\item If $t(q) = t(r)$ and $q \neq r$, then $qr^{-}\neq 0$.
\item If $s(q) = s(r)$ and $q \neq r$, then $q^{-}r\neq 0$.
\end{enumerate}

A \defn{stationary walk} in $Q'$ is a symbol $\un{q}{r}$ where $q, r \in Q'_1$ are arrows such that $qr\neq 0$.

A \defn{word} is a sequence of letters indexed by a possibly infinite interval of $\Z$, where letters in it are read from left to right in the increasing order of $\Z$.
A \defn{walk} in $Q'$ is a stationary walk or a (non-empty) word $\gamma = \cdots \alpha_i \alpha_{i+1} \cdots$ constituted of letters of $Q'$ such that $\alpha_i \alpha_{i+1}\neq 0$.

For a walk $\gamma$, $\gamma^{-}$ is defined in the obvious way when $\gamma$ is non-stationary, and is defined uniquely by $\gamma^{-} = \un{q'}{r'}$ when $\gamma = \un{q}{r}$, where $q'r'\notin R'^\pm$, $q' \neq q$ and $t(q') = t(q)$.

For a walk $\gamma$ in $Q'$, we denote by $\st \gamma$ the pair $\{\gamma, \gamma^-\}$, and we call this the associated \defn{string} (in $Q'$).
Obviously, $\st \gamma = \st {\gamma^-}$.
Note that strings induced by stationary walk correspond bijectively to the vertices of $Q$.
\end{definition}

On the marked surface, the distinction between a walk and a string is the same the distinction between the image of a curve (an unoriented geometric object) and the curve itself (where the orientation can be somewhat important).

The following property is probably well-known to algebraists experienced with strings (of bounded length).
\begin{lemma} \label{noselfinverse}
 No walk $\gamma$ satisfies $\gamma = \gamma^-$.
\end{lemma}
\begin{proof}
 If $\gamma$ is stationary, $\gamma = \gamma^-$ is impossible by definition. Otherwise, let us write $\gamma = \cdots a_i a_{i+1} \cdots$ where indices runs over an interval $I$ of $\Z$. If $\gamma = \gamma^-$, it implies that there exists $n \in \Z$ such that $a_i = a_{n-i}^-$ for any $i \in I$. If $n$ is an even number, we get $a_{n/2} = a_{n/2}^-$, which is absurd. If $n$ is odd, we get $a_{(n+1)/2} = a_{(n-1)/2}^-$, which contradicts the definition of a walk.
\end{proof}

As a consequence of Remark \ref{rem:blossom} (2) and the requirement on the consecutive letters in a walk (Definition \ref{def:walks}), we have the following.
\begin{proposition}\label{blo-indep(walk)}
For a gentle quiver $(Q,R)$, any choice of blossoming $(Q', R')$ gives rise to the same set of confined strings, and the same set of strings.
\end{proposition}

Note that for each vertex $v\in Q$, there are precisely two stationary walks as there are precisely two outgoing and two incoming arrows at $v$ in the blossoming $Q'$.  Moreover, stationary strings naturally correspond to vertices of $Q$ (not $Q'$!).  Hence, a stationary walk should be think of as one-half of the trivial path - an \emph{oriented trivial path}.

\begin{example}
The stationary walks in the $\vec{A}_2$-quiver (see Example \ref{eg:blossom}) are $\un{a_1}{a_0}$, $\un{b_1}{b_2}$, $\un{a_0}{a_2}$, and $\un{c_1}{c_2}$.  Note that $\st{\un{a_1}{a_0}}=\{\un{a_1}{a_0}, \un{b_1}{b_2}\}$ corresponds to vertex $1$ of $\vec{A}_2$, and $\st{\un{a_0}{a_2}}$ corresponds the vertex $2$.
\end{example}

In the topological picture, we can canonically identify stationary walks with the \emph{interior} half-edges of the dissection associated to $(Q,R)$.  A walk on its own, and likewise for its associated string, should be thought as a `0-dimensional curve' (a.k.a. a point) lying inside the corresponding arc.   Whenever concatenation of walks are involved (such as in Definition \ref{def:ext-fac}), it will be more suitable to think of a stationary walk as a turn around a marked point across an arc.

The topological interpretation of a non-stationary walk is an oriented curve on the marked surface. Note that, as opposed to most literature on topological model of gentle algebras, we allow curves whose defining domain is a ray or the real line, instead of just interval or the circle only.  In the following, we distinguish walks that involve the extra arrows in the blossomed quiver.

\begin{definition}[Confined, infinite, and periodic walks]
Let $\gamma$ be a walk.
\begin{itemize}
\item $\gamma$ is \defn{left-infinite} if 
  \begin{itemize}
  \item it is \defn{left-unbounded}, i.e. the indexing set of $\gamma$ has no finite lower bound;
  \item or it is left-bounded $\gamma = \alpha_i\alpha_{i+1}\alpha_{i+2}\cdots $ with $s(\alpha_i)\in Q'_0 \setminus Q_0$.
  \end{itemize}
Similarly, a walk is \defn{right-infinite} if it is either right-unbounded, or right-bounded with target of the last arrow in $Q_0'\setminus Q_0$.

\item $\gamma$ is \defn{infinite} if it is both left-infinite and right-infinite.

\item $\gamma$ is \defn{periodic} if it is unbounded in both ends, and there is some $r\in \Z$ such that the $i$-th letter in $\gamma$ is equal to the $(i\pm r)$-th letters in $\gamma$ for all $i\in \Z$.

\item $\gamma$ is \defn{left-confined} (respectively, \defn{right-confined}) if it is not left-infinite (respectively, right-inifinite).

\item $\gamma$ is \defn{confined} if it is left-confined and right-confined; in particular, it consists of only finitely many letters of $Q$.
\end{itemize}
We also say that a string is confined (resp. infinite, resp. periodic) if so is its underlying walk.
\end{definition}
In the case of walks with bounded length (and the associated gentle algebra is of finite dimension), infinite strings are the long strings of \cite{BDMTY}.  The term `infinite' here has nothing to do with the length of a string.  This is also one of the reasons why we use the term `confined' as opposed to `finite'; after all, the complement of the set of infinite strings is \emph{not} the set of confined string.  The distinction between these two notion is better explained from the topological and algebraic point of view as follows.

Topologically, confined strings correspond to curve which starts and ends at an internal arc.  A curve corresponding to an infinite string is one that satisfies (L1) or (L2') in the Introduction; c.f. Figure \ref{fig:A2-all-walks} for the case of $\vec{A}_2$-quiver.  

For algebraist, it would be helpful to think of confined strings as  indecomposable modules, and infinite strings as projective presentation of indecomposable modules - here we think of any paths (or its inverse) of unbounded length, or consisting of arrows in the blossoming, as zero maps.  For example, Figure \ref{fig:A2-P1} shows the modules and complexes arising from the projective module $P_1$ associated to the vertex $1\in \vec{A}_2$, and what their string interpretations are.

\begin{figure}[!htbp]
\centering
\begin{tikzpicture}[scale=0.4]
\begin{scope}[rotate=90, shift={(0,2)}]
\node at (9,0) {$\mathrm{mod}\, k\vec{A}_2$}; 
\node at (5,0) {$P_1 \leftrightarrow \st{a_0}$};
\penta{0,0}{gray};
\draw[thick, red] (0.2,0.6) -- (0.2,-0.6);
\end{scope}

\node at (11,9) {`$\mathsf{K}^{\mathrm{b}}(\mathrm{proj}\, k\vec{A}_2)/[2]$'};
\begin{scope}[shift={(6,0)}, rotate=90]
\draw[thick, blue] (-35:2) -- (120:2); 
\node at (6,0) {$P_1$ `=' $\left[\phantom{\begin{array}{c}cccccccccc\\c\\c\end{array}} \right]$};
\node at (3.3,1) {$\leftrightarrow \st{b_2^-a_0a_2}$};
\node (p1at0) at (6,-3.5) {$P_1$};\node (01) at (7,0.5) {$0$};\node (02) at (5,0.5) {$0$};
\draw[->] (01) -- (p1at0) node [midway,above] {\scriptsize `$b_2$'};
\draw[->] (02) -- (p1at0) node [midway,below] {\scriptsize `$a_0a_2$'};
\penta{0,0}{gray};
\end{scope}

\begin{scope}[shift={(16.5,0)}, rotate=90]
\draw[thick, blue] (35:2) -- (170:2); 
\node at (6,0) {$P_1[1]$ `=' $\left[\phantom{\begin{array}{c}ccccccccc\\c\\c\end{array}} \right]$};
\node at (3.3,0.7) {$\leftrightarrow \st{a_1^-c_1}$};
\node (p1at0) at (6,0) {$P_1$};\node (01) at (7,-4) {$0$};\node (02) at (5,-4) {$0$};
\draw[->] (p1at0) -- (01)  node [midway,above] {\scriptsize `$a_1$'};
\draw[->] (p1at0) -- (02) node [midway,below] {\scriptsize `$c_1$'};
\penta{0,0}{gray};
\end{scope}

\draw (1,9.5) -- (1,-2.5);
\draw (11,7.5) -- (11,-2.5);
\end{tikzpicture}
\caption{Algebraic and topological interpretations of confined and infinite strings}\label{fig:A2-P1}
\end{figure}
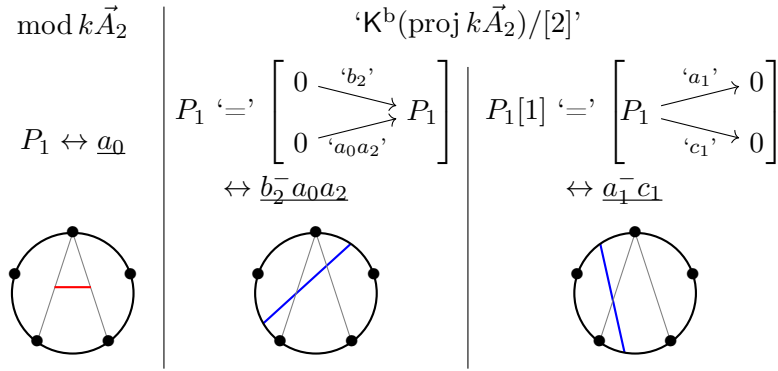

Periodic (infinite) strings can be identified with closed curves that are not isotopic to a point.  For example, the Kronecker quiver $K_2$ has an infinite string given by infinitely repeating the word $a_1^{-}b_1$.  This corresponds to the unique closed curve of the annulus - as shown in the first configuration of Figure \ref{fig:lam-eg}.

\begin{example}\label{eg:A2 walks strings}
For the $\vec{A}_2$-quiver, there are only 3 confined strings and 8 infinite strings.  We display all of these in Figure \ref{fig:A2-all-walks}.  The first configuration in Figure \ref{fig:A2-all-walks} on the left shows that 3 confined strings, and the remaining three configurations on the right shows all the infinite strings.
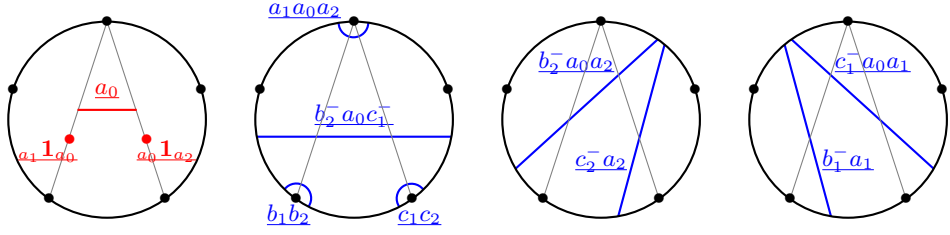
\begin{figure}[!htbp]\centering
\begin{tikzpicture}[scale=0.65, every node/.style={font=\footnotesize }]
\begin{scope}[shift={(0,0)}, rotate=90]
\draw[thick, blue] (100:2) -- (260:2) node [midway, above] {$\st{b_2^-a_0c_1^-}$}; 
\draw[thick, blue] (9:2) arc (91:269:0.3); \node[blue] at (2.1,1) {$\st{a_1a_0a_2}$};
\draw[thick, blue, rotate=144] (8.5:2) arc (91:269:0.3);\node[blue] at (148:2.5) {$\st{b_1b_2}$};
\draw[thick, blue, rotate=-144] (8.5:2) arc (91:269:0.3); \node[blue] at (-148:2.5) {$\st{c_1c_2}$};
\penta{0,0}{gray};
\end{scope}

\begin{scope}[shift={(5,0)}, rotate=90]
\draw[thick, blue] (-40:2) -- (-170:2); \node[blue] at (-1,0) {$\st{c_2^-a_2}$};
\draw[thick, blue] (-35:2) -- (120:2); \node[blue] at (1,0.5) {$\st{b_2^-a_0a_2}$};
\penta{0,0}{gray};
\end{scope}

\begin{scope}[shift={(10,0)}, rotate=90, yscale=-1]
\draw[thick, blue] (-40:2) -- (-170:2) (-35:2) -- (120:2);
\node[blue] at (-1,0) {$\st{b_1^-a_1}$};\node[blue] at (1,0.5) {$\st{c_1^-a_0a_1}$};
\penta{0,0}{gray};
\end{scope}

\begin{scope}[shift={(-5,0)}, rotate=90]
\penta{0,0}{gray};
\draw[thick, red] (0.2,0.6) -- (0.2,-0.6) node [midway, above] {\footnotesize $\st{a_0}$} ;
\draw[red] node at (-0.4,0.76) {$\bullet$} node at (-0.4,-0.8) {$\bullet$};
\node[red] at (-0.8,1.2) {$\st{\un{a_1}{a_0}}$} ;
\node[red] at (-0.8,-1.2) {$\st{\un{a_0}{a_2}}$} ;
\end{scope}
\end{tikzpicture}
\caption{All confined strings and infinite strings on the $\vec{A}_2$-quiver}\label{fig:A2-all-walks}
\end{figure}
\end{example}

Since a vertex of $Q$ correspond to a stationary string, which consists of \emph{two} walks, it will be helpful (in fact, essential) to enhance the source and target functions of a gentle quiver by keeping in mind the orientation as follows.

\begin{definition}[Source and targets for walks]
Suppose $\gamma$ is a walk.
\begin{itemize}
\item If $\gamma$ is left-confined non-stationary, we define $\us(\gamma)$ to be the unique stationary walk $\un{q}{r}$ such that $q \gamma$ is a walk.

\item If $\gamma$ is right-confined non-stationary, we define $\ut(\gamma)$ to be the unique stationary walk $\un{q}{r}$ such that $\gamma r$ is a walk.

\item For a stationary walk $\un{q}{r}$, we define $\us(\un{q}{r}) = \ut(\un{q}{r}) = \un{q}{r}$.
\end{itemize}
\end{definition}

Note that we have immediately $\us(\gamma^{-}) = \ut(\gamma)^{-}$ and $\ut(\gamma^{-}) = \us(\gamma)^{-}$.

\begin{definition}[Concatenation]
Consider two walks $\gamma$ and $\delta$.
If $\gamma$ is right-confined, $\delta$ is left-confined, and $\ut(\gamma) = \us(\delta)$, then the \defn{concatenation} of $\gamma$ and $\delta$ as words gives rise to a well-defined walk, denoted by $\gamma \delta$.
\end{definition}

\begin{example}\label{eg:concat}
If $\gamma$, $\gamma\gamma$ are both confined walks, then we have 
\begin{itemize}
\item a confined walk $\gamma^m$ given by concatenating $\gamma$ with itself $m$ times;
\item a right-infinite walk $\gamma^\infty := \gamma\gamma\gamma\cdots$;
\item a left-infinite walk ${}^\infty\gamma:=\cdots \gamma\gamma\gamma$;
\item an infinite walk ${}^\infty \gamma^\infty:= ({}^\infty\gamma)(\gamma^\infty) = \cdots \gamma\gamma \cdots$.
\end{itemize}
For convenience, let us denote by ${}^{-\infty}\gamma$ the string ${}^\infty(\gamma^{-1}) = (\gamma^\infty)^{-1}$, and $\gamma^{-\infty}:=(\gamma^{-1})^\infty = ({}^\infty \gamma)^{-1}$.

For a concrete example, in the Kronecker quiver $K_2$ of Example \ref{eg:blossom} (2), there is an infinite string $\st\gamma:=\st{a_2^-(a_1^-b_1)^\infty}$.  Unlike the case of the periodic string $\st\delta:=\st{{}^\infty a_1^-b_1^\infty}$, we display $\st\gamma$ topologically represented by a curve that keeps winding along $\st\delta$, as shown in blue (the `outer curve') in the second and third configurations of Figure \ref{fig:lam-eg}.
\end{example}

Example \ref{eg:concat} (1) shows how one can extend certain left/right-confined walks to left/right-infinite ones; this can be done more generally as follows.

\begin{definition}[Infinite extensions of confined walks]
Suppose $\gamma$ is a right-confined walk.
Then there is
\begin{itemize}
\item a unique right-infinite walk $\extr\gamma$ consisting only of arrows such that $\gamma \extr\gamma$ is a walk, and also 
\item a unique right-infinite walk $\extir\gamma$  consisting only of inverses of arrows such that $\gamma \extir\gamma$ is a walk.
\end{itemize}
Dually, for a left-confined walk $\gamma$, we can define $\extl \gamma = \extir{\gamma^-}^-$ and $\extil \gamma = \extr{\gamma^-}^-$.
\end{definition}
Note that the notation is designed so that `direction' of the bracket is the same as the direction of the letters; see the following Example \ref{eg:inf-extn} (3).  Reader experienced with string combinatorics should be noted that these operations above are related to, but not the same as, the notion of attaching and removing (co)hooks.

\begin{example}\label{eg:inf-extn}
(1) Consider a stationary walk $\gamma:=\un{q}{r}$.  If $q\in Q'\setminus Q$, then $\extl \gamma=q$.

(2) If $\gamma$ is a walk given by an oriented cycle in $Q$ and $\gamma^2\neq 0$, then we have 
\[
 \extr\gamma=\gamma^\infty,\;\; \extir\gamma=\gamma^{-\infty},\;\; \extl\gamma={}^\infty \gamma,\;\; \extil\gamma={}^{-\infty}\gamma.
\]

(3) Consider the $\vec{A}_2$-quiver as in Example \ref{eg:blossom}.  Let $\gamma:= a_0$.  Then we have $\extl\gamma = a_1, \extr\gamma = a_2$, $\extil\gamma = b_2^-, \extir\gamma = c_1^-$; c.f. Figure \ref{fig:A2-all-walks}.
\end{example}

\begin{definition}[Overlap]
Consider two walks $\gamma$ and $\delta$ with a \emph{maximal} common subwalk $\omega$ between them.
The associated \defn{overlap} is the corresponding pair of decomposition $\gamma = \gamma_1 \omega \gamma_2$ and $\delta = \delta_1 \omega \delta_2$, which will be denoted in the following for clarity:
 \[
  \crosswn {\gamma}{\gamma_1} {\omega} {\gamma_2}{\delta}{\delta_1}{\delta_2.} 
 \]

If $\omega$ is stationary, we may just abbreviate to the form:
 \[
  \crosswn {\gamma}{\gamma_1} {} {\gamma_2}{\delta}{\delta_1}{\delta_2.}
 \]
\end{definition}

 We stress that overlap depends on the position of $\omega$ appearing in $\gamma$ and in $\delta$.  

\begin{definition}[Crossings]\label{def:cross}
 For two walks $\gamma$ and $\delta$, a \defn{positive crossing from $\gamma$ to $\delta$} is an overlap of the form
 \begin{align}
  \crosswn{\gamma}{\gamma_1 q^{-}}{\omega}{q' \gamma_2}{\delta}{\delta_1 r}{r'^{-} \delta_2}\label{eq-crossing}
 \end{align}
 where $q$, $q'$, $r$ and $r'$ are arrows. We denote by $c^+(\gamma, \delta)$ the set of positive crossings from $\gamma$ to $\delta$. Notice that there is an immediate identification $c^+(\gamma, \delta) = c^+(\gamma^{-}, \delta^{-})$. For two strings $\st \gamma$ and $\st \delta$, we write $c^+(\st \gamma, \st \delta) = c^+(\gamma, \delta) \sqcup c^+(\gamma, \delta^{-})$. We say that $\st \gamma$ and $\st \delta$ are \defn{crossing} if $c^+(\st \gamma, \st \delta) \neq \emptyset$ or $c^+(\st \delta, \st \gamma) \neq \emptyset$.
\end{definition}
\begin{remark}
For the readers experienced with $\tau$-tilting theory \cite{AIR}, the direction we chosen matches with the partial order of support $\tau$-tilting modules.  Namely, let $\Lambda$ be the associated (finite-dimensional, for simplicity) gentle algebra, and $M, N$ be the indecomposable $\tau$-rigid modules corresponding to \emph{infinite} strings $\st\gamma$ and $\st\delta$ respectively, then $c^+(\st{\gamma}, \st{\delta})\neq \emptyset$ is equivalent to $\Hom_\Lambda(M,\tau N)\neq 0$.
Note also that this is also equivalent to $\Hom_{\mathcal{T}}(V,U[1])\neq 0$ for the projective presentation $U,V$ of $M,N$ respectively and $\mathcal{T}$ the bounded homotopy category of finitely generated projective $\Lambda$-modules.
In fact, $c^+(\st\gamma, \st\delta)$ is a canonical basis of $\Hom_{\mathcal{T}}(V,U[1])$, c.f. \cite{EJR}.
\end{remark}

\begin{example}
Consider the walks $\gamma=c_2^-a_2$, $\delta=b_2^-a_0c_1$, $\eta=b_2^-a_0a_2$ on the $\vec{A}_2$-quiver.
Then there is a positive crossing from $\gamma$ to $\delta$ given by
\[
\crosswn {\gamma}{c_2^-}{}{a_2}{\delta}{b_2^-a_0}{c_1}
\]
(with overlap $\un{a_0}{a_1}$).  On the other hand, one can check that there is no positive crossing from $\gamma$ to $\eta$, and vice versa.  In practice, it is often easier to find crossings between \emph{infinite} strings using the topological model, because of the existence of a crossing implies the existence of an intersection between the corresponding curves; see $\gamma$ and $\delta$ using Figure \ref{fig:A2-all-walks}.
\end{example}

We will freely swap the rows in the diagram \eqref{eq-crossing}.
In particular, if we have not fixed the choice of walk representatives $\gamma, \delta$ of the strings $\st\gamma, \st\delta$, then $\st\gamma, \st\delta$ is crossing means that we can write a diagram of the form:
\[
\crosswn{\gamma}{\gamma_1 q^\pm}{\omega}{q'^\mp\gamma_2}{\delta}{\delta_1r^\mp}{r'^{\pm}\delta_2},
\]
where the signs on the four letters adjacent to $\omega$ are determined by whether the crossing turns out to be a positive crossing from $\st\gamma$ to $\st\delta$ (in which case, we have $q^-, q', r, r'^-$), or a positive crossing from $\st\delta$ to $\st\gamma$ (in which case, we have $q, q'^-, r^-, r'$).

The following will be useful to reduce the verification process on whether a string crosses itself.
\begin{lemma} \label{samedir} 
 Consider a walk $\gamma \delta$. Then we have a possibly non-disjoint union of sets 
 \[c^+(\gamma \delta, \delta^- \gamma^-) = c^+(\gamma, \delta^- \gamma^-) \cup c^+(\delta, \delta^- \gamma^-) \cup c^+(\gamma \delta, \gamma^-)
 \cup c^+(\gamma \delta, \delta^-)\]
 where the identification are obvious.
\end{lemma}

\begin{proof}
 Via the obvious identification, $\supseteq$ is trivial. For the other inclusion, notice that any crossing $\eta \in c^+(\gamma \delta, \delta^- \gamma^-)$ that is not in the right hand side has the form
 \[\crosswn{\gamma \delta}{\gamma_1 q^-}{\omega_1 \omega_2 \omega_3}{r \gamma_2}{\delta^- \gamma^-}{\gamma'_1 q'}{r'^- \gamma'_2}\]
 where
 \begin{itemize}
  \item[] $\gamma = \gamma_1 q^- \omega_1$, $\delta = \omega_2 \omega_3 r \gamma_2$, $\delta^- = \gamma'_1 q' \omega_1 \omega_2$ and $\gamma^- = \omega_3 r'^- \gamma'_2$;
  \item[or] $\gamma = \gamma_1 q^- \omega_1 \omega_2$, $\delta = \omega_3 r \gamma_2$, $\delta^- = \gamma'_1 q' \omega_1$ and $\gamma^- = \omega_2 \omega_3 r'^- \gamma'_2$.
 \end{itemize}
 In both cases, $\omega_2 = \omega_2^-$, contradicting Lemma \ref{noselfinverse}.
\end{proof}


\section{Torsion sets and maximal non-crossing sets}\label{sec:torset}


The first step in classifying torsion classes of gentle algebras is to play a similar game \emph{only} on strings, which is the purpose of this section.
Namely, by mimicking the combinatorics of constructing modules in torsion classes, we can label a (combinatorial) torsion classes by `generators' which are certain special sets of infinite strings - this combinatorial labelling is what we called `refined lamination' in the introduction.

\subsection{Torsion sets and non-crossing sets}

We introduce the following combinatorial analogue of torsion class.

\begin{definition}[Factor and extension of strings]\label{def:ext-fac}
For a string $\st \gamma$ in $(Q',R')$, a \defn{factor of $\st \gamma$} is any string $\st \omega$ such that $\gamma = \omega$ or there exist a decomposition $\gamma = \gamma_1 q^{-} \omega r \gamma_2$ or $\gamma = \gamma_1 q^{-} \omega$ or $\gamma = \omega r \gamma_2$, with $q,r\in Q_1'$.

For two strings $\st \gamma$ and $\st \delta$ and $q$ an arrow in $Q_1$ such that $\st{\gamma q \delta}$ is also a string, then we say that $\st{\gamma q \delta}$ is an \defn{extension of $\st \delta$ by $\st\gamma$} (or extension of $\st\delta$ and $\st\gamma$ if it is clear how $q$ is concatenated to the two strings). 
\end{definition}

As a consequence of Proposition \ref{blo-indep(walk)}, we have the following.
\begin{proposition}\label{blo-indep(fac,ext)}
The notions of factor and extension are both independent of the choice of blossoming.
\end{proposition}

\begin{definition}[Torsion set]
A set $\Tr$ of strings in $(Q', R')$ is called a \defn{torsion set} of $(Q,R)$ if it is closed under factors and extensions. 
If, moreover, $\Tr$ consists only of confined strings, we say that $\Tr$ is a \defn{confined torsion set}.
Denote by $\torsQ$ the set of confined torsion sets in $(Q',R')$.  
\end{definition}

If the reader is a representation theorist, we remark that confined torsion set are pretty much the same as the usual torsion class, except that we are taking ``bands without parameter'' instead.
Much of the classification of torsion classes can be done by ignoring the choice of parameters first, and add these extra information back in later; this will be done in the last subsection.

\begin{example}\label{eg:torset}
Consider the case of $\vec{A}_2$-quiver in Example \ref{eg:blossom} (1).  
Define a set \[
\tilde{\Tr} :=\left\{ \begin{array}{ccccc}
\st{(b_2)^{-}a_0a_2}, & \st{(b_2)^{-1}a_0(c_1)^{-}}, & \st{a_1a_0c_1^{-}}, & \st{a_1b_1^{-}}, & \\
\st{b_2}, & \st{(b_2)^{-1}a_0}, &  \st{a_0a_2}, & \st{a_0c_1^{-}}, & \st{a_1},  \\
\st{a_0}, & \st{\un{a_1}{a_0}} & & &
\end{array}\right\}.
\]
Then one can check that $\tilde{\Tr}$ is a torsion set of the blossoming $(Q^\star,R^\star )$ (as well as $(Q', R')$).
Note that $\tilde{\Tr}$ contains a maximal confined torsion set $\Tr$ given by the set of confined strings, i.e. $\Tr = \{\st{a_0}, \st{\un{a_1}{a_0}}\}$.
There are two other confined torsion sets strictly contained in $\Tr$, namely, $\{\st{\un{a_1}{a_0}}\}$ and $\emptyset$.
\end{example}

\begin{proposition}
$\torsQ$ is a complete lattice where the partial order is given by inclusion.
\end{proposition}
\begin{proof}
Observe that torsion sets are preserved under taking intersection, so we have meets and joins given by
\[\bigwedge_{i \in I} \Tr_i = \bigcap_{i \in I} \Tr_i,
\quad \text{ and }\quad 
\bigvee_{i \in I} \Tr_i = \bigwedge_{\Tr'\supseteq \Tr_i\;\forall i\in I} \Tr',
\]
which is the same as in the case of torsion classes in an abelian length category verbatim; c.f. Proposition \ref{torclass1}.
\end{proof}

\begin{example}\label{eg:torset2}
Consider the Kronecker quiver $K_2$ in Example \ref{eg:blossom} (2).
Any confined string in $K_2$ comes from precisely one of the following six walks.
\begin{align*}
(1)\,\, \pi_0 &:= {\un{b_{1}}{b_2}}^{-} = \un{a_{1}}{a_2}, & (2)\,\,a_1,& & (3)\,\,\iota_0 &:= {\un{b_{0}}{b_1}}^- = \un{a_{0}}{a_1},\\
(4)\,\,\pi_n&:= (b_1^-a_1)^n, & (5)\,\,b_1, & &  (6)\,\,\iota_n &:= (b_1a_1^-)^n, 
\end{align*}
with $n\in \mathbb{N}_{\geq 1}$.
Let $\Tr_0$ be the confined torsion set that contains all of these confined strings.
Then the complete lattice of confined torsion set has the following Hasse quiver (c.f. \cite[Example 3.6]{DIRRT}):
\[
\xymatrix@R=12pt{
&\Tr_0\ar[lddddd] \ar[rd] & &\\
&& \Tr_1 \ar[d] & \\
&& \Tr_2 \ar[d] &\\
&& \vdots &\\
&& \Tr_\infty\ar[ld]\ar[rd] & \\
\{\st{\pi_0}\}\ar[rddddd] & \Tr_{\infty}\setminus\{\st{a_1}\}\ar[rd] & &\Tr_{\infty}\setminus\{\st{b_1}\}\ar[ld]\\
&& \Tr_\infty' & \\
&& \vdots \ar[d]&\\
&& \Tr_2' \ar[d] &\\
&& \Tr_1' \ar[ld] & \\
& \Tr_0':=\emptyset  && 
}
\]
where
\begin{align*}
\Tr_m & := \Tr_0\setminus\{\st{\pi_r}\mid 0\leq r<m\} \text{ for all } m\in \mathbb{Z}_+\cup\{\infty\}\\
\text{and }\quad \Tr_n' & := \{ \st{\iota_r} \mid 0\leq r < n+1 \} \text{ for all } n\in \mathbb{Z}_+\cup\{\infty\}.
\end{align*}
\end{example}

\begin{definition}[Non-crossing set]
 A set $\Sr$ of strings in $(Q',R')$ is called \defn{non-crossing} if for any $\st \gamma, \st \delta \in \Sr$, $c^+(\st \gamma, \st \delta) = \emptyset$. In particular, it consists of non-self-crossing strings.

 A set of strings is called \defn{maximal non-crossing} if 
 \begin{itemize}
  \item It consists only of infinite strings. 
  \item It is non-crossing.
  \item It is a maximal set with respect to the above properties.
 \end{itemize}
As in the case of torsion sets, since the set of arrows in a blossoming is always the same, maximal non-crossing sets of strings are dependent only on $(Q,R)$, and not on the choice of blossoming $(Q', R')$.  Therefore, we denote by $\maxstr$ the set of such maximal non-crossing sets of strings.
\end{definition}

\begin{definition}[Null strings]\label{def:nullstr}
We call an infinite walk $\gamma$ \defn{null} if all its letter are arrows, or all its letters are inverse of arrows.   An infinite string is null if so is its underlying walk.  Denote by $\nstr=\nstr(Q,R)$ the set of all null infinite strings.  Note that a maximal non-crossing set always contains $\nstr$.
\end{definition} 

\begin{example}\label{eg:NCset}
Consider the $\vec{A}_2$-quiver in Example \ref{eg:blossom} (1).  An example of a maximal non-crossing set of infinite strings is $\{\st{b_2^-a_0a_2}, \st{b_2^-a_0c_1^-}\}\cup \nstr$, where $\nstr=\{\st{a_1a_0a_2}, \st{b_1b_2}, \st{c_1c_2}\}$.  In fact, it is not difficult to find all the maximal non-crossing sets; see Example \ref{eg:mainthm1}.
\end{example}

A null walk consisting of ordinary arrow is more traditionally called \emph{maximal path} of the gentle algebra associated to the \emph{blossoming} quiver.  The terminology can be justified both topologically and algebraically as follows.

In the topological model, null strings correspond marked points that are attached to at least one non-boundary arc; see Figure \ref{fig:A2-all-walks} for the null strings appearing in the above Example \ref{eg:NCset}.  In the literature \cite{FG,FT}, these curves are included in laminations, but it is customary to omit them when drawing out the laminations.  Algebraically, we interpret a null string (as well as any arrow in $Q'\setminus Q$) as a zero map between projective modules - hence the choice of the name.

\medskip

We are now going to detail the combinatorial construction that translates between torsion sets and maximal non-crossing sets.

Let $\Sr$ be a set of strings. We denote by $\fin(\Sr)$ the confined part of $\Sr$, i.e. the set of all confined strings in $\Sr$.
We define
\[
\Tors^\infty(\Sr) := \bigcap_{\text{torsion set }\Tr\supseteq \Sr} \Tr\quad \text{ and }\quad  \Tors(\Sr):=\fin\Tors^\infty(\Sr).
\]
Note that $\Tors^\infty(\Sr)$ is the smallest torsion set containing $\Sr$.
Also, if $\Sr$ contains non-confined strings, then $\Tors^\infty(\Sr)$ is not confined.
On the other hand, $\Tors(\Sr)$ is the maximum confined torsion set contained in $\Tors^\infty(\Sr)$, and so the two torsion sets coincide if $\Sr$ consists only of confined strings.

Consider now a confined torsion set $\Tr\in \torsQ$ and define 
\begin{align*}
L(\Tr)  &:=\{\st \gamma \text{ string in } (Q',R')\mid \fin\{\text{all factors of }\st\gamma\}\subseteq \Tr\} \\
\text{ and }G(\Tr) &:= \{\st \gamma \in L(\Tr) \text{ infinite} \mid  c^+(\st \delta, \st \gamma) = \emptyset \text{ for all }\st\delta\in L(\Tr)\}.
\end{align*}

\begin{example}\label{eg:L(T)G(T)}
Consider the $\vec{A}_2$-quiver and the set $\Sr:=\{\st{b_2^-a_0a_2}, \st{b_2^-a_0c_1^-}\}$ as in Example \ref{eg:NCset}.
Then $\Tors^\infty(\Sr)$ is the torsion set $\tilde{\Tr}$ of Example \ref{eg:torset} and $\Tors(\Sr)=\{a_0, \st{\un{a_1}{a_0}}\}$.
On the other hand, if we take the confined torsion set $\Tr=\{a_0, \st{\un{a_1}{a_0}}\}$, then we have $L(\Tr)=\tilde{\Tr}\cup \nstr$ and $G(\tilde{\Tr}) = \Sr\cup \nstr$.
\end{example}

We will see that $G(\Tr)$ is the set of all infinite strings so that everything in $\Tr$ can be \defn{generated} (hence, the choice of notation $G$) by iterated extensions of their factors.
Note that while $G(\Tr)$ is by definition a non-crossing set of infinite string, {\it a priori} they are not necessarily maximal; but it turns out this is true, and is part of the following the main result of this section.

\begin{theorem} \label{mainthmstrings}
Let $(Q,R)$ be a gentle quiver.
 \begin{enumerate}[\rm (a)]
  \item The set $\maxstr$ of all maximal non-crossing sets of infinite strings is in one-to-one correspondence with the set $\torsQ$ of confined torsion sets given by
\begin{align*} 
\xymatrix@R=-1pt@C=50pt{\maxstr  \ar@{<->}[r]^>(.75){1:1} & \torsQ.\\
\Sr  \ar@{|->}[]!<8ex,0ex>;[r]!<-5ex,0ex>  & \Tors(\Sr) \\
G(\Tr) & \ar@{|->}[]!<-5ex,0ex>;[l]!<8ex,0ex> \Tr}
\end{align*}
  \item Let $\geq$ be the partial order on $\maxstr$ induced by the bijection in (a).
  Then we have $\Sr \geq \Sr'$ if and only if $c^+(\Sr', \Sr) = \emptyset$.
 \end{enumerate} 
\end{theorem}

\begin{example}\label{eg:mainthm1}
For $\vec{A}_2$, we have the following 5 confined torsion sets forming the following complete lattice.
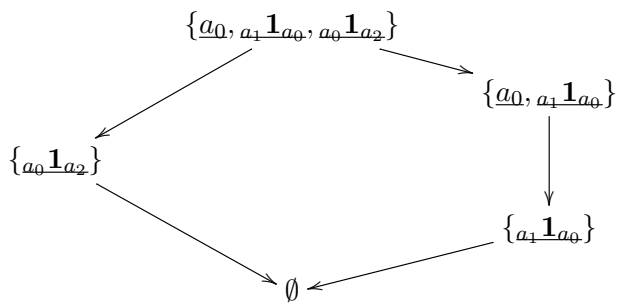
\[
\xymatrix@R=8pt{
 & \{\st{a_0}, \st{\un{a_1}{a_0}}, \st{\un{a_0}{a_2}}\} \ar[ldd]\ar[rd]& \\
 & & \{\st{a_0}, \st{\un{a_1}{a_0}}\} \ar[dd]\\ 
\{\st{\un{a_0}{a_2}}\}\ar[rdd] & & \\ & & \{\st{\un{a_1}{a_0}}\}\ar[ld] \\ & \emptyset &
}
\]
The 5 corresponding maximal non-crossing sets of strings are shown as follows with their subset $\nstr$ of null strings removed.
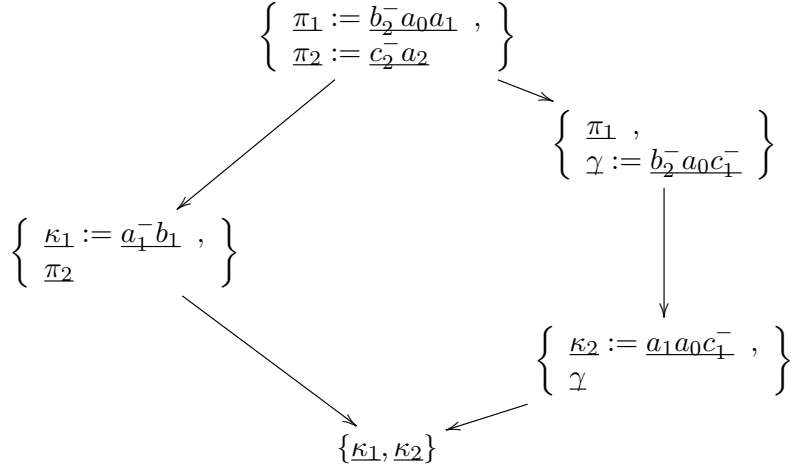
\[
\xymatrix@R=8pt@C=0pt{
& \left\{{\begin{array}{l}\st{\pi_1}:= \st{b_2^-a_0a_1}\;\;, \\ \st{\pi_2}:= \st{c_2^-a_2} \end{array}}\right\} \ar[ldd]\ar[rd]&  \\
& & \left\{{\begin{array}{l}\st{\pi_1}\;\;, \\ \st{\gamma}:=\st{b_2^-a_0c_1^-} \end{array}}\right\}  \ar[dd] \\
\left\{{\begin{array}{l}\st{\kappa_1}:=\st{a_1^-b_1}\;\;, \\ \st{\pi_2} \end{array}}\right\} \ar[rdd] & & \\
& & \left\{{\begin{array}{l}\st{\kappa_2}:=\st{a_1a_0c_1^-}\;\;, \\ \st{\gamma} \end{array}}\right\}  \ar[ld] \\
& \left\{\st{\kappa_1}, \st{\kappa_2}\right\} &
}
\]
\end{example}

\begin{example}\label{eg:mainthm2}
Consider $K_2, K_2^\star$ and the list of confined torsion sets given in Example \ref{eg:torset2}.
We have $\nstr=\{\st{a_0a_1a_2},\; \st{b_0b_1b_2}\}$; note that the corresponding curves are the orange ones surrounding the marked points in Figure\ref{fig:lam-eg}.
Let $\Sr_m, \Sr_n'$ be the maximal non-crossing set of string corresponding to $\Tr_m, \Tr_n'$, respectively, for non-negative integers $m,n$.
Then we have
\begin{align*}
\Sr_m & := \{\st{b_2^-\pi_ma_2}, \;\;\st{b_2^-\pi_{m+1}a_2} \}\cup \nstr\\
\text{and }\quad \Sr_n' &:=\{ \st{b_0\iota_na_0^-},\;\; \st{b_0\iota_{n+1}a_0^-}\}\cup \nstr.
\end{align*}
The case of the remaining five confined torsion sets are listed as follows
\begin{align*}
\{\st{\pi_0}\} & \leftrightarrow \{\st{b_2^-a_2},\;\;\, \st{b_0a_0^-} \}\cup \nstr,\\
\Tr_\infty & \leftrightarrow \{\st{b_2^-\pi_1^\infty},\;\; \st{a_2^-\pi_1^{-\infty}},\;\; \st{{}^\infty\pi_1^\infty} \}\cup \nstr,\\
\Tr_\infty\setminus\{\st{a_1}\} & \leftrightarrow \{\st{b_2^-\pi_1^\infty},\;\; \st{a_0\iota_1^{-\infty}},\;\; \st{{}^\infty\pi_1^\infty} \}\cup \nstr,\\
\Tr_\infty\setminus\{\st{b_1}\} & \leftrightarrow \{\st{a_2^-\pi_1^{-\infty}},\;\; \st{b_0\iota_1^{\infty}},\;\; \st{{}^\infty\pi_1^\infty} \}\cup \nstr,\\
\Tr_\infty' & \leftrightarrow \{\st{a_0\iota_1^{-\infty}},\;\; \st{b_0\iota_1^{\infty}},\;\; \st{{}^\infty\pi_1^\infty} \}\cup \nstr.
\end{align*}
Note that the periodic walk ${}^\infty \pi_1^\infty$ is the same as ${}^\infty \iota_1^\infty$.

Again, for reader with experience working with surface combinatorics of gentle algebras, we present in Figure \ref{fig:kron2} a way to visualise these maximal non-crossing sets without the null strings as `refined laminations' in the terminology used in the Introduction section.
\end{example}

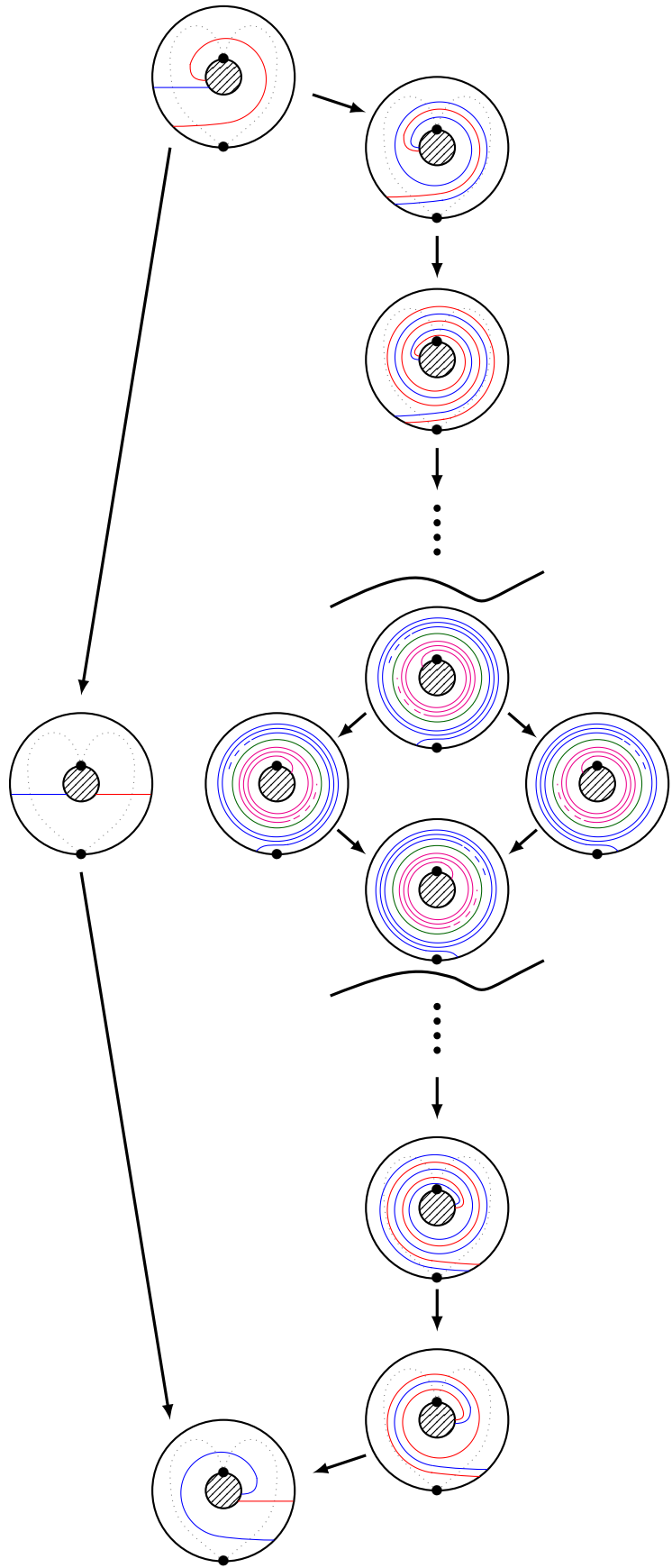
\begin{figure}[!pht]
\centering
\begin{tikzpicture}[scale=0.52]

\hrt{0,0}; 
\cyl{0,0}; 
\strPa{0,0}{red};  
\strPb{0,0}{blue};
\hrt{-4,-20}; \cyl{-4,-20}; \strPb{-4,-20}{blue}; \strPb{0,-20}{red, xscale=-1, shift={(4,0)}}; 
\hrt{0,-40}; \cyl{0,-40}; \strPa{0,-40}{blue, xscale=-1}; \strPb{0,-40}{red, xscale=-1};
 \hrt{6,-2}; \cyl{6,-2}; \strPa{6,-2}{red}; \strPba{6,-2}{blue};
\begin{scope}[shift={(6,-38)}]\hrt{0,0}; \cyl{0,0}; \strPba{0,0}{red, xscale=-1}; \strPa{0,0}{blue, xscale=-1};\end{scope}
\hrt{6,-8}; \cyl{6,-8}; \strPab{6,-8}{red}; \strPba{6,-8}{blue};
\begin{scope}[shift={(6,-32)}]\hrt{0,0}; \cyl{0,0}; \strPba{0,0}{red, xscale=-1}; \strPab{0,0}{blue, xscale=-1};\end{scope}

\draw[black!60!green] (6,-17) circle (1.25);
\cyl{6,-17}; 
\strInftyInner{6,-17}{magenta};
\strInftyOuter{6,-17}{blue}; 

\draw[black!60!green] (1.5,-20) circle (1.25);
\cyl{1.5,-20}; \strInftyOuter{1.5,-20}{blue}; 
\begin{scope}[shift={(1.5,-20)}]\begin{scope}[xscale=-1]\strInftyInner{0,0}{magenta};\end{scope}\end{scope}

\draw[black!60!green] (10.5,-20) circle (1.25);
\cyl{10.5,-20}; \strInftyInner{10.5,-20}{magenta}; 
\begin{scope}[shift={(10.5,-20)}]\begin{scope}[xscale=-1]\strInftyOuter{0,0}{blue};\end{scope}\end{scope}

\draw[black!60!green] (6,-23) circle (1.25);
\cyl{6,-23};
\begin{scope}[shift={(6,-23)}]\begin{scope}[xscale=-1]\strInftyInner{0,0}{magenta};\end{scope}\end{scope}
\begin{scope}[shift={(6,-23)}]\begin{scope}[xscale=-1]\strInftyOuter{0,0}{blue};\end{scope}\end{scope}

\draw[-latex, very thick] (-1.5,-2) -- (-4,-17.5);
\draw[-latex, very thick] (-4, -22.5) -- (-1.5,-38);
\draw[-latex, very thick] (2.5,-0.5) -- (4,-1);
\draw[-latex, very thick] (4,-39) -- (2.5,-39.5);
\draw[-latex, very thick] (6,-4.5) -- (6, -5.7);
\draw[-latex, very thick] (6,-34.3) -- (6,-35.5);
\draw[-latex, very thick] (6,-10.5) -- (6,-11.7);
\draw[-latex, very thick] (6,-28.3) -- (6,-29.5);

\draw [line width=3pt, line cap=round, dash pattern=on 0pt off 2\pgflinewidth] (6, -12.2) -- (6, -13.7) (6,-26.3)--(6,-27.8); 
\draw [-latex, very thick] (4,-18) -- (3.2,-18.7);
\draw [-latex, very thick] (8,-18) -- (8.8,-18.7);
\draw [-latex, very thick] (3.2, -21.3) -- (4,-22);
\draw [-latex, very thick] (8.8,-21.3) -- (8,-22);

\draw [very thick](3,-15) .. controls (5,-14) and (5.5,-14) .. (6.5,-14.5) .. controls (7.5,-15) and (7,-15) .. (9,-14);
\draw [very thick, shift={(0,-11)}](3,-15) .. controls (5,-14.2) and (5.5,-14.2) .. (6.5,-14.5) .. controls (7.5,-15) and (7,-15) .. (9,-14);
\end{tikzpicture}
\caption{Maximal non-crossing sets for the Kronecker quiver}\label{fig:kron2}
\end{figure}

We separate the ingredients of the proof of Theorem \ref{mainthmstrings} into subsections as follows.  
In the next subsection \ref{subsec:sets}, we will give two useful description of the operations $\Tors, \Tors^\infty, L(-)$.
Then in subsection \ref{subsec:completion} and \ref{subsec:max}, we show that $G(\Tr)$ is indeed maximal.
We will then prove that $\Tors:\maxstr\to \torsQ$ is surjective in subsection \ref{subsec:surj}.
In the last subsection \ref{subsec:partb}, we will present the remaining arguments needed to finish the proof of Theorem \ref{mainthmstrings}.

\subsection{Characterisation of various torsion sets}\label{subsec:sets}

We first give a construction of $\Tors(\Sr)$ and $\Tors^\infty(\Sr)$ for any set of strings $\Sr$.
Let us setup some notation before explaining how.

For a set of strings $\Sr$, we define
\begin{align*}
\Fac^\infty(\Sr) :=\{ \text{factors of }\st\gamma\mid \st\gamma\in\Sr \}\quad\text{ and }\quad \Fac(\Sr) := \fin (\Fac^\infty(\Sr)).
\end{align*}
In particular, $L(\Tr)$ is the set of all strings $\st\gamma$ such that $\Fac\{\st\gamma\}\subset \Tr$.
Let us define also a set $\Filt(\Sr)$ as the smallest set of strings containing $\Sr$ and closed under extensions, i.e. the set of all strings obtained by iterated extensions of strings in $\Sr$.

These constructions are parallel to usual construction in representation theory as presented in subsection \ref{subsec:torclass}.
The following is then a natural expectation from the way we setup the combinatorics.

\begin{lemma} \label{filtfac}
 For a set $\Sr$ of strings, we have \[\Tors^\infty(\Sr) = \Filt \Fac^\infty(\Sr) \quad \text{and} \quad \Tors(\Sr) = \Filt \Fac(\Sr).\]
\end{lemma}
\begin{proof}
 First of all, it is immediate by definition of $\Tors^\infty(\Sr)$ that $\Fac^\infty(\Sr) \subseteq \Tors^\infty(\Sr)$ and therefore that $\Filt \Fac^\infty(\Sr) \subseteq \Tors^\infty(\Sr)$. 

 \begin{claim}
  $\Filt \Fac^\infty(\Sr)$ is a torsion set.
 \end{claim}
 \begin{cproof}
  The only thing to check is that, if $\st{\gamma q \delta} \in  \Filt \Fac^\infty(\Sr)$, then $\st \gamma \in \Filt \Fac^\infty(\Sr)$. It is immediate by definition of $\Filt \Fac^\infty(\Sr)$ that we can write $\gamma q \delta = [\gamma_1 q'^\pm] \gamma_2 q \delta_1 [r'^\pm \delta_2]$, where $\gamma_1 q'^\pm$ and $r'^\pm \delta_2$ may appear or not, $\st{\gamma_1}, \st{\delta_2} \in \Filt \Fac^\infty(\Sr)$ if they appear, and $\st{\gamma_2 q \delta_1} \in \Fac^\infty(\Sr)$. By definition of $\Fac^\infty(\Sr)$, we get $\st{\gamma_2} \in \Fac^\infty(\Sr)$. Therefore, by definition of extensions, $\st{\gamma} = \st{[\gamma_1 q'^\pm] \gamma_2} \in \Filt \Fac^\infty(\Sr)$.
 \end{cproof}
 As $\Sr \subseteq \Filt \Fac^\infty(\Sr)$, we deduce $\Tors^\infty(\Sr) = \Filt \Fac^\infty(\Sr)$. 
 \begin{claim}
  $\Filt \Fac(\Sr)=\fin\Filt \Fac^\infty(\Sr)$.
 \end{claim}
 \begin{cproof}
  It is immediate that $\Filt \Fac(\Sr) \subseteq \Filt \Fac^\infty(\Sr)$ consists only of confined strings. Moreover, if $\st{\gamma} \in \Filt \Fac^\infty(\Sr)$, we have $\gamma = \gamma_0 q_1^\pm \gamma_1 q_2^\pm \cdots q_\ell^\pm \gamma_\ell$ with $\st{\gamma_i} \in \Fac^\infty(\Sr)$ for all $i$. It is immediate that if $\st{\gamma}$ is confined then each $\st{\gamma_i}$ is confined, hence $\st{\gamma_i} \in \Fac(\Sr)$. Therefore $\st \gamma \in \Filt \Fac(\Sr)$.
 \end{cproof}
 Using the claim, we deduce that $\Tors(\Sr) = \Filt \Fac(\Sr)$. 
\end{proof}

Recall that $G(\Tr)$ is defined by looking at infinite strings in a set $L(\Tr)$, and $L(\Tr)$ is formed by looking at the set of all (confined) factors of strings.
In fact, we can show that $L(\Tr)$ is a torsion set; so one can think of $L$ stands for lifting or large.

\begin{lemma}\label{lifttor}
For any torsion set $\Tr$, $L(\Tr)$ is the biggest torsion set such that $\fin(L(\Tr))=\fin(\Tr)$.
\end{lemma}

\begin{proof}
 First of all, it is immediate that $\Tr \subseteq L(\Tr)$.

 Let us prove that $L(\Tr)$ is a torsion set. First, if $\st \gamma, \st \delta \in L(\Tr)$, for any $\st \omega \in \Fac \{\st{\gamma q \delta}\}$, we have two possibilities:
 \begin{itemize}
  \item $\st \omega \in \Fac\{\st \gamma, \st \delta\}$. Then, by definition of $L(\Tr)$, we have $\st \omega \in \Tr$.
  \item Otherwise, we can decompose $\omega = \omega_1 q \omega_2$ with $\st \omega_1 \in \Fac \{\gamma\}$ and $\st \omega_2 \in \Fac \{\delta\}$. Again by definition of $L(\Tr)$, we have $\st \omega_1, \st \omega_2 \in \Tr$. As $\Tr$ is a torsion set, we get that $\st \omega = \st{\omega_1 q \omega_2} \in \Tr$.
 \end{itemize}
 Hence, by definition of $L(\Tr)$, we have $\st{\gamma q \delta} \in L(\Tr)$. Secondly, if $\st{\gamma q \delta} \in L(\Tr)$, we have $\Fac \{\st \gamma\} \subseteq \Fac \{\st{\gamma q \delta}\} \subseteq \Tr$, so by definition of $L(\Tr)$, $\st \gamma \in L(\Tr)$. We have now finished proving that $L(\Tr)$ is a torsion set.

 Since $\Tr$ is a torsion set, $\Fac\{\st\gamma\}\subset \fin\Tr\subseteq \fin L(\Tr)$ for any confined $\st\gamma\in \fin\Tr$.
 Conversely, any $\st \omega \in \fin L(\Tr)$ is, by definition of the construction $(-)^\infty$, we also have $\st\omega\in \fin\Tr$.
 Thus, we have $\fin\Tr=\fin(L(\Tr))$.

 Consider now a torsion set $\Tr'$ such that $\fin\Tr'=\fin \Tr$.  Let $\st \gamma \in \Tr'$ and $\st \omega \in \Fac \{\gamma\}$. By definition of a torsion set, $\st \omega \in \Tr'$.  As $\st \omega$ is confined, $\st \omega \in \Tr$. Therefore, by definition of $L(\Tr)$, we get that $\st \gamma \in L(\Tr)$, so $\Tr' \subseteq L(\Tr)$.
\end{proof}

\subsection{A completion of a non-crossing set of infinite strings}\label{subsec:completion}
The goal in this subsection is Proposition \ref{tomax}.  A consequence of this is that if $\Sr$ is a non-crossing set of infinite strings that is not maximal, then we can always complete it into a maximal one using strings in $L(\Tors(\Sr))$ so that the confined torsion set generated by the completed set is $\Tors(\Sr)$.  To help digest the proof, which spans the whole subsection, let us now think about how one would possibly approach this problem.

Firstly, being non-maximal means that one can adjoin a new infinite string $\st\gamma$ so that $\Sr\cup\{\st\gamma\}$ is still non-crossing.  So now we need to consider a way to guarantee  $\st\gamma$ is in $L(\Tors(\Sr))\setminus \Sr$.  In other words, if we have a new string $\st\gamma\notin L(\Tors(\Sr))$ but $\Sr\cup\{\st\gamma\}$ is non-crossing, then we should ask `can we modify $\st\gamma$ to some $\st{\delta}\in L(\Tors(\Sr))$ so that $\Sr\cup\{\st{\delta}\}$ is non-crossing?'.

By the definition of $L(\Tors(\Sr))$, the assumption of $\st\gamma$ not being in this set says that there is factor $\st\omega$ of $\st\gamma$ not in $\Tors(\Sr)$.  It turns out that $q\omega r^-$ for some arrows $q,r$ so that $\Sr\cup \{\st{q\omega r^-}\}$ is non-crossing.  Since $\st{q\omega r^-}$ is not necessarily infinite, we want to find something to concatenate $q\omega r^-$ with which results in the desired infinite string $\st\delta$.  The natural candidate $\st{\delta_0}$ for the desired $\st\delta$ is to concatenate on the left of $\omega$ a left-infinite walk which is a path in $Q'$, and dually on the right (see Lemma \ref{NCextension} (c); for algebraists, one guiding example is to take $\omega=\un{q}{r}$, in which case this enlarged string corresponds to the stalk complex concentrated in degree $-1$  given by an indecomposable projective module corresponding to $t(q)$).

\begin{figure}[!hbtp]
\begin{tikzpicture}
\draw (0,-1) .. controls (0,0) and (2.7,-0.1) .. (3,1);
\draw (0,1) .. controls (0.4,0.1) and (3,-0.5) .. (3,-1);
\node at (-0.7,-1) {$S\ni \eta$};
\node at (-0.5,1) {$\delta_0$};
\draw (5,-1) .. controls (5,0) and (7.7,-0.1) .. (8,1);
\draw[dotted] (5,1) .. controls (5.3,0.1) and (8,-0.5) .. (8,-1);
\node at (4,0) {Smooth out};
\draw [->,decorate, decoration={zigzag, amplitude=.9, post=lineto, post length=2pt}]  (3,-0.3) --(5,-0.3);
\draw (5,1.2) .. controls (5.5,-0.1) and (7.7,0.1) .. (8,1.2);
\end{tikzpicture}
\caption{Smoothing out a crossing}\label{fig:smooth}
\end{figure}
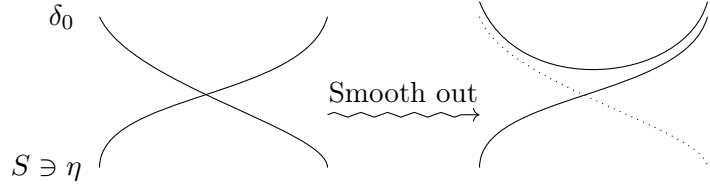

Such naive concatenation does not necessarily result in $\Sr\cup\{\st{\delta_0}\}$ being non-crossing.  Fortunately, if crossing exists, it turns out that we can `smooth out' the crossing (like in knot theory) using strings in $\Sr$ (see Figure \ref{fig:smooth}).  Moreover, one also need to check the smoothed out string $\st\delta$ generates a torsion set that is contained in $\Tors(\Sr)$.  The following three lemmas will be dedicated to show that smoothing out crossing while preserving in the same torsion set is possible.

\begin{lemma} \label{cutstay}
 Let $\Sr$ be a set of strings and $\st \gamma \in L(\Tors(\Sr))$. For any subwalk $\alpha$ of $\gamma$ that does not cross any string of $\Sr$, there exists a subwalk $\hat{\alpha}$ of $\gamma=[\gamma_1]\hat{\alpha}[\gamma_2]$, where $\gamma_1$ or $\gamma_2$ may not appear, such that 
 \begin{itemize}
  \item $\alpha$ is a subwalk of $\hat{\alpha}$;
  \item $\st{\hat{\alpha}}$ does not cross any string in $\Sr$;
  \item any substring of $\st \gamma$ containing strictly $\st{\hat{\alpha}}$ crosses a string of $\Sr$;
  \item $\st{\hat{\alpha}} \in L(\Tors(\Sr))$.
 \end{itemize}
\end{lemma}

\begin{proof}
 Let us take $\alpha_1$ as long as possible such that $\st{\alpha_1 \alpha}$ is a substring of $\st \gamma$ that does not cross any string in $\Sr$ (it is always possible as crossings only involve bounded parts of strings). Then, take $\alpha_2$ as long as possible such that $\st{\hat{\alpha}}:=\st{\alpha_1 \alpha \alpha_2}$ is a substring of $\st \gamma$ that does not cross any string in $\Sr$. Then, the first three points of the lemma are immediate.

 \begin{claim}
  We have $\st{[\gamma_1] \hat{\alpha}} \in L(\Tors(\Sr))$.
 \end{claim}
 \begin{cproof}
  If $\gamma = [\gamma_1] \hat{\alpha}$, it is immediate by assumption.  Otherwise, by maximality of $\hat{\alpha}$, there exists a crossing
  \[\crosswn{\hat{\alpha} \gamma_2}{\gamma'_1 q^\mp}{\omega}{r^\pm \gamma'_2}{\delta}{\delta_1 q'^\pm}{r'^\mp \delta_2}\]
  with $\st \delta \in \Sr$ and $\hat{\alpha} = \gamma'_1 q^\mp \omega$. 
  Recall from Lemma \ref{lifttor} that $L(\Tors(\Sr))$ is a torsion set (containing $\st\gamma$), so it suffices to show that $\st{[\gamma_1]\hat{\gamma}}$ can be obtained from extensions of factors of $\st\gamma$.

  Suppose that $r^\pm = r$.
  Since $\st \gamma = \st{[\gamma_1] \hat{\alpha} r \gamma'_2} \in L(\Tors(\Sr))$, $\st{[\gamma_1]\hat{\alpha}}$ is a factor of $\st{\gamma}\in \Sr$, and the result follows.

  Suppose that $r^\pm = r^{-}$.
  Then we have $q^\mp = q$, $q'^\pm = q'^-$ and $r'^\mp = r'$. 
  In particular, $\st{\gamma}$ is a factor of $\st{\delta}$ and $\st{[\gamma_1]\gamma_1'}$ is a factor of $\st{\gamma}$.
  Since $L(\Tors(\Sr))$ is a torsion set containing $\st{\gamma}$ and $\st\delta$, by definition of a torsion set we have $\st \omega, \st{[\gamma_1]\gamma_1'} \in L(\Tors(\Sr))$.  Hence $\st {[\gamma_1] \gamma'_1 q \omega} = \st{[\gamma_1] \hat{\alpha}} \in L(\Tors(\Sr))$.
 \end{cproof}

  To finish the proof, we applies the claim with $\gamma$ replaced by $\alpha^-[\gamma_1^-]$, which gives $\st{\hat{\alpha}^-} \in L(\Tors(\Sr))$, as required.
\end{proof}

\begin{example}\label{eg:comp1}
Consider $(Q,R)=(K_2,\emptyset)$ with the classical blossoming $(Q^\star, R^\star)$ as in Example \ref{eg:mainthm2}.
Our goal is to adjoin a new infinite string to $\Sr=\{\st{{}^\infty\pi_1^\infty} = \st{\cdots a_1^-b_1 a_1^-b_1 \cdots}\}$ so that the resulting set generate the same torsion class as $\Sr$.
Let us consider $\gamma = b_0b_1a_1^-a_0^-$.
Since $\Fac(\st{\gamma})=\emptyset$, we have $\st\gamma \in L(\Tors(\Sr))$.
Take $\alpha=b_1$, it is easy to see that $\alpha$ does not cross $\st{{}^\infty\pi_1^\infty}$.  In this case the subwalk $\hat{\alpha}$ of Lemma \ref{cutstay} is given by $b_1a_1^-$.  Note that $\Fac(\st{b_1a_1^-}) = \{ \st{\un{b_0}{b_1}}=\st{\un{a_0}{a_1}} \}\subset \Tors(\Sr)$, so $\Tors(\Sr\cup \{\hat{\alpha}\}) = \Tors(\Sr)$.

Since any subwalk containing $\hat{\alpha}$ crosses $\st{{}^\infty\pi_1^\infty}$, in order to obtain a new infinite string from $\hat{\alpha}$, we smooth out positive-crossings from $\st{{}^\infty\pi_1^\infty}$ to $\st\gamma$.  Any such crossing will be of the following form up to shifting letters of $\st{{}^\infty\pi_1^\infty}$:
\[
\crosswn{\gamma}{b_0}{b_1a_1^-}{a_0^-}{{}^\infty\pi_1^\infty}{\cdots b_1a_1^-}{b_1a_1^-\cdots}
\]
According to Figure \ref{fig:smooth}, the smoothed out string will be $b_0(b_1a_1^-)^\infty$.
Note that if we reverse the two walks, then smoothing out gives another infinite string ${}^\infty(b_1a_1^-)a_0^-$.
In either cases, the new infinite string does not cross $\Sr$.
Let us now explain rigourous the construction of $b_0(b_1a_1^-)^\infty$ from smoothing the displayed crossing in the following two lemmas.
\end{example}

\begin{lemma} \label{extendable}
 Let $\Sr$ be a non-crossing set of strings and $\gamma$ be a right-confined walk such that $\Sr \cup \{\st \gamma\}$ is non-crossing. Then there is a walk $\gamma \alpha$ such that
 \begin{enumerate}[\rm(a)]
  \item $\alpha$ is not a stationary walk.
  \item $\Sr \cup \{\st{\gamma \alpha}\}$ is non-crossing;
  \item $\Tors(\Sr \cup \{\st{\gamma \alpha}\}) \subseteq \Tors(\Sr \cup \{\st\gamma\})$.
 \end{enumerate} 
\end{lemma}
\begin{proof}
 Suppose that $\Sr \cup \{\st{\gamma \extir \gamma}\}$ is non-crossing. We put $\alpha = \extir \gamma$, so that (a) anb (b) clearly hold. Consider a string $\st \omega \in \Fac(\Sr \cup \{\st{\gamma \alpha}\})$. If $\st \omega \in \Fac(\Sr)$ then $\st \omega \in \Tors(\Sr \cup \{\st\gamma\})$. Otherwise, $\st \omega$ is a confined factor of $\st{\gamma \alpha}$. As $\st \alpha$ consists only of inverses of arrows and is right-infinite, it implies, by definition of a confined factor, that $\st \omega$ is actually a factor of $\st{\gamma}$, so that $\st \omega \in \Tors(\Sr \cup \{\st\gamma\})$ again. Then, we deduce immediately that (c) holds in this case.

 Suppose now that there is a crossing
 \[\crosswn{\gamma \extir \gamma}{\gamma_1 q^{\pm}}{\omega}{r^{\mp} \gamma_2}{\delta}{\delta_1 q'^{\mp}}{r'^{\pm} \delta_2}\]
 with $\st{\delta}\in \Sr\cup \{\st{\gamma\extir{\gamma}}\}$.
 As $\Sr \cup \{\st\gamma\}$ is a non-crossing set, we have $r^{\mp}$ being in $\extir \gamma$.

 By definition of $\extir{\gamma}$, $r^\pm$ is an inverse arrow $r^-$, and so $q^{\pm} = q$, $q'^{\mp} = q'^{-}$, and $r'^{\pm} = r'$.  In particular, if $\delta=\gamma\extir{\gamma}$, then $\delta_1q'^-\omega$ is a strict prefix of $\gamma$.
 Likewise, if $\delta = (\gamma\extir{\gamma})- = \extl{\gamma^-}\gamma$, then $\omega r'\delta_2$ is a strict suffix of $\gamma^-$.
 Hence, we can always replace $\st \delta$ by $\st \delta \in \Sr \cup \{\st \gamma\}$.
  Moreover, we have decompositions $\gamma = \gamma_1 q \omega'$ and $\extir \gamma = \omega'' r^{-} \gamma_2$ with $\omega = \omega' \omega''$.
  
  Let us write $\delta_2 = z^{-} [\delta_2']$, where $z^{-}$ is the maximal prefix of $\delta_2$ that involves only inverses of arrows (so $\delta_2'$ does not appear if $z^-$ is right-infinite).
  Note that $z$ is stationary if there is no inverse arrow to the right of $r'$ in $\delta$. 
  Writing these new information into the crossing yields
 \begin{equation}
  \crosswn{\gamma \extir \gamma}{\gamma_1 q}{\omega' \omega''}{r^{-} \gamma_2}{\delta}{\delta_1 q'^{-}}{r' z^{-} [\delta'_2]}. \label{cra}
 \end{equation}

 Suppose that $\delta$ and the crossing \eqref{cra} have been chosen in such a way that 
 \begin{itemize}
  \item $\omega''$ in \eqref{cra} is of minimal possible length;
  \item for such an $\omega''$, $z^{-}$ is of minimal length. 
 \end{itemize}

 We define $\alpha := \omega'' r' z^{-}$, and so (a) clearly holds. 
 Let us split the proof of (b) into the following two claims.
 \begin{claim}
  There is no crossing between $\st{\gamma \alpha}$ and any string in $\Sr \cup \{\st \gamma\}$.
 \end{claim}
 \begin{cproof}
  Consider such a crossing:
 \begin{equation}\crosswn{\gamma \alpha = \gamma_1 q \omega' \omega'' r' z^{-}}{z_1 s^{\pm}}{\bar \omega}{t^{\mp} z_2}{\epsilon}{\epsilon_1 s'^{\mp}}{t'^{\pm} \epsilon_2} \label{crb} \end{equation}
 where $\st{\epsilon}\in \Sr\cup \{\st{\gamma}\}$.
 Since $\Sr\cup\{\st{\gamma}\}$ is non-crossing, $t^{\mp}$ is a letter of $\alpha = \omega'' r' z^-$.  Like, since $\delta\in \Sr\cup\{\st{\gamma}\}$, $s^{\pm}$ is a letter of $\gamma_1 q$.  Hence, $\omega'$ appears as a subwalk of $\bar \omega$, i.e. we can write $\bar \omega = \bar \omega_1 \omega' \bar \omega_2$.  Let us consider the following cases.
  \begin{itemize}
   \item If $\bar \omega_1$ is stationary, then $s^{\pm} = q$.
   \item If $\bar \omega_1$ is non-stationary, $\bar \omega_1 = \bar \omega_1^\circ q$, and \eqref{cra} and \eqref{crb} induce an overlap
  \[\crosswn{\delta}{\delta_1 q'^{-}}{\omega' \bar \omega_2}{t^{\mp} z_2 [\delta'_2]}{\epsilon}{\epsilon_1 s'^{\mp} \bar \omega_1^\circ q }{t'^{\pm} \epsilon_2,}\]
  which is not a crossing as $\st \delta, \st \epsilon \in \Sr \cup \{\st \gamma\}$, so that $t^{\mp} = t^{-}$, $t'^{\pm} = t'$, $s^{\pm} = s$ and $s'^{\mp} = s'^{-}$.
  \end{itemize}
 So, in both cases, \eqref{crb} becomes:
 \begin{equation}\crosswn{\gamma \alpha = \gamma_1 q \omega' \omega'' r' z^{-}}{z_1 s}{\bar \omega_1 \omega' \bar \omega_2}{t^- z_2}{\epsilon}{\epsilon_1 s'^-}{t' \epsilon_2} \label{crc} \end{equation}

 If $t^{-}$ is in $\omega''$, \eqref{cra} and \eqref{crc} induce a crossing 
 \[\crosswn{\gamma \extir \gamma}{z_1 s}{\bar \omega_1 \omega' \bar \omega_2}{t^{-} \gamma_2'}{\epsilon}{\epsilon_1 s'^{-}}{t' \epsilon_2}\]
 which contradicts the minimality of $\omega''$ in \eqref{cra}.
 On the other hand, if $t^{-}$ is in $r' z^{-}$, then \eqref{cra} and \eqref{crc} induce a crossing
 \[\crosswn{\gamma \extir \gamma}{z_1 s}{\bar \omega_1 \omega' \omega''}{r^{-} \gamma_2}{\epsilon}{\epsilon_1 s'^{-}}{r' z'^{-} t' \epsilon_2}\]
 which contradicts the minimality of $z^{-}$ in \eqref{cra} as $z'^-$ is shorter than $z^-$.  
 \end{cproof}
 \begin{claim}
  The string $\st{\gamma \alpha}$ has no self-crossing.
 \end{claim}
 \begin{cproof}
  As $\st{\gamma \alpha}$ does not cross $\st \gamma$ or $\st \alpha$ which is a substring of $\st \delta$, by Lemma \ref{samedir}, the walks $\gamma \alpha$ and $\alpha^- \gamma^-$ are not crossing. So any self-crossing of $\st{\gamma \alpha}$ would have the form
 \begin{equation}\crosswn{\gamma \alpha = \gamma_1 q \omega' \omega''r' z^{-}}{z_1 s^{-}}{\bar \omega_1 q \omega' \omega''}{r' z^{-}}{\gamma \alpha = \gamma_1 q \omega' \omega''r' z^{-}}{\gamma'_1 s'}{t'^{-} \gamma'_2.} \end{equation}
 Combining with \eqref{cra}, we get a crossing
 \begin{equation}\crosswn{\delta}{\delta_1 q'^{-}}{\omega' \omega''}{r' z^{-} [\delta'_2]}{\gamma \alpha}{\gamma'_1 s' \bar \omega_1 q}{t'^{-} \gamma'_2,} \end{equation}
 and it contradicts the fact that $\st{\gamma \alpha}$ does not cross $\st \delta \in \Sr$. 
 \end{cproof}
  It remains to prove (c); it suffices to prove that $\st{\gamma \alpha} \in \Tors^\infty(\Sr \cup \{\st \gamma \})$. As $z^{-}$ is defined as the maximal prefix of $\delta_2=z^{-} [\delta'_2]$ that involves only inverses of arrows, and $\st \delta = \st{\delta_1 q'^{-} \omega' \omega'' r' z^{-} [\delta'_2]} \in \Tors^\infty(\Sr \cup \{\st \gamma\})$, we obtain, by definition of a torsion set, $\st{\omega' \omega'' r' z^{-}} \in \Tors^\infty(\Sr \cup \{\st \gamma\})$. On the other hand, the factor $\st{\gamma_1}$ of $\st{\gamma}= \st{\gamma_1 q \omega'} \in \Tors^\infty(\Sr \cup \{\st \gamma\})$ clearly belongs to $\Tors^\infty(\Sr \cup \{\st \gamma\})$ by definition of a torsion set.
  Since the decomposition $\gamma \alpha = \gamma_1 q \omega' \omega'' r' z^{-}$ says $\st{\gamma\alpha}$ is and extension of $\gamma_1$ and $\omega'\omega''r'z^-$, the string $\st{\gamma\alpha}$ must also be in $\Tors^\infty(\Sr \cup \{\st \gamma\})$.
\end{proof}

\begin{example}\label{eg:comp2}
For the case of Example \ref{eg:comp1}, the walks $\omega''$, $r'$, $z$, $\delta_2$ in proof of the lemma are $\un{a_0}{a_1}$, $b_1$, $a_1$, and $(b_1a_1^-)^\infty$ respectively.
\end{example}

Example \ref{eg:comp1} and \ref{eg:comp1} presents a rather small example; we perform the operation in the proof once and obtain the desired infinite string that does not cross $\Sr$ (and preserving generated torsion class) straight away.  In general, we may need to apply this operation repeatedly.

\begin{lemma} \label{infextend}
 Let $\Sr$ be a non-crossing set of strings. Suppose that $\gamma$ is a  right-confined walk such that $\Sr \cup \{\st \gamma\}$ is non-crossing. Then there exists an infinite walk $\tilde\gamma$ such that
 \begin{enumerate}[\rm(a)]
  \item $\gamma$ is a subwalk of $\tilde\gamma$;
  \item $\Sr \cup \{\st{\tilde\gamma}\}$ is non-crossing;
  \item $\Tors(\Sr \cup \{\st{\tilde\gamma}\}) \subseteq \Tors(\Sr \cup \{\st\gamma\})$.
 \end{enumerate}
\end{lemma}

\begin{proof}
The construction goes in the same way as in Lemma \ref{ext3}.
Let $\gamma_0:=\gamma$.
Inductively define a walk $\gamma_i$ for $i\geq 1$ as follows.
\begin{itemize}
\item If $\gamma_{i-1}$ is right-infinite, then $\gamma_{i}:=\gamma_{i-1}$;
\item otherwise, it follows from Lemma \ref{extendable} that we can define $\gamma_{i}:=\gamma_{i-1}\alpha_i$ for some non-stationary walk $\alpha_i$ so that $\Sr\cup \{\st {\gamma_i}\}$ is non-crossing and $\Tors(\Sr\cup \{\st {\gamma_i}\})\subseteq \Tors(\Sr\cup \{\st\gamma\})$.
\end{itemize}

Let us take $\gamma_\infty$ the natural right-infinite walk obtained as a limit of the $\gamma_i$'s.  
Note that it follows from Lemma \ref{extendable} (a) that at least one letter of $\gamma_\infty$ is an inverse of letter.
If there is a crossing in $\Sr\cup\{\st{\gamma_\infty}\}$, since the overlap of a crossing is confined, such a crossing can be restrict to one in $\Sr\cup \{\st{\gamma_i}\}$ for some big enough $i$; hence, a contradiction.

Since any confined string in $\Fac^\infty(\{\st{\gamma_\infty}\})$ must belong to $\Fac(\{\st{\gamma_i}\})\subset \Fac(\Sr\cup \{\st{\gamma_i}\})$ for some big enough $i$, it follows from Lemma \ref{filtfac} that $\Tors(\Sr\cup \{\st{\gamma_\infty}\})\subseteq \Tors(\Sr\cup \{\st{\gamma_i}\}) \subseteq \Tors(\Sr\cup \{\st\gamma\})$.

Now, if $\gamma_\infty$ is already left-infinite, then taking $\tilde{\gamma}:=\gamma_\infty$ finishes the proof.
Otherwise, we run through the inductive construction above replacing $\gamma_0$ by $\gamma_\infty^-$ and take $\tilde{\gamma}$ to be the limiting point of the new sequence.  It is immediate that (a) holds.  By similar arguments as in the previous two paragraphs, we also get (b), (c).
\end{proof}

\begin{example}\label{eg:kron-completion}
Consider the Kronecker quiver $K_2$ and blossoming $K_2'$ as in Example \ref{eg:torset2} and \ref{eg:mainthm2}.
Consider the infinite periodic walk $\iota_\infty := {}^\infty \iota_1^\infty = {}^\infty (b_1^-a_1)^\infty$.  Then $\Sr:=\{\st{\iota_\infty}\}$ is a non-crossing set.  Take any $n\geq 0$, then we have $\Sr\cup \{\st{\iota_n}\}$ also non-crossing.  If we take $\gamma_0:= \iota_0 = \un{a_0}{a_1} $ and follow the first step in the proof of Lemma \ref{infextend} (i.e. the procedure described in the proof of Lemma \ref{extendable}), then we get that $\gamma_i=\gamma_\infty = \extir{\un{a_0}{a_1}} = a_0^-$ for all $i\geq 1$.  Now take $\gamma_0:= a_0$ and run through the same procedure again.  As $\extir{a_0} = b_0^-$ crosses $\iota_\infty$, so the proof of Lemma \ref{extendable} tells us that $\gamma_1 = a_0a_1b_1^-= a_0\iota_1^-$.  Similarly, we have $\gamma_ i = a_0\iota^{-i}$ for all $i\geq 0$, and so $\Sr\cup \{a_0\iota_1^{-\infty}\}$ is a non-crossing set for which one can easily check that $\Tors(\Sr) =  \Tors(\Sr \cup \{\st{\un{a_0}{a_1}}\}) = \Tors(\Sr\cup \{a_0\iota_1^{-\infty}\})$. 

We encourage the reader to try to play the same game with $\gamma_0 = a_1$.  In such a case, we have $\gamma_1 = a_1b_1^-$, $\gamma_2 = (a_1b_1^-)^2$, $\ldots$, and the resulting infinite string $\st{\tilde{\gamma}}$ is $\st{a_2^-\pi_1^{-\infty}}$. 
\end{example}

Before proving the desired completion of a non-crossing set of infinite strings, we need one more lemma which will help keeping track of changes in torsion sets when we perform the `infinite-enhancement' in Lemma \ref{infextend}.

\begin{lemma} \label{NCextension}
\begin{enumerate}[\rm (a)]
  \item For any string of the form $\st{u v^{-}}$ where $u$ and $v$ are paths, $\st{\extl u u v^{-} \extir{v^-}}$ is a non-self-crossing string, \[\Tors(\st{\extl u u v^{-}}) \subseteq \Tors(\st{v}) \quad\text{ and }\quad \Tors(\st{\extl u u v^{-} \extir{v^-}}) = \emptyset.\]

  \item Let $\st\gamma$ is a non-self-crossing string. If $\gamma=[u] q^-\omega$ with $q$ an arrow in $Q$ and $u$ a path in $Q$ that may not appear, then there exists some right-confined walk $\alpha$ involving only arrows such that 
	\begin{enumerate}[\rm (i)]
	\item $\st{\alpha \gamma}$ is a non-self-crossing string;
	\item $\Tors(\st{\alpha \gamma}) = \Tors(\st \omega)$.
	\end{enumerate}

  \item Let $\st\gamma$ is a non-self-crossing string.  If $\gamma:=[u] q^-\omega r [v^-]$ with $q,r$ being arrows in $Q$ and $u,v$ being paths in $Q$ that may not appear, then there exist some right-confined walks $\alpha$ and $\beta$ involving only arrows such that

	\begin{enumerate}[\rm (i)]
	\item $\st{\alpha \gamma \beta^-}$ is a non-self-crossing string;
	\item $\Tors(\st{\alpha \gamma \beta^-}) = \Tors(\st \omega)$.
	\end{enumerate}
\end{enumerate}
\end{lemma}

\begin{proof}
(a) The string $\st{\extl u u v^{-} \extir{v^-}}$ is not self-crossing as it is a path followed by the inverse of path. Any factor of $\st{\extl u u v^{-}}$ is a factor of $\st{\extl u u}$ or a factor of $\st{v}$. As no factor of $\st{\extl u u}$ is confined, $\Fac(\st{\extl u u v^{-}}) \subseteq \Fac(\st v)$ so that by Lemma \ref{filtfac}, $\Tors(\st{\extl u u v^{-}}) \subseteq \Tors(\st{v})$. The second claim follows in the same way.

(b) By the condition on $\gamma$, there is some arrow $p\in Q'$ with $p\gamma$ a walk.  So we have a walk $\alpha$ given by the longest possible right-confined path involving only arrows such that $\st\eta := \st{\alpha \gamma}$ is a non-self-crossing string.  It is immediate by definition of $\st{\eta}$ that $\Tors(\st\eta) \supseteq \Tors(\st\omega)$.

 If $\alpha u$ is left-infinite, then $\Tors(\st\eta) = \Tors(\st\omega)$ is immediate by Lemma \ref{filtfac}. Otherwise, by maximality of $\alpha$, there is a crossing
 \[\crosswn{s \eta = s \alpha u q^{-} \omega}{s}{\alpha u \omega'}{t^{-} \varepsilon_2}{(s\eta)^{\pm}}{\eta_1 s'^{-}}{t' \eta_2,}\]
 where $s$ is the only possible arrow. Comparing the signs of letters on both sides, we deduce the following crossing:
 \[\crosswn{s \eta = s \alpha u q^{-} \omega}{s}{\alpha u \omega'}{t^{-} \varepsilon_2}{(q^{-} \omega)^{\pm}}{\eta_1' s'^{-}}{t' \eta_2'.}\]
 Using the second row, we deduce that $\st{\alpha u \omega'} \in \Tors(\st \omega)$.  On the other hand, using the first row, we have $\st{\varepsilon_2} \in \Tors^\infty(\st \omega)$, and so $\st \eta = \st{\alpha u \omega' t^{-} \varepsilon_2} \in \Tors^\infty(\st \omega)$. Hence, we have $\Tors(\st\eta) \subseteq  \Tors(\st \omega)$.

 (c) By (b), there exists $\alpha$ such that $\Tors(\st{\alpha \gamma}) = \Tors(\st\omega)$. Then, using (b) with $\gamma$ replaced by $\gamma^-\alpha^-$ we obtain a $\beta$ such that $\Tors(\st{\alpha \gamma \beta^{-}}) = \Tors(\st{\alpha \gamma}) = \Tors(\st \omega)$.
\end{proof}

We can finally prove the goal of this section.

\begin{proposition} \label{tomax} 
 Let $\Sr$ be a non-crossing set of infinite strings that is not maximal. Then there exists an infinite string $\st \delta\in L(\Tors(\Sr))\setminus \Sr$ so that $\Sr\cup \{\st\delta\}$ is non-crossing.
 In particular, there is always some $\hat{\Sr}\in \maxstr$ so that $\Tors(\hat{\Sr}) = \Tors(\Sr)$.
\end{proposition}

\begin{proof}
 As $\Sr$ is not maximal, there exists an infinite non-self-crossing  string $\st \gamma$ that is not in $\Sr$ and that does not cross $\Sr$. If $\st \gamma \in L(\Tors(\Sr))$, then take $\st \delta = \st \gamma$ and we are done.

 Let us consider the case when $\st\gamma\notin L(\Tors(\Sr))$.  By definition of $L(-)$, there exists a decomposition $\gamma = \gamma_1 q^{-} \omega r \gamma_2$ such that $\st \omega \notin \Tors(\Sr)$. Let us choose such a decomposition with $\omega$ of minimal length. 

 Let $q'$ and $r'$ be the two arrows such that $q' \omega r'^{-}$ is a walk. 
 \begin{claim}
  $\st{q' \omega r'^{-}}$ is not self-crossing.
 \end{claim}
 \begin{cproof}
 Since $\st\gamma$ (hence $\st\omega$) is not self-crossing, if on the contrary that $\st{q'\omega r'^{-}}$ was self-crossing, the crossing would be of the form (under an appropriate choice of the underlying walk $\omega$):
 \[\crosswn{q' \omega r'^{-}}{q'}{\omega'}{t^{-} \varepsilon_2 r'^{-}}{\omega^{\pm}}{\varepsilon_1' s'^{-}}{t' \varepsilon_2'}.\]
 By minimality of $\omega$, using the second row, $\st{\omega'} \in \Tors(\Sr)$. We also get that $t^{-} \varepsilon_2 r$ is a subwalk of $\gamma$, so by minimality again, $\st{\varepsilon_2} \in \Tors(\Sr)$. By definition of a torsion set, we deduce that $\st \omega = \st{\omega' t^{-} \varepsilon_2} \in \Tors(\Sr)$, which is a contradiction.  
 \end{cproof}
 \begin{claim}
  $\st{q' \omega r'^{-}}$ does not cross any string $\st \varepsilon$ of $\Sr$.
 \end{claim}
 \begin{cproof}
 As $\st \omega$ does not cross $\st \varepsilon$, such a crossing would be, without loss of generality,
 \[\crosswn{q' \omega r'^{-}}{q'[\omega_1 t]}{\omega_2}{r'^{-}}{\varepsilon}{\varepsilon_1 t'^{-}}{r \varepsilon_2,}\]
 where $\omega_1 t$ may not appear.
 By definition of a torsion set, as $\st \varepsilon \in \Sr$, we get $\st{\omega_2} \in \Tors(\Sr)$. If $\omega_1 t$ appears, then $\st{\omega_1}$ is a factor of $\st{\gamma}$ whose length is smaller than that of $\st{\omega}$, which contradicts minimality.  Hence, we have $\st \omega=\st{\omega_2} \in \Tors(\Sr)$ - a contradiction.
 \end{cproof}

 Consider the case when $\omega$ is not a stationary walk.  By Lemma \ref{NCextension} (c) (with $\gamma$ therein being the $q' \omega r'^-$ here), we have two right-confined walks $\alpha,\beta$ both involving only arrows, such that
\begin{enumerate}[\rm (i)]
\item $\st{\alpha q' \omega r'^{-} \beta^{-}}$ is a non-self-crossing string;
\item $\Tors(\st{\alpha q' \omega r'^{-} \beta^{-}}) = \Tors(\st{\omega'})$, where $\omega'$ is a walk so that there is a decomposition $\omega = u s^{-} \omega' t v^{-}$.
\end{enumerate}
 Similarly, if $\omega$ is a stationary walk, then using Lemma \ref{NCextension} (a) (with $\gamma$ therein being the $q'r'^- = q'\omega r'^-$ here),  we can take right-confined walks $\alpha:=\extl{q'}, \beta:=\extl{r'}$ involving only arrows such that the property (i) above is satisfied, and property (ii) is replaced by $\Tors(\st{\alpha q' \omega r'^{-} \beta^{-}}) = \emptyset$.

By minimality of $\omega$, we deduce that $\Tors(\st{\alpha q' \omega r'^{-} \beta^{-}}) \subseteq \Tors(\Sr)$. 
 
 By Lemma \ref{cutstay} (on the substring $\st{q'\omega r'}$ of  $\st{\alpha q' u r'^{-} \beta^{-}}$), there exists a (non-self-crossing) substring $\st{\alpha' q' \omega r'^{-} \beta'^{-}}$ of $\st{\alpha q' u r'^{-} \beta^{-}}$ that is in $L(\Tors(\Sr))$ and that $\Sr\cup\{\st{\alpha q' u r'^{-} \beta^{-}}\}$ is non-crossing.
 Hence, by Lemma \ref{infextend}, there exists an infinite string $\st \delta = \st{\delta_1 q' \omega r'^{-} \delta_2}$ so that $\Sr\cup \{\st{\delta}\}$ is non-crossing and $\Tors(\Sr\cup\{\st\delta\})\subseteq \Tors(\Sr\cup\{\st{\alpha q' u r'^{-} \beta^{-}}\})\subseteq L(\Tors(\Sr))$.  Thus, $\st{\delta}\in L(\Tors(\Sr))$.  As $\st\delta$ crosses $\st \gamma$, it is not in $\Sr$; this finishes the proof.
\end{proof}

\subsection{Maximality of $G(\Tr)$} \label{subsec:max}

Now we want to consider the case of $\Sr=G(\Tr)$ in Proposition \ref{tomax} for some confined torsion set $\Tr$.  The idea is to use the algebraic theory: a non-crossing set of infinite strings is the combinatorial analogue of (the projective presentation $P_U$ of) Ext-projective $U$ of the torsion class $\Tors(U)$, i.e. $\Ext^1(U,M)=0$ (or $\Hom(P_U, M[1])=0$ in the bounded homotopy category) for all $M\in \Tors(U)$.  The combinatorial meaning of this Ext-orthogonality will be given in Lemma \ref{condp}.  We need an intermediate step first.

\begin{lemma} \label{condp1}
 Let $\Sr$ be a set of infinite strings and $\st \gamma$ be a string such that $c^+(\st \delta, \st \gamma) = \emptyset$ for any $\st \delta \in \Sr$. Then $c^+(\st \eta, \st \gamma) = \emptyset$ for any $\st\eta \in L(\Tors(\Sr))$.
\end{lemma}

\begin{proof}
 Let us first prove the following claim.
 \begin{claim}
   There is no decomposition $\gamma = \gamma_1 s \omega t^{-} \gamma_2$ with $\st \omega \in \Tors(\Sr)$
 \end{claim}
 \begin{cproof}
 Suppose the contrary.  Take $\st\omega$ so it is of minimal length with respect to the existence of such decomposition.  As $\Tors(\Sr) = \Filt \Fac(\Sr)$ by Lemma \ref{filtfac}, $\omega \in \Fac\{\delta\}$ for some $\st \delta \in \Sr$ or $\omega = \omega_1 q \omega_2$ with $\st{\omega_1}, \st{\omega_2} \in \Tors(\Sr)$.

 In the first case, since $\omega$ is confined and $\delta$ is infinite, we have a decomposition $\delta = \delta_1's'^-\omega t'\delta_2$ for some arrows $s'$ and $t'$.  This gives a positive crossing in $c^+(\delta, \gamma)$, which contradict the assumptions.
 In the second case, we have $\gamma = \gamma_1 s \omega_1 q \omega_2 t^{-} \gamma$, which contradict the induction hypothesis for $\st{\omega_2}$.
 \end{cproof}

 Consider now $\st\eta\in L(\Tors(\Sr))$.
 Assume on the contrary that $c^+(\st{\eta}, \st{\gamma})\neq \emptyset$.
 Then (up to reversing the orientation on the underlying walks) there is a crossing 
 \[\crosswn{\gamma}{\gamma_1 r}{\omega}{r^{-} \gamma_2}{\eta}{\eta_1 q'^{-}}{r' \eta_2}.\]
 The second row gives $\st \omega \in \Tors(\Sr)$, and so the first row gives a decomposition of $\gamma$ that contradicts the above observation.
\end{proof}

\begin{lemma} \label{condp}
 Let $\Sr$ be a non-crossing set of infinite strings. Then, for any $\st \gamma \in \Sr$ and $\st\delta\in L(\Tors(\Sr))$, $c^+(\st \delta, \st \gamma) = \emptyset$.
\end{lemma}

\begin{proof}
 As $\Sr$ is non-crossing, $c^+(\st \delta, \st \gamma) = \emptyset$ for any $\st \gamma, \st \delta \in \Sr$, so by Lemma \ref{condp1}, $c^+(\st \delta, \st \gamma) = \emptyset$ for any $\st \gamma \in \Sr$ and $\st\delta\in L(\Tors(\Sr))$.
\end{proof}

\begin{proposition}\label{GisMax}
For a confined torsion set $\Tr$, $G(\Tr)$ is a maximal non-crossing set of infinite strings.
\end{proposition}
\begin{proof}
Note that by definition, $G(\Tr)$ is a non-crossing set of infinite strings.  Suppose $G(\Tr)$ is not maximal, then by Proposition \ref{tomax}, there exists an infinite string $\st \gamma \in L(\Tr) \setminus G(\Tr)$ such that $G(\Tr) \cup \{\st \gamma\}$ is non-crossing.  So, by Lemma \ref{condp}, $c^+(\st \delta, \st \gamma) = 0$ for any $\st \delta \in \Tors(G(\Tr)) = \Tr$ so $\st \gamma \in G(\Tr)$. It is a contradiction.
\end{proof}

\subsection{Generation of torsion sets}\label{subsec:surj}

This subsection is dedicated to showing that $G(\Tr)$ does generate $\Tr$, i.e. $\Tr=\Tors(G(\Tr))$ (in particular, $\Tors$ is a surjective map).
Instead of working directly with $G(\Tr)$, let us consider a larger set of strings:
$$G^\circ(\Tr) := \{\st \gamma \in L(\Tr) \mid c^+(\st \delta, \st \gamma) = \emptyset \text{ for all }\st\delta\in L(\Tr)\}.$$
In particular, $G(\Tr)$ is just the set of infinite strings in $G^\circ(\Tr)$.

\begin{lemma} \label{s0generate}
 We have $\Tr = \Tors(\fin( G^\circ(\Tr)))$.
\end{lemma}

\begin{proof}
 It is immediate that $\Tr \supseteq \Tors(G^\circ(\Tr))$. Let us prove the converse inclusion.

 Let $\st \delta \in \Tr$. We prove $\st \delta \in \Tors(\{\st \gamma \in G^\circ(\Tr) \text{ confined}\})$ by induction on the length of $\delta$. If $c^+(\st \epsilon, \st \delta) = \emptyset$ for any $\st \epsilon \in \Tr$, by definition, $\st \delta \in G^\circ(\Tr)$.  Otherwise, there is a crossing
 \[\crosswn{\delta}{\delta_1 q}{\omega}{r^{-} \delta_2}{\epsilon}{\epsilon_1 q'^{-}}{r' \epsilon_2}\]
 for some $\st \epsilon \in \Tr$. The second row gives $\st \omega \in \Tr$, and the first row gives $\st \delta_1, \st \delta_2 \in \Tr$. By induction hypothesis, $\st \delta_1, \st \delta_2, \st \omega \in \Tors(\fin( G^\circ(\Tr)))$. So, as $\Tors(\fin( G^\circ(\Tr)))$ is torsion set, we deduce $\st \delta = \st{\delta_1 q \omega r^{-} \delta_2} \in \Tors(\fin( G^\circ(\Tr)))$.
\end{proof}

In order to replace $\fin(G^\circ(\Tr))$ in Lemma \ref{s0generate} by the set $G(\Tr)$ of infinite strings in $G^\circ(\Tr)$, which means that  strings in $\fin(G^\circ(\Tr))$ are in $\Tors(G(\Tr))$.  To do this, we  show that strings in $\fin(G^\circ(\Tr))$ can be concatenated into the infinite ones in $G^\circ(\Tr)$.
Hence, the strategy is analogous to what we do in completing a non-crossing set, i.e. Lemma \ref{extendable} and Lemma \ref{infextend}, but this time we have a different set of conditions to check (instead of `$\Sr\cup\{\st\gamma\}$ non-crossing $\Rightarrow\;\Sr\cup\{\st{\gamma\beta}\}$ non-crossing, we want `$\gamma\in G^\circ(\Tr)\Rightarrow \gamma\beta\in G^\circ(\Tr)$').
The following lemma is analogous to Lemma \ref{extendable}.

\begin{lemma} \label{ext2}
 For any right-confined walk $\gamma$ with $\st \gamma \in G^\circ(\Tr)$,
 we have a string $\st{\gamma q \eta}\in G^\circ(\Tr)$ for some walk $\eta$ and some arrow $q$ .
\end{lemma}

\begin{proof}
 As $\gamma$ is right-confined, there exists an arrow $q$ such that $\gamma q$ is a walk. 
 \begin{claim}
  The string $\st{\gamma q}$ is non-self-crossing.
 \end{claim}
 \begin{cproof}
  If it was self-crossing, the crossing would have the form
 \begin{equation} \crosswn{\gamma q}{\alpha_1 s^{-}}{\omega}{q}{\gamma^\pm}{\alpha'_1 s'}{q'^{-} \alpha'_2}. \label{eqcrd} \end{equation}
 Let us choose $\omega$ of minimal length. 

 The string $\st{s^{-} \omega}$ is not self-crossing because it is a substring of $\st \gamma$, so a self-crossing of $\st{s^{-} \omega q}$ would have the form
 \[\crosswn{s^{-} \omega q}{\varepsilon_1 s''^{-}}{\omega'}{q}{\omega^\pm}{\varepsilon_1' s'''}{q''^{-} \varepsilon_2}.\]
 Combining with \eqref{eqcrd}, we obtain a crossing
 \[\crosswn{\gamma q}{\alpha_1 \varepsilon_1 s''^{-}}{\omega'}{q}{\gamma^\pm = (\alpha_1 s^{-} \omega)^\pm}{[\alpha_1 s^{-}]\varepsilon_1' s'''}{q''^{-} \varepsilon_2[s \alpha_1^{-}]},\]
 contradicting the fact that $\omega$ has been taken minimal, so that in fact $\st{s^{-} \omega q}$ is not self-crossing. 
 
 We define a walk $\epsilon:=[\alpha]s^-\omega q[\beta^-]$, where $\alpha$ (resp. $\beta$) appears if and only if $s\in Q$ (resp. $q\in Q$), in which case $\alpha$ (resp. $\beta$) is given as in (resp. the left-analogue of) Lemma \ref{NCextension} (b).

 In any case, we have $\st{\epsilon} \in {L(\Tors(\st\omega))}$ using Lemma \ref{NCextension} and Lemma \ref{lifttor}.
 Since $\omega$ is a factor of $\gamma$, we have $L(\Tors(\st \omega)) \subseteq {L(\Tors(\st \gamma))} \subseteq {L(\Tr)}$ and so $\st{\epsilon}\in L(\Tr)$.

 Using \eqref{eqcrd}, we can find a crossing
 \[\crosswn{ [\alpha]s^{-} \omega q [\beta^{-}]}{[\alpha] s^{-}}{\omega}{q [\beta^{-}]}{\gamma^\pm}{\alpha'_1 s'}{q'^{-} \alpha'_2,}\]
 which contradicts the defining property of $\st \gamma \in G^\circ(\Tr)$. 
 \end{cproof}

 Since $q\in Q$, by Lemma \ref{NCextension} (c), there exists a  walk $v$ involving only arrows so that $\st{\gamma q v^{-}}$ is not self-crossing and $\st{\gamma q v^{-}} \in L(\Tors(\st \gamma)) \subseteq L(\Tr)$.  Define $\eta$ to be the shortest prefix of $q v^-$ so that $\st{\gamma \eta^-} \in L(\Tr)$.
 
 To conclude the proof, it suffices to see that $\st{\gamma \eta} \in G^\circ(\Tr)$. If it was not the case, there would be a crossing
 \[\crosswn{\gamma \eta}{\gamma_1 q'}{\omega}{r'^{-} \gamma_2}{\delta}{\delta_1 q''^{-}}{r'' \delta_2} \]
 for some $\st \delta \in L(\Tr)$. As $c^+(\st \gamma, \st \delta) = \emptyset$, $r'^{-}$ is necessarily a letter of $\alpha$, so we are done if the subwalk $\gamma_1q'\omega$ of $\gamma \eta$ is in $L(\Tr)$, as this will contradict the minimality of $\eta$.
 Indeed, as $L(\Tr)$ is a torsion set (by Lemma \ref{lifttor}),  if $\st{\gamma q \alpha} \in L(\Tr)$, then so is its factor $\st{\gamma_1}$.
 Likewise, as $\st\delta\in L(\Tr)$, we have its (confined) factor $\st \omega\in L(\Tr)$.
 Hence, we have the extension $\st{\gamma_1 q' \omega} \in L(\Tr)$.  
\end{proof}

Now we have the analogue of Lemma \ref{infextend}.

\begin{lemma}\label{ext3}
Let $\gamma$ be a right-confined walk with $\st{\gamma}\in G^\circ(\Tr)$.
Then there is a right-infinite walk $\gamma_\infty$ so that $\Fac^\infty\{\st{\gamma}\}\subset \Fac^\infty\{\st{\gamma_\infty}\}$ and $\st{\gamma_\infty}\in G^\circ(\Tr)$.
\end{lemma}
\begin{proof}
Let $\gamma_0:=\gamma$.
Inductively define a walk $\gamma_i$ for $i\geq 1$ as follows.
\begin{itemize}
\item If $\gamma_{i-1}$ is right-infinite, then $\gamma_{i}:=\gamma_{i-1}$;
\item otherwise, applying the construction in Lemma \ref{ext2} to define $\gamma_{i}:=\gamma_{i-1}\alpha_i\in G^\circ(\Tr)$ for some non-stationary walk $\alpha_i$ whose first letter is an arrow.
\end{itemize}
In particular, $\st{\gamma}$ is a factor of $\st{\gamma_i}$ for all $i\geq 0$.

Let us take $\gamma_\infty$ the natural right-infinite walk obtained as a limit of the $\gamma_i$'s.  It is immediate that $\st {\gamma}$ is a factor of $\st{\gamma_\infty}$, and so the claimed inclusion of sets follows.

It remains to argue that $\st {\gamma_\infty} \in G^\circ(\Tr)$. Indeed, any crossing in $c^+(\st \delta, \st{\gamma_\infty})$ for $\st\delta\in L(\Tr)$ can be restricted to a prefix of $\gamma_\infty$, hence to one of the $\gamma_i$'s, which contradicts the construction that $\st{\gamma_i}\in G^\circ(\Tr)$.
\end{proof}

We have enough tools to prove the goal of this subsection.

\begin{proposition} \label{sgenerate}
For any confined torsion class $\Tr\in \torsQ$, we have $\Tr = \Tors(G(\Tr))$.
\end{proposition}

\begin{proof}
  Since $G(\Tr)=\{\st \gamma \in G^\circ(\Tr) \text{ infinite}\}\subseteq G^\circ(\Tr) \subseteq L(\Tr)$, we have
\[
 \Tors(G(\Tr)) \subseteq \Tors(G^\circ(\Tr)) \subseteq \Tors(L(\Tr)).
\]
 By Lemma \ref{lifttor}, $L(\Tr)$ is already a torsion set, so $\Tors(L(\Tr))=\fin(L(\Tr))$, and by Lemma \ref{lifttor} again, this in turn is equal to $\Tr$.
 Hence, we have $\Tors(G(\Tr))\subseteq \Tr$.

 To show that $\Tors(G(\Tr))\supseteq \Tr$, by Lemma \ref{s0generate} it is sufficient to prove that any $\st\gamma\in \fin(G^\circ(\Tr))$ also belongs to $\Tr'$.  Indeed, since $\st{\gamma}$ is confined, Lemma \ref{ext3} yields a left-confined walk $\gamma_\infty$ with $\st{\gamma}\in \Fac^\infty\{\st{\gamma_\infty}\}$ and $\st{\gamma_\infty}\in G^\circ(\Tr)$.
 Apply Lemma \ref{ext3} again with the right-confined walk $\gamma_\infty^-$ then yields an infinite walk $\tilde{\gamma}$ with 
 $\st{\tilde{\gamma}}\in G^\circ(\Tr)$ and $\st{\gamma} \in \Fac^\infty\{\st{\gamma_\infty}\}\subseteq \Fac^\infty\{ \st{\tilde{\gamma}}\}$.
 By Lemma \ref{filtfac} and $\st{\gamma}$ being confined, we have $\st{\gamma}\in \Tors(\st{\tilde{\gamma}}) \subset \Tr'$, as required.
\end{proof}

\subsection{The proof}\label{subsec:partb}
We are one more lemma away from finishing the proof of Theorem \ref{mainthmstrings}.
The following provides the translation of the partial order structure to $\maxstr$.

\begin{lemma} \label{crosstoorder}
 Let $\Sr \in \maxstr$ and any set $\Sr'$ of strings such that $c^+(\st\gamma', \st\gamma) = \emptyset$ for all $\st\gamma\in \Sr$ and all $\st\gamma'\in \Sr'$. Then $\Tors(\Sr') \subseteq \Tors(\Sr)$. 
\end{lemma}

\begin{proof}
 Fix any $\st\gamma\in \Sr$.
 Since the assumption says that $c^+(\st\delta,\st\gamma)=0$ for all $\st\delta\in \Sr\cup \Sr'$, it follows from Lemma \ref{condp1} that implies $c^+(\st\epsilon, \st\gamma) = \emptyset$ for any $\st\epsilon\in L(\Tors(\Sr\cup\Sr'))$.   Hence, by definition of $G(-)$, we have $\Sr \subseteq G(\Tors(\Sr' \cup \Sr))$.  On the other hand, maximality of $\Sr$ forces that $\Sr = G(\Tors(\Sr' \cup \Sr))$.  Therefore, applying $\Tors$ on both side and Proposition \ref{sgenerate} yields $\Tors(\Sr) = \Tors(\Sr \cup \Sr')$.
 By the complete lattice property of confined torsion sets, we have $\Tors(\Sr\cup \Sr')=\Tors(\Sr) \vee \Tors(\Sr')$, and so $\Tors(\Sr) \supseteq \Tors(\Sr')$.
\end{proof}

We are now ready to prove Theorem \ref{mainthmstrings}.

\begin{proof}[Proof of Theorem \ref{mainthmstrings}] \eqloc{mainthmstrings}
 (a)  We showed in Proposition \ref{GisMax} that $G(-)$ is a well-defined map.  It is immediate from Proposition \ref{sgenerate} that $\Tors$ is surjective, and so it remains to show that $G(\Tr)$ is the only maximal non-crossing set of infinite strings that generates $\Tr$.
 But this is just a consequence of Lemma \ref{condp1}, which says that any non-crossing set of infinite strings $\Sr$ such that $\Tr = \Tors(\Sr)$ satisfies $\Sr \subseteq G(\Tr)$.

 (b) Let $\Sr,\Sr'\in \maxstr$.  If we have $\Sr\geq \Sr'$, i.e. $\Tors(\Sr') \subseteq \Tors(\Sr)$, then it is clear that $\Sr' \subseteq L(\Tors(\Sr')) \subseteq  L(\Tors(\Sr))$.  So it follows from the definition of $G(\Tors(\Sr))$ and (a) that $c^+(\Sr', \Sr) = \emptyset$.  Conversely, suppose that $c^+(\Sr', \Sr) = \emptyset$.  Then, by Lemma \ref{crosstoorder}, we have $\Tors(\Sr) \supseteq \Tors(\Sr')$, i.e. $\Sr\geq \Sr'$.
\end{proof}


\subsection{Parametrized version}\label{subsec:wts}

We are going to enhance the non-crossing set of infinite strings, which is needed to apply Theorem \ref{mainthmstrings} to the setting of torsion classes for gentle algebras.

From now on, we fix a set $B$, and call its elements \defn{parameters}.

\begin{definition}[Periodic strings, parametrized set of strings]
 A walk $\gamma$ is \defn{periodic} if there is some $r\in \Z$ such that the $i$-th letter in $\gamma$ is equal to the $(i\pm r)$-th letters  in $\gamma$ for all $i\in \Z$.
 A string is periodic if its underlying walk is so.
 For a set $\Sr$ of strings, denote by $\Sr^p$ the subset of non-null periodic strings in $\Sr$.

 A ($B$-)\defn{parametrized} set of strings is a pair $(\Sr, \blambda)$ consisting of a set $\Sr$ of strings and a ($B$-)\defn{parametrization} map $\blambda: \Sr^p\to 2^B$, where $2^B$ is the power set of $B$.
 Denote by $\maxstr[B]$ the set of parametrized maximal sets of non-crossing (infinite) strings, i.e. 
\[
\maxstr[B] := \{ (\Sr,\blambda) \text{ $B$-parametrized set } \mid \Sr \in \maxstr\}.
\]
 We define a relation $\geq$ on $\maxstr[B]$ given by $(\Sr, \blambda)\geq (\Sr', \blambda')$ if $\Sr \geq \Sr'$ (i.e. $c^+(\st\gamma,\st\delta)=\emptyset$ for all $\st\gamma\in \Sr'$ and $\st\delta\in \Sr$) and for any $\st\gamma\in \Sr^p \cap \Sr'^p$, we have $\blambda(\st\gamma)\supseteq \blambda'(\st\gamma)$.
\end{definition}

\begin{proposition}\label{parametrized}
 The set $\maxstr[B]$ is a complete lattice whose joins are given by
\[\arraycolsep=2pt \bigvee_{i \in I} (\Sr_i, \blambda_i) = (\Sr_\vee, \blambda_\vee)\text{, where } \left\{\begin{array}{rcl} \Sr_\vee &:=& \bigvee_{i \in I} \Sr_i,\\
\blambda_\vee(\st\gamma) &:=& \bigcup_{\Sr_i \ni \st\gamma} \blambda_i(\st\gamma) \;\; \forall \st\gamma \in \Sr_\vee^p,  \end{array}\right.
\]
and whose meets are given by
 \[\arraycolsep=2pt \bigwedge_{i \in I} (\Sr_i, \blambda_i) = (\Sr_\wedge, \blambda_\wedge)\text{, where } \left\{\begin{array}{rcl} \Sr_\wedge &:=& \bigwedge_{i \in I} \Sr_i,\\ \blambda_\wedge(\st\gamma) &:=& \bigcap_{\Sr_i \ni \st\gamma} \blambda_i(\st\gamma)\;\; \forall \st\gamma \in \Sr_\wedge^p. \end{array}\right.\]
\end{proposition}

\begin{proof}
 We start by proving that $\maxstr[B]$ is a partially ordered set. The reflexivity and antisymmetry are immediate. For the transitivity, suppose that $(\Sr, \blambda) \leq (\Sr', \blambda')$ and $(\Sr', \blambda') \leq (\Sr'', \blambda'')$. Then we have $\Sr \leq \Sr' \leq \Sr''$. Suppose that $\st\gamma \in \Sr^p \cap \Sr''^p$. By Theorem \ref{mainthmstrings}, as $\Sr \leq \Sr' \leq \Sr''$, $c^+(\Sr, \Sr') = c^+(\Sr', \Sr'') = \emptyset$, so $c^+(\st\gamma, \Sr') = c^+(\Sr', \st\gamma) = \emptyset$, so $\st \gamma \in \Sr'$ by maximality of $\Sr'$. Therefore, we have $\blambda(\st\gamma) \subseteq \blambda'(\st\gamma) \subseteq \blambda''(\st\gamma)$.

 It is immediate that $(\Sr_\vee, \blambda_\vee) \geq (\Sr_i, \blambda_i)$ for all $i \in I$. Suppose that $(\Sr, \blambda) \geq (\Sr_i, \blambda_i)$ for all $i \in I$. Then $\Sr \geq \bigvee_{i \in I} \Sr_i = \Sr_\vee$. Moreover, if $\st\gamma \in \Sr^p \cap \Sr_\vee^p$, for each $i \in I$ such that $\st\gamma \in \Sr_i$, we have $\blambda_i(\st\gamma) \subseteq \blambda(\st\gamma)$. Hence $\blambda_\vee(\st\gamma) \subseteq \blambda(\st\gamma)$. We proved that $(\Sr_\vee, \blambda_\vee) \leq (\Sr, \blambda)$. As a conclusion, we proved that $(\Sr_\vee, \blambda_\vee)$ is the join of all $(\Sr_i, \blambda_i)$. Similarly, $(\Sr_\wedge, \blambda_\wedge)$ is the meet of all $(\Sr_i, \blambda_i)$.
\end{proof}


\section{Reminder on gentle algebras}\label{sec:gentlealg}

In this section, we recall various facts and constructions around the representation theory of gentle algebras.

\subsection{Gentle algebras}

Fix a field $k$ throughout.

\begin{definition}[Gentle algebra]
Let $(Q,R)$ be a gentle quiver.
The \defn{gentle algebra} associated to $(Q,R)$ is the completion of the bound path algebra $kQ/\langle R\rangle$ by the ideal generated by $Q_1$.
\end{definition}
Note that a gentle algebra associated to $(Q,R)$ is finite-dimensional if, and only if, for any cycle (path) $\rho$ in $Q$, $\rho$ is zero as an element of the algebra.
The usual convention in the literature tends to call a gentle algebra (in the sense presented above) a \emph{locally gentle algebra}.
Since finite-dimensionality does not affect any of our arguments, we will use gentle algebra instead for a less bulky exposition.

From now on, we will fix a gentle quiver $(Q,R)$ and a blossoming $(Q',R')$.
Denote by $\Lambda$ the gentle algebra associated to $(Q,R)$ and $\fl \Lambda$ the category of finite-dimensional right $\Lambda$-modules.
By ($\Lambda$-)module, we will always mean the finite-dimensional ones.

\subsection{String and band modules} \label{defmods}

Suppose $\gamma= \cdots a_ia_{i+1} \cdots$ is a walk in $(Q,R)$ with indexing interval $I\subset \Z$.
For convenience, by a \defn{position} in $\gamma$ we mean a number $i$ in the set
\[
\hat{I}:= \begin{cases} I\cup \{1+\max I\}, &\text{if $I$ is bounded above};\\
I, & \text{otherwise.}
\end{cases}
\]
For the position $i \in \hat I$, we denote by $v_i \in Q_0$ the corresponding vertex.  In particular, for $i \in I$, $s(a_i) = v_i$ and $t(a_i) = v_{i+1}$.  Consider the vector space $\bigoplus_{i \in \hat I} k x_i$, we define the action of $\Lambda$ on it as follows:
 \begin{itemize}
  \item for $v\in Q_0$, $x_i e_v = x_i \delta_{v,v_i}$, where $\delta_{v,v_i}=1$ if $v=v_i$ and $0$ otherwise;
  \item for $q \in Q_1$ and $i \in \bar I$,
   \[
    x_iq = \left\{\begin{array}{ll}
                    x_{i+1} & \text{if $q=a_i$,} \\
	            x_{i-1} & \text{if $q=a_i^-$,} \\
                    0 & \text{else.}
                   \end{array}\right.
   \]
 \end{itemize}
It is clear that the resulting $\Lambda$-module is isomorphic to the one defined by $\gamma^-$, and so we denote such a \defn{string module} by $X(\st\gamma)$.
The basis constructed above is called the \defn{canonical basis} of  $X(\st \gamma)$.
Note that if $\gamma = \un{q}{r}$, then $X(\st \gamma)$ is just the corresponding simple.
Clearly, $X(\st\gamma)$ is finite-dimensional if and only if $\st\gamma$ is bounded (equivalently, confined).

Consider now the case when $\gamma$ is periodic that contains at least one arrow and one inverse of arrow in its letters.
A subwalk $\delta$ of $\gamma$ is said to be the \defn{primitive period} of $\gamma$ if $\gamma = {}^\infty\delta^\infty$ and there is no other subwalk $\epsilon$ of $\gamma$ such that $\epsilon^r=\delta$ for some $r>1$.
Fix now such a primitive period $\delta$ and let $J\subset I$ be the indexing set of letters in $\delta$.
Let $\sigma$ be the operation on $\gamma$ given by shifting the letters to the right by $|J|$ places.
Abusing notation, we denoted also by $\sigma$ the induced  \defn{period-shifting} automorphism 
\begin{align*}
\sigma: X(\st\gamma)  & \to X(\st\gamma) \\
 x_i & \mapsto x_{i-|J|}.
\end{align*}
This gives $X(\st\gamma)$ the structure of a $k[T,T^{-1}]$-$\Lambda$-bimodule where $T$ acts as $\sigma$. Therefore, for a finite-dimensional $k[T,T^{-1}]$-module $M$, we can define a $\Lambda$-module 
\[X_M(\st \gamma) := M \otimes_{k[T, T^{-1}]} X(\st \gamma).\]
If $M$ is indecomposable, we call $X_M(\st \gamma)$ a \defn{band module}.

We can describe the action of $\Lambda$ on a band module $X_M(\st\gamma)$ more explicitly.  By construction, $X_M(\st\gamma)$ is free of rank $|J|$ over $k[T,T^{-1}]$ with \defn{canonical basis} $\{x_j\}_{j\in J}$, i.e. we have $X_M(\st \gamma) \cong  \bigoplus_{j \in J} M x_j$ as a vector space.  Then $\Lambda$-module structure is given by 
 \begin{itemize}
  \item for $v\in Q_0$ and $m\in M$, $mx_ie_{v} = mx_i\delta_{v,v_i}$;
  \item for $q \in Q_1$, $i \in J$ and $m \in M$,
   \[
    m x_i q = \left\{\begin{array}{ll}
         m x_{i+1} & \text{if $i < \max J$ and $q=a_i$,} \\
         m x_{i-1} & \text{if $i > \min J$ and $q = a_{i-1}^-$,} \\
        (m\cdot T) x_{\min J} & \text{if $i = \max J$ and $q = a_i$,} \\
        (m\cdot T^{-1}) x_{\max J} & \text{if $i = \min J$ and $q = a_{i-1}^-$,} \\
        0 & \text{else.}
       \end{array}\right.
   \]
 \end{itemize}
Note that $X_M(\st\gamma)\cong X_{\iota(M)}(\st{\gamma^-})$, where $\iota$ is the $k[T,T^{-1}]$-automorphism on $M$ given by swapping the action of $T$ and $T^{-1}$.
In particular, $X_M(\st\gamma)\cong X_{M'}(\st\gamma')$ if and only if (1) $M\cong M'$, and (2) $\gamma'$ can be obtained by $\gamma$ via shifting letters, or reversing the whole walk, or both.
The equivalence class $[\st\gamma, M]$ induced is called a \defn{band with parameter $M$}.


\begin{proposition}{\rm \cite{BR,CB18}}
 The string modules and band modules defined above are indecomposable.
 Moreover, the set of confined strings and bands with parameters in $(Q,R)$ corresponds to the set of isoclasses of indecomposable finite-dimensional $\Lambda$-modules bijectively.
\end{proposition}
\begin{remark}
 If we add the condition of confined to the strings in the first statement, then the entire proposition is a consequence of \cite{BR}.
\end{remark}

\subsection{Morphisms of string and band modules} 

Let $\st \gamma$ be a string and $\mt{d}$ be a decomposition of the form $\gamma = [\gamma_1 q^{-}] \omega [q' \gamma_2]$, where $\gamma_1 q^{-}$ and $q' \gamma_2$ may not appear, and $q$ and $q'$ are arrows if they appear. We fix $\{x_i\}_{i \in \hat I}$ and $\{y_j\}_{j \in \hat J}$ to be the respective canonical bases of $X(\st \gamma)$ and $X(\st \omega)$ as in Subsection \ref{defmods}, where $\hat J$ is identified to a subset of $\hat I$ according to $\mt{d}$. Then, we define 
\begin{align*}
\alpha_{\mt d}: X(\st\gamma)& \to X(\st\omega) \\
x_i & \mapsto  \begin{cases}
y_{i} & \text{ if $i \in \hat J$;}\\
0  & \text{ if $i \in \hat I \setminus \hat J$.}
\end{cases} 
\end{align*}

Similar, for any decomposition $\mt{d}$ of the form $\gamma = [\gamma_1 q] \omega [q'^{-} \gamma_2]$, we define
\begin{align*}
\beta_{\mt d}: X(\st\omega)& \to X(\st\gamma) \\
y_i &  \mapsto x_i .
\end{align*}

%
%



\begin{definition}
 Consider two strings $\st \gamma$ and $\st \delta$. 
 An \emph{almost positive crossing} from $\st \gamma$ to $\st \delta$ is an overlap of the form 
 \[\crosswn{\gamma}{[\gamma_1 q^{-}]}{\omega}{[q' \gamma_2]}{\delta^\pm}{[\delta_1 r]}{[r'^{-} \delta_2]}\] 
 where each of the parts between brackets may not appear. We denote by $c^\geq(\st \gamma, \st \delta)$ the set of almost positive crossings from $\st \gamma$ to $\st \delta$. 
\end{definition}

Notice that $c^+(\gamma, \delta) \subseteq c^\geq(\gamma, \delta)$ in an obvious way. Each choice of an almost positive crossing $\mt d \in c^\geq(\st \gamma, \st \delta)$ induces a morphism $f_{\mt d} : X(\st \gamma) \to X(\st \delta)$ in the following way. We fix directions of $\gamma$ and $\delta$ such that $x$ is an overlap between walks, 
 \[\crosswn{\gamma}{[\gamma_1 q^{-}]}{\omega}{[q' \gamma_2]}{\delta}{[\delta_1 r]}{[r'^{-} \delta_2].}\] 
 Then we have a decomposition $\mt d$ of the form $\gamma = [\gamma_1 q^{-}]\omega[q' \gamma_2]$ and $\mt d'$ of the form $\delta = [\delta_1 r]\omega[r'^{-} \delta_2]$. We define a map
\[f_{\mt d} := \beta_{\mt{d}'}\circ\alpha_{\mt{d}}: X(\st\gamma)\to X(\st\delta).\]


We cite the following result of Crawley-Boevey:
\begin{theorem}{\rm\cite{CB89}} \label{basisMorphString} 
 For two confined strings $\st \gamma$ and $\st \delta$, $\{f_{\mt{d}} \mid \mt{d} \in c^\geq(\st \gamma, \st \delta)\}$
   is a basis of $\Hom_{\Lambda}(X(\st \gamma), X(\st \delta))$.
\end{theorem}

Suppose $\st\gamma$ is a periodic string.
We can define morphisms from band modules arising from $\st\gamma$ to string modules in a similar way by introducing a $k[T,T^{-1}]$-twisting the description before.
More precisely, let $\sigma$ be the period-shifting automorphism on $\gamma$ and on $X(\st{\gamma})$ as in Subsection \ref{defmods}, $M$ be an indecomposable finite-dimensional $k[T,T^{-1}]$-module. 
Then for a confined string $\st \delta$, an almost positive crossing $\mt{d} \in c^\geq(\st \gamma, \st \delta)$, and a $k$-linear form $\alpha:M\to k$, we define 
\begin{align*}
g_{\mt{d}, \alpha}^{\mathrm{sb}} : X_M(\st \gamma) & \to X(\st \delta)\\
m \otimes x &\mapsto  \sum_{\ell \in \Z} \alpha(mT^\ell) f_{\mt{d}}(\sigma^{-\ell}(x)).
\end{align*}
Note that there is a $\Z$-action on the set $c^\geq(\st\sigma, \st\delta)$ induced by the period-shifting $\sigma$ on $\gamma$.  This in turn gives $g_{\sigma (\mt{d}), \alpha}^{\mathrm{sb}} = g_{\mt{d},  \alpha\circ(-\cdot T)}^{\mathrm{sb}}$.  We note also that the sum above is finite as the overlap in $\mt{d}$ is a confined substring of $\st \gamma$.

Similarly, if $\mt{d} \in c^\geq(\st \delta, \st \gamma)$ and $m \in M\cong \Hom_k(k,M)$, then we can define a morphism $g_{\mt{d}, m}^{\mathrm{bs}} : X(\st \delta) \to X_M(\st \gamma)$ by 
\[g^{\mathrm{bs}}_{\mt{d}, m}(x) = m \otimes f_{\mt{d}}(x).\]
As before, the $\Z$-action on $c^\geq(\st\delta, \st\gamma)$ yields the equality $g_{\sigma(\mt{d}), m}^{\mathrm{bs}} = g_{\mt{d}, m T}^{\mathrm{bs}}$.

Consider now the case of a periodic string $\st\delta$.
For a positive crossing $\mt{d} \in c^+(\st \gamma, \st \delta)$ and $\alpha: M \to N$ a  $k$-linear (not necessarily $k[T, T^{-1}]$-linear) map, we define $g_{\mt{d}, \alpha}^{\mathrm{bb}}: X_M(\st \gamma) \to X_N(\st \delta)$ by
\[g_{\mt{d}, \alpha}^{\mathrm{bb}}(m \otimes x) = \sum_{\ell \in \Z} \alpha(m T^\ell) \tens f_{\mt{d}}(\sigma^{-\ell}(x)).\]
Since both domain and range are band modules, the period-shifting induces a $\Z\times \Z$-action on $c^+(\st\gamma, \st\delta)$.
The effect induced on the morphisms are $g_{(1,\sigma)(\mt{d}), \alpha}^{\mathrm{bb}} = g_{\mt{d}, (-\cdot T)\circ \alpha}^{\mathrm{bb}}$ and $g_{(\sigma,1)(\mt{d}), \alpha}^{\mathrm{bb}} = g_{\mt{d},  \alpha\circ(-\cdot T)}^{\mathrm{bb}}$.

The superscript on the three types of morphisms introduced above signifies the input and output (\emph{b}and or \emph{s}tring) of it; we may omit them if the context is clear.

Finally, for $\beta\in \Hom_{k[T, T^{-1}]}(M,N)$, we have a $\Lambda$-linear map 
\begin{align*}
h_\beta:=\beta\tens\mathrm{id}_{X(\st\gamma)} : X_M(\st\gamma) & \to X_N(\st\gamma)\\
m \tens x & \mapsto  \beta(m) \tens x.
\end{align*}

\begin{theorem}[\cite{Kr91, Geiss}] \label{basisMorphBand}
 \begin{enumerate}[\rm (i)]
  \item If $X_M(\st \gamma)$ is a band module and $X(\st \delta)$ is a finite-dimensional string module, then
    \begin{align*}
\Hom_\Lambda(X_M(\st \gamma), X(\st \delta)) & = \bigoplus_{\mt{d} \in c^\geq(\st \gamma, \st \delta) / \Z} \{g_{\mt{d}, \alpha}^{\mathrm{sb}} \mid \alpha \in \Hom_k(M, k)\}, \\
\Hom_\Lambda(X(\st \delta), X_M(\st \gamma)) &= \bigoplus_{\mt{d} \in c^\geq(\st \delta, \st \gamma) / \Z} \{g_{\mt{d}, m}^\mathrm{bs} \mid m \in M\} 
	\end{align*}
	where the $\Z$-action on the sets of positive crossings are induced by the period-shift on $\gamma$.

  \item Suppose $\st\gamma,\st\delta$ are both periodic.  Let $\delta_{\st\gamma,\st\delta}$ be zero when $\st\gamma\neq\st\delta$ and one otherwise.
  Then we have 
   \begin{align*}
\Hom_\Lambda(X_M(\st \gamma), X_N(\st \delta)) =&  \{h_\beta \mid \beta \in \Hom_{k[T, T^{-1}]}(M, N)\}\delta_{\st\gamma,\st\delta} \\
 & \oplus \bigoplus_{\mt{d} \in c^+(\st \gamma, \st \delta) / (\Z \times \Z)} \{g_{\mt{d}, \alpha}^{\mathrm{bb}} \mid \alpha \in \Hom_k(M, N)\},
	\end{align*}
  where the $\Z\times\Z$-action on $c^+(\st\gamma,\st\delta)$ are induced by the period-shifts of the periodic strings.
 \end{enumerate}
\end{theorem}
\begin{remark}
Krause proved this result under the setting of $k$ being algebraically closed in \cite{Kr91}; whereas Gei\ss' result \cite{Geiss} is for maps between the more generalised setting of clannish algebras (instead of gentle or string algebras) -  we are only taking a small special case from it.
\end{remark}

We can use this to obtain some canonical short exact sequences that are useful later on.

\begin{lemma} \label{trivext}
\begin{enumerate}[\rm (i)]
\item  Consider a confined string of the form $\st{\gamma_1 q \gamma_2}$. Then there is a short exact sequence:
  \[0 \to X(\st {\gamma_2}) \to X(\st \gamma) \to X(\st {\gamma_1}) \to 0.\]

\item  Consider a confined string $\st\delta$ of a periodic non-null string $\st\gamma$.  If $\st\omega$ is a primitive period of $\st\gamma$ so that we have decompositions $\delta=\omega^n\omega$ and $\omega=\omega_1 r \omega_2 q^-$, then for any $M\in \mod k[T,T^{-1}]$,
$X(\st\delta)$ is a factor module of $X_M(\st\gamma)$ if $n=0$; otherwise, we have a short exact sequence:
  \[
   0 \to X_M(\st \gamma) \to M \tens_k X(\st\delta) \to M \tens_k X(\st {\omega^{n-1} \omega_1}) \to 0.
  \] 
\end{enumerate}
\end{lemma}

\begin{proof}
(i) Consider the following almost positive crossings
 \[\crosswn{\mt{d}: \gamma}{}{\gamma_1}{q \gamma_2}{\gamma_1}{}{} \quad \text{and} \quad \crosswn{\mt{e}: \gamma_2}{}{\gamma_2}{}{\gamma}{\gamma_1 q}{.}\]
 Then, using notation of Theorem \ref{basisMorphString}, we have a short exact sequence
 \[0 \to X(\st {\gamma_2}) \xto{f_{\mt e}} X(\st \gamma) \xto{f_{\mt{d}}} X(\st {\gamma_1}) \to 0 \]

(ii)  Let $\mt{d}$ be the almost positive crossing in $c^\geq(\st\gamma, \st\delta)$ with overlap $\delta$:
 \[\crosswn{\mt{d}: \gamma}{{}^\infty\omega \delta r \omega_2 q^-}{\omega^n \delta}{r \omega_2 q^- \omega^\infty}{\delta}{}{}.\]
 Then the case when $n=0$ follows by checking the definition of $g_{\mt{d},\alpha}:X_M(\st\gamma)\to X(\st\delta)$, for any linear form $\alpha: M\to k$, that it is an epimorphism.
  
  Suppose now that $n > 0$ and consider the following two almost positive crossings:
  \[\crosswn{\mt{e}: \omega^n \omega_1}{\omega_1 r \omega_2 q^-}{\omega^{n-1} \omega_1}{}{\omega^{n-1} \omega_1}{}{} ,\quad
   \crosswn{\mt{e}':  \omega^n \omega_1}{}{\omega^{n-1} \omega_1}{r \omega_2 q^- \omega_1}{\omega^{n-1} \omega_1}{}{}.
  \]
  It is immediate that 
   \begin{align}\label{twisteq}
    f_{\mt e'} \circ f_{\mt{d}} = f_{\mt e} \circ f_{\mt{d}} \circ \sigma
   \end{align} where $\sigma: X(\st \gamma) \to X(\st \gamma)$ is the period-shift automorphism. Define 
\begin{align*}
  f: X_M(\st \gamma) = M \tens_{k[T, T^{-1}]} X(\st \gamma) &\to M\tens_k X(\st\delta) \\ m \tens x &\mapsto \sum_{\ell \in \Z} m T^\ell \tens f_{\mt{d}}(\sigma^{-\ell})x, \\
  g: M\tens_kX(\st\delta) &\to M\tens_k X(\st{\omega^{n-1}\omega_1})\\
  m \tens x &\mapsto  m T \tens f_{\mt e}(x) - m \tens f_{\mt e'}(x).
\end{align*}
  Now $g\circ f=0$ follows from the identity \eqref{twisteq}. 

  Observe that $\ker g$ is spanned by elements of the form
\[ m \tens x_{i} + m T^{-1} \tens x_{i+p} + \cdots + m T^{-r} \tens x_{i + rp},\] 
  where $m \in M$, ${x_i}$ is the basis introduced in Subsection \ref{defmods}, $p$ is the length of $\omega$, $i \leq p$ and $r \in \{n-1, n\}$ in such a way that $i + rp \leq \dim X(\st {\omega^n \omega_1}) < i + (r+1)p$.  As $f_{\mt{d}}$ is surjective, there exists $z \in X(\st \gamma)$ such that $x_i = f_{\mt{d}}(z)$. Then we have
  \[f(m \tens z) = \sum_{\ell \in \Z} m T^\ell \tens f_{\mt{d}}(\phi^{-\ell} z) = \sum_{\ell = -r}^0 m T^\ell \tens x_{i - \ell p} = y.\]
  We proved that $\ker g = \im f$. It is easy to prove that $\dim (M \tens_k X(\st {\omega^n \omega_1})) = \dim X_M(\st \gamma) + \dim (M \tens_k X(\st {\omega^{n-1} \omega_1}))$, so $f,g$ yields the short exact sequence as claimed.
\end{proof}

\subsection{Classification of finite-dimensional bricks}

Recall that a module $S \in \mod \Lambda$ is a \emph{brick} if $\End_\Lambda(S)$ is a division ring.
A consequence of Theorem \ref{basisMorphString} and \ref{basisMorphBand} (ii) is a simple combinatorial classification of the finite-dimensional bricks of $\Lambda$.
Before presenting that, let us do some simple linear algebra first.

\begin{lemma} \label{classbrickmat}
 For any finite-dimensional $k[T, T^{-1}]$-module $M$, $M$ is a brick if and only if it is simple.
\end{lemma}

\begin{proof}
 It is immediate that if $M$ is simple then it is a brick. Conversly, suppose that $M$ is given by the action of a linear invertible endomorphism $\phi$ on a $k$-vector space $V$. If $M$ is not simple, then there is a non-trivial decomposition $V = V_1 \oplus V_2$ such that, along this decomposition,
  \[\phi = \begin{bmatrix} \phi_{11} & \phi_{12} \\ 0 & \phi_{22} \end{bmatrix}.\]
 Consider $\Phi: \Hom_k(V_2, V_1) \to \Hom_k(V_2, V_1), \alpha \mapsto \phi_{11} \circ \alpha - \alpha \circ \phi_{22}$. If $0 \neq \beta \in \ker \Phi$, then 
 \[\begin{bmatrix} 0 & \beta \\ 0 & 0 \end{bmatrix}\]
 is a non-zero nilpotent endomorphism of $M$, so $M$ is not a brick. Otherwise, there exists $\beta: V_2 \to V_1$ such that $\Phi(\beta) = \phi_{12}$ and we get
 \[\begin{bmatrix} \Id & \beta \\ 0 & \Id \end{bmatrix} \begin{bmatrix} \phi_{11} & \phi_{12} \\ 0 & \phi_{22} \end{bmatrix} \begin{bmatrix} \Id & \beta \\ 0 & \Id \end{bmatrix}^{-1} = \begin{bmatrix} \phi_{11} & 0 \\ 0 & \phi_{22} \end{bmatrix}\]
 so that $M$ is not indecomposable, hence not a brick.
\end{proof}

\begin{proposition} \label{classbrick}
 The set of all finite-dimensional bricks over $\Lambda$ consist
\begin{enumerate}[\rm (1)]
\item string modules of the form $X(\st \gamma)$ where $\st\gamma$ is a confined string satisfying $\# c^{\geq}(\st \gamma, \st \gamma) = 1$
\item and band modules of the form $X_M(\st \gamma)$ where $[\st, \gamma M]$ is a non-self-crossing band with parameter $M$ being a simple $k[T, T^{-1}]$-module.
\end{enumerate}
\end{proposition}
\begin{proof}
 First of all, by Theorem \ref{basisMorphString}, a finite-dimensional string module $X(\st \gamma)$ is a brick if and only if its only non-trivial almost positive selfcrossing is the one with full overlapping part. Consider now a band module $X_M(\st \gamma)$. By Theorem \ref{basisMorphBand}(iv), if $X_M(\st \gamma)$ is a band, we should have $c^+(\st \gamma, \st \gamma) = \emptyset$ as it is immediate that $g''_{\mt{d}, \alpha}$ is never invertible. Then, we obtain in this case that $\End_{\Lambda}(X_M(\st \gamma)) = H \cong \End_{k[T, T^{-1}]}(M)$ so that $X_M(\st \gamma)$ is a brick if and only if $\st \gamma$ is non-selfcrossing and $M$ is a brick. By Lemma \ref{classbrickmat}, $M$ is a brick if and only if it is simple.
\end{proof}


\section{Classification of torsion classes} \label{sec:torcl}

We are going to combine Theorem \ref{mainthmstrings}, the parametrized analogue of non-crossing set of strings, and the classification of bricks of $\Lambda$ to give a classification of torsion classes of $\fl{\Lambda}$.

\subsection{Setup and statement}

Similar to Section \ref{sec:torset}, for a set $\Mr$ of (representatives of isoclasses of) modules, we denote by $\Tors(\Mr)$ the smallest torsion class in $\fl\Lambda$ that contains all the modules in $\Sr$.

Let us denote by $B := \irr k[T, T^{-1}]$ a set of representative of isoclasses of simple $k[T, T^{-1}]$-modules in the category $\mod k[T,T^{-1}]$ of finitely generated (left) $k[T,T^{-1}]$-modules.
This set $B$ will be the set of parameters in the sense of subsection \ref{subsec:wts}.
For a parametrized set $(\Sr, \blambda:\Sr^p\to 2^B)$ of infinite strings, we define the following sets of (representatives of isoclasses of) indecomposable $\Lambda$-modules
\begin{align*}
 X(\Sr) & := \{X (\st \gamma) \mid \st \gamma \in \Sr\},\\
 X_{\blambda}(\Sr^p) & := \{X_{M}(\st \gamma) \mid \st \gamma \in \Sr^p, M \in \blambda(\st \gamma)\}.
\end{align*}
Now we can state our main theorem.  To save some spaces, let us denote
\[
\maxstr[k] := \maxstr[{\irr k[T,T^{-1}]}].
\]

\begin{theorem} \label{classtors}
 There is a complete lattice isomorphism between  $\maxstr[k]$, the lattice of maximal parametrized non-crossing sets of infinite strings, and  $\tors(\fl\Lambda)$, the lattice of torsion classes in $\fl\Lambda$, via
 \begin{align*}
\Tr_\Lambda:\maxstr[k] & \xrightarrow{\sim} \tors(\fl\Lambda) \\
(\Sr, \blambda) &\mapsto \Tr_\Lambda(\Sr,\blambda):= \Tors\big(X(\Fac\Sr)\cup X_{\blambda}(\Sr^p)\big).
 \end{align*}
\end{theorem}

The remaining of this section is devoted to proving Theorem \ref{classtors}.  Our proof is to gather a series of lemmas that eventually combined to say that $\Tr_\Lambda$ is a bijective morphism of complete join-lattice, and then the theorem, as explained in Remark \ref{rem:lattice}, follows immediately.

\subsection{A morphism of complete join-semilattices}

Let us start with some easy observations that will explain how a confined torsion set $\Tors(\Sr)$ and a torsion class $\Tr_\Lambda(\Sr,\blambda)$ is related.

\begin{lemma} \label{remfilt}
 We have $\Tors(X(\Filt(\Mr))) = \Tors(X(\Mr))$ for any set $\Mr$ of confined strings.  In particular, $\Tors(X(\Tors(\Sr))) = \Tors(X(\Fac(\Sr)))$ for any set of strings $\Sr$.
\end{lemma}

\begin{proof}
 The inclusion $\supseteq$ is trivial. For $\subseteq$, if $\st \delta \in \Filt(\Mr)$, we prove by induction on its length that $X(\st \delta) \in \Tors(X(\Mr))$.

 Note that we have $\st \delta \in \Mr$ or $\st \delta = \st{\delta_1 q \delta_2}$ with $\st{\delta_1}, \st{\delta_2} \in \Filt(\Mr)$. In the first case, the result is immediate. In the second case, by induction hypothesis we have $X(\st{\delta_1}), X(\st{\delta_2}) \in \Tors(X(\Mr))$. So, as $\Tors(X(\Mr))$ is closed under extensions, we deduce by Lemma \ref{trivext} that $X(\st \delta) \in \Tors(X(\Mr))$.

 The final statement follows from Lemma \ref{filtfac}.
\end{proof}

Now we can see that $\Tr_\Lambda(\Sr,\blambda)$ is generated by string modules arising from strings in $\Tors(\Sr)$ and band modules arising from bands with parameter given by $\Sr^p$ and $\blambda$.

\begin{lemma} \label{facttot}
 For any $(\Sr, \blambda) \in \maxstr[k]$, we have
 \[\Tr_\Lambda(\Sr, \blambda) = \Tors\Big( X (\Tors(\Sr)) \cup X_{\blambda}(\Sr^p) \Big).\]
\end{lemma}

\begin{proof}
We have
  \begin{align*} \Tors\Big( X (\Tors(\Sr)) \cup X_{\blambda}(\Sr^p) \Big) =& \Tors\Big(X (\Tors(\Sr))\Big) \join \Tors\Big(X_{\blambda}(\Sr^p)\Big) \\
  =& \Tors\Big(X (\Fac(\Sr)\Big) \join \Tors\Big(X_{\blambda}(\Sr^p)\Big)\\ =& \Tors\Big(X (\Fac(\Sr) \cup X_{\blambda}(\Sr^p)\Big)
  = \Tr_\Lambda(\Sr,\blambda).
  \end{align*}
Here, the first equality follows from the defining property of join of torsion classes and the second equality follows from Lemma \ref{remfilt}.
\end{proof}

The first main step is to show that the map $\Tr_\Lambda$ is a morphism of complete join-semilattice.

\begin{lemma} \label{createparam}
 Consider a confined string of the form $\st \gamma = \st{\alpha q^{-} \beta r \alpha}$. Then, for any $M \in \mod k[T, T^{-1}]$,  we have $X_M(\st{{}^\infty(\alpha q^{-} \beta r)^\infty}) \in \Tors(X(\st \gamma))$.
\end{lemma}

\begin{proof}
 Let $\mt{d}$ be the following almost positive crossing:
 \[\crosswn{\gamma}{}{\gamma}{}{{}^\infty(\alpha q^{-} \beta r)^\infty}{{}^\infty(\alpha q^- \beta r)}{(q^- \beta r \alpha)^\infty}\]
 and let $\{m_1, m_2, \dots, m_n\}$ be a $k$-basis of $M$. Then, using notations of Theorem \ref{basisMorphBand}, it is immediate that $(g_{\mt{d}, m_i}^{\mathrm{bs}})_{1 \leq i \leq n}$ is an epimorphism from $X(\st \gamma)^{\oplus n}$ to $X_M(\st{{}^\infty(\alpha q^{-} \beta r)^\infty})$.
\end{proof}

\begin{lemma} \label{getband}
 Let $\st \gamma$ be a periodic string, $M \in \mod k[T, T^{-1}]$ and $\Sr \in \maxstr$.
 If $c^+(\st \gamma, \st\epsilon) = \emptyset$ for all $\st\epsilon\in \Sr$ and $c^+(\st\delta, \st \gamma)\neq \emptyset$ for some $\st\delta\in \Sr$, then $X_M(\st \gamma) \in \Tors(X (\Fac(\Sr)))$.
\end{lemma}

\begin{proof}
 Consider $\st \delta \in \Sr$ as given in the assumption.  Then we can  choose a primitive period $\omega = \gamma_a r^{-} \gamma_b q$ of $\gamma$ in such a way that a positive crossing from $\st \delta$ to $\st \gamma$ can be written as
 \[\crosswn{\gamma}{\gamma_1 q}{\omega^m \gamma_a}{r^{-} \gamma_bq\omega^\infty}{\delta}{\delta_1 q'^{-}}{r' \delta_2.}\]
 Using the second row, we get that $\st{\omega^m \gamma_a} \in \Fac(\Sr)$, whereas the first row gives $\st{\gamma_b} \in \Fac(\{\st{\gamma_1}\}) \subseteq \Fac(\{\st \gamma\})$. 

 As $c^+(\st \gamma, \epsilon) = \emptyset$ for all $\st\epsilon\in \Sr$, it follows form Lemma \ref{condp1} that $c^+(\st\eta, \st\epsilon)=\emptyset$ for all $\st\eta\in L(\Tors(\st\gamma))\supset \Sr'$, where $\Sr'\in \maxstr$ is the maximal non-crossing set corresponding to $\Tors(\st\gamma)$.
 Hence, by Theorem \ref{mainthmstrings} (b), we get that $\Tors(\st \gamma) \subseteq \Tors(\Sr)$, and it follows that $\st{\gamma_b}$ is in $\Tors(\Sr)$.

 Therefore, the string $\st \epsilon := \st{\omega^m \gamma_a r^{-} \gamma_b q \omega^m \gamma_a} = \st{\gamma_a  r^{-} \gamma_b q \omega^{2m} \gamma_a}$ is in $\Tors(\Sr)$.  By Lemma \ref{createparam}, we have $X_M(\st \gamma) \in \Tors(X(\st \epsilon))$.
 But the previous sentence says that $\Tors(X(\st \epsilon)) \subseteq \Tors(X(\Tors(\Sr))) = \Tors(X(\Fac(\Sr)))$, where the equality is from Lemma \ref{remfilt}.
\end{proof}

\begin{proposition} \label{stjoin}
 For a family $(\Sr_i, \blambda_i)_{i \in I}$ of $\maxstr[k]$, we have
 \[\bigvee_{i \in I} \Tr_\Lambda(\Sr_i, \blambda_i) = \Tr_\Lambda\left( \bigvee_{i \in I} (\Sr_i, \blambda_i)\right).\]
\end{proposition}

\begin{proof}
 Let $(\Sr, \blambda) := \bigvee_{i \in I} (\Sr_i, \blambda_i)$ and fix some $i \in I$.  
 We first show that $\bigvee_{i \in I} \Tr_\Lambda(\Sr_i, \blambda_i) \subseteq \Tr_\Lambda( \Sr, \blambda)$.  By Lemma \ref{facttot}, it suffices to show that $X(\Tors(\Sr_i))\cup X_{\blambda_i}(\Sr_i^p) \subseteq \Tr_\Lambda(\Sr,\blambda)$.  Let us consider the string modules first. By Proposition \ref{parametrized} and Theorem \ref{mainthmstrings}, we have $\Tors(\Sr_i)\subseteq \Tors(\Sr)$, and so $X(\Tors(\Sr_i))\subseteq X(\Tors(\Sr))\subseteq \Tr_\Lambda(\Sr,\blambda)$.

 Let us consider now the case $\st \gamma \in \Sr_i^p$ and $M \in \blambda_i(\gamma)$.  There are two possibilities - the first being $\st \gamma \in \Sr^p$, in which case, we have $M \in \blambda(\st \gamma)$ by Proposition \ref{parametrized}, and so $X_M(\st \gamma) \in \Tr_\Lambda(\Sr, \blambda)$.  For the second possibility, i.e. $\st \gamma \notin \Sr^p$, by the description of $\Sr \geq \Sr_i$ in Theorem \ref{mainthmstrings}, we have $c^+(\st \gamma, \st\epsilon) = \emptyset$ for all $\st\epsilon\in \Sr$ whereas there is some $\st\delta\in \Sr$ so that $c^+(\st\delta, \st \gamma) \neq \emptyset$.  It now follows from Lemma \ref{getband} that $X_M(\st \gamma) \in \Tors(X(\Fac(\Sr)))$.
 Consequently, by Lemma \ref{remfilt}, we have $X_M(\st\gamma)\in \Tors(X(\Tors(\Sr))) \subset \Tr_\Lambda(\Sr, \blambda)$.

 Now it remains to show the containment $\bigvee_{i \in I} \Tr_\Lambda(\Sr_i, \blambda_i) \supseteq \Tr_\Lambda( \Sr, \blambda)$.
 First observe that $\Tors(\Sr) = \Tors\left(\bigcup_{i\in I} \Tors(\Sr_i)\right)= \Tors\left(\bigcup_{i \in I} \Sr_i\right)$, so combining with Lemma \ref{filtfac} yields 
\[
  X(\Tors(\Sr)) =  X( \Tors\left(\bigcup_{i\in I} \Sr_i\right) ) =  X( \Filt\Fac\left(\bigcup_{i\in I} \Sr_i\right) ).  \]
  Using the fact that $\Fac\left(\bigcup_{i\in I}\Sr_i\right) = \bigcup_{i\in I} \Fac(\Sr_i)$ and Lemma \ref{remfilt}, we can then see that $X(\Filt\Fac\left(\bigcup_{i\in I} \Sr_i\right))$ is equal to
 \[
 X( \Filt\left(\bigcup_{i \in I} \Fac (\Sr_i)\right) ) = X( \bigcup_{i \in I} \Fac (\Sr_i) ) = \bigcup_{i \in I} X(  \Fac (\Sr_i) ).
 \]
 So combining with the defining property of join in the complete lattice of torsion classes, we get that
 \[
   \Tors(X(\Tors(\Sr))) = \Tors \left( \bigcup_{i \in I} X(  \Fac (\Sr_i) )\right) =  \bigvee_{i \in I} \Tors \left( X( \Fac (\Sr_i) )\right),
 \]
 which is a subset of $\bigvee_{i \in I} \Tr_\Lambda(\Sr_i, \blambda_i)$ by definition of $\Tr_\Lambda$.

 Suppose now that $\st \gamma \in \Sr^p$ and $M \in \blambda(\st \gamma)$. By definition of $\blambda$, there exists $i \in I$ such that $\st \gamma \in \Sr_i^p$ and $M \in \blambda_i(\st \gamma)$.  Hence, for such an $i$, we have $X_M(\st \gamma) \in \Tr_\Lambda(\Sr_i, \blambda_i)$. This finishes the proof of $\Tr_\Lambda(\Sr, \blambda) \subseteq \bigvee_{i \in I} \Tr_\Lambda(\Sr_i, \blambda_i)$. 
\end{proof}

\subsection{Completely join-irreducible torsion classes}

The next ingredient needed is that completely join-irreducible torsion classes can be realised by $\Tr_\Lambda$; hence, we get that the morphism $\Tr_\Lambda$ of complete join-semilattice is surjective.

\begin{lemma} \label{getbackstrings}
 We have $\Tors(X(\Tors(\{\st\gamma\}))) \subset \Tors(X_M(\st\gamma))$ for any non-null periodic string and any $M\in\mod k[T,T^{-1}]$.
\end{lemma}

\begin{proof}
 We have $\Tors(X(\Tors(\{\st\gamma\}))) = \Tors(X(\Filt\Fac(\{\st\gamma\}))) = \Tors(X(\Fac(\{\st\gamma\})))$, where the first equality follows from Lemma \ref{filtfac} and the second follows from Lemma \ref{remfilt}.
 So it is enough to show that $X(\st\delta)$ with $\st\delta\in\Fac(\{\st\gamma\})$ is also in $\Tors(X_M(\st\gamma))$.

 This assumption implies that we have a decomposition $\gamma = \gamma_1 q^{-}\delta r \gamma_2$ for some arrows $q,r$.  Let us write $\delta = \omega^n \omega_1$ where $\omega = \omega_1 r \omega_2 q^{-}$ is a primitive period of $\gamma$. We prove by induction on $n\geq 0$ that $X(\st \delta) \in \Tors(X_M(\st \gamma))$.  For $n=0$, this follows immediately from Lemma \ref{trivext} (ii).  For $n>0$, it follows by induction hypothesis that  $X (\st {\omega^{n-1} \omega_1}) \in  \Tors(X_M(\st \gamma))$, and hence so is $M\tens_k X (\st {\omega^{n-1} \omega_1}) \cong X (\st {\omega^{n-1} \omega_1})^{\oplus \dim_k M}$.  Now the short exact sequence in Lemma \ref{trivext} (ii) implies that $M \tens_k X (\st {\omega^{n} \omega_1}) \in  \Tors(X_M(\st \gamma))$. Again, since $M \tens_k X (\st {\omega^{n} \omega_1}) \cong X (\st {\omega^{n} \omega_1})^{\oplus \dim M}$, we conclude that the direct summand (which is a factor) $X (\st {\omega^{n} \omega_1}) \in \Tors(X_M(\st \gamma))$.
\end{proof}

\begin{proposition} \label{brickposs}
 If $S\in \mod\Lambda$ is a brick, then there exists some $(\Sr, \blambda) \in \maxstr[k]$ such that $\Tors (S) = \Tr_\Lambda(\Sr, \blambda)$.
\end{proposition}

\begin{proof}
 First of all, suppose that $S$ is a string module, say $S = X(\st \gamma)$. By Proposition \ref{tomax}, there exists $\Sr \in \maxstr$ such that $\Tors(\Sr) = \Tors(\{\st \gamma\})$. Then, by Lemma \ref{facttot}, $\Tr_\Lambda(\Sr, \blambda) = \Tors(S)$ where $\blambda(\st \delta) = \emptyset$ for any $\st \delta \in \Sr^p$.

 Suppose now that $S$ is a band module, say $S = X_M (\st \gamma)$. By Corollary \ref{classbrick}, $\st \gamma$ is not self-crossing and $M$ is in $B$. As in the previous paragraph, there exists $\Sr \in \maxstr$ such that $\Sr$ contains $\st\gamma$ and $\Tors(\Sr) = \Tors(\{\st \gamma\})$.  In particular, we have $\st \gamma \in \Sr^p$. Consider $\blambda: \Sr^p \to 2^B$ defined by $\blambda(\st \gamma) = \{M\}$ and $\blambda(\st \delta) = \emptyset$ for all $\st \delta \neq \st \gamma$. It is immediate by definition that $\Tors(S) \subseteq \Tr_\Lambda(\Sr, \blambda)$.  Conversely, we have by Lemma \ref{remfilt} and the assumption on $(\Sr,\blambda)$ that
\[ \Tr_\Lambda(\Sr, \blambda)=\Tors\Big(X(\Tors(\Sr))\cup X_{\blambda}(\Sr^p)\Big)=\Tors\Big(X(\Tors(\{\st\gamma\}))\cup \{S\}\Big), \] and also $X(\Tors(\{\st\gamma\})) \subseteq \Tors(S)$ by Lemma \ref{getbackstrings}, so we conclude that $\Tors(S) = \Tr_\Lambda(\Sr, \blambda)$.
\end{proof}

\subsection{Preserving strictness of finite chains}

In this subsection, we show that $\Tr_\Lambda$ preserves strictness of finite chains, and so it is an injective map by Lemma \ref{injmor}.
For this purpose, let us recall how one can tell whether a chain of two torsion classes $\Tr\supseteq \Tr'$ is strict.  Recall that for a full subcategory $\Xr\subseteq\mod \Lambda$, we have a right perpendicular class
\[
\Xr^\perp := \{ Y\in \mod\Lambda\mid \Hom_\Lambda(X,Y)=0\}.
\]
We remind the reader the we can determine whether two comparable torsion classes $\Tr\supseteq \Tr'$ are distinct, namely, $\Tr \supsetneqq \Tr'\Leftrightarrow \Tr\cap \Tr'^\perp\neq 0$.
Indeed, by the defining property of torsion class, any $M\in \Tr\setminus \Tr'$ admits a short exact sequence $0\to X \to M\to Y \to 0$ with $X\in \Tr'$ and $Y\in \Tr'$; hence, we have $Y\in \Fac(M)\subset \Tr$.

The following lemma simplifies the verification procedure a module in the right perpendicular class of a torsion class.

\begin{lemma} \label{shortermorp}
 Let $\Xr$ be a set of (isoclass representatives of) modules in $\mod\Lambda$ and $M \in \mod \Lambda$. If there exists $Y \in \Tors(\Xr)$ such that $\Hom_\Lambda(Y, M) \neq 0$ then there exists $X \in \Xr$ such that $\Hom_\Lambda(X, M) \neq 0$.
\end{lemma}

\begin{proof}
  Take a module $Y\in \Tors(\Xr)$ so that $\Hom_\Lambda(Y,M)\neq 0$ and it is of smallest possible dimension with respect to this property.   We claim that $Y \in \Fac(\Xr)$.  Indeed, if this is not the case, then it follows from the definition of torsion classes that there exists a short exact sequence 
  $0 \to Y_1 \to Y \to Y_2 \to 0$
 with $Y_1, Y_2 \in \Tors(\Xr) \setminus \{0\}$.  Applying $\Hom_\Lambda(-, M)$ to this sequence gives $\Hom_\Lambda(Y_1, M) \neq 0$ or $\Hom_\Lambda(Y_2, M) \neq 0$, contradicting the minimality of $Y$.

  By definition of $\Fac(\Xr)$, there exists $X \in \Xr$ with an epimorphism $X \twoheadrightarrow Y$, where $Y$ is chosen as in the previous paragraph.  It now follows that $\Hom_\Lambda(X, M) \neq 0$, as required.
\end{proof}

In the case of gentle algebras, Lemma \ref{shortermorp} provides some interesting combinatorial properties.

\begin{lemma}\label{homfromtor}
Consider $(\Sr,\blambda)\in \maxstr[k]$ and $S\in \mod\Lambda$ so that there is a non-zero morphism from $\Tr_\Lambda(\Sr,\blambda)$ to $S$.
\begin{enumerate}[\rm (i)]
\item If $S=X(\st\omega)$ for some confined string $\st\omega$, then there is an almost positive crossing from some $\st\delta\in \Sr$ to $\st\omega$.

\item If $S=X_M(\st\omega)$ such that $\st\omega\notin \Sr^p$ or $M\notin \blambda(\st\omega)$, then there is a positive crossing from some $\st\delta\in \Sr$ to $\st\omega$.
\end{enumerate}

\end{lemma}
\begin{proof}
(i) By Lemma \ref{shortermorp} and the definition of $\Tr_\Lambda$, there is some (indecomposable) module $S\in X(\Fac(\Sr))\cup X_{\blambda}(\Sr^p)$ so that $\Hom_\Lambda(S, X(\st\omega))\neq 0$.

If $S\in X_{\blambda}(\Sr^p)$, say $S=X_M(\st\delta)$ with $\st\delta\in \Sr^p$ and $M\in \blambda(\st\delta)$, then the statement follows from Theorem \ref{basisMorphBand}.
On the other hand, if $S\in X(\Fac(\Sr))$, say $S=X(\st\epsilon)$ with $\st\epsilon\in \Fac(\Sr)$, then we have an almost positive crossing in $\mt{d}\in c^\geq(\st\epsilon,\st\omega)$ by Theorem \ref{basisMorphString}.  We can extend this crossing naturally to an almost crossing of the form as claimed (i.e. the overlap $\omega'$ is the same overlap in $\mt{d'}$). 

(ii) The argument is similar to (i), except that the use of Theorem \ref{basisMorphString} is replaced by Theorem \ref{basisMorphBand}.
This also means that the conditions on $\st\omega$ guarantees the morphism from $S$ to any module in $X_{\blambda}(\Sr^p)$ does not have the canonical morphism $h_\beta$ as its constituent.  This means that the crossing from $\st\delta$ to $\omega$ must be positive.
\end{proof}

\begin{lemma} \label{separate}
 Let $(\Sr, \blambda), (\Sr', \blambda') \in \maxstr[k]$ satisfying $(\Sr, \blambda) > (\Sr', \blambda')$. 
 \begin{enumerate}[\rm (a)]
  \item For any $\st \gamma \in \Sr$, $\st{\gamma'} \in \Sr'$ and $\mt{d} \in c^+(\st \gamma, \st \gamma')$, we have \[X(\st \omega) \in \Tr_\Lambda(\Sr, \blambda) \cap \Tr_\Lambda(\Sr', \blambda')^\perp\] where $\omega$ is the overlapping part of $\mt{d}$.
  \item For any $\st \gamma \in \Sr^p$ and $M \in \blambda(\st \gamma)$ satisfying either $\st \gamma \notin \Sr'^p$ or $M \notin \blambda'(\st \gamma)$, we have $X_M(\st \gamma) \in \Tr_\Lambda(\Sr, \blambda) \cap \Tr_\Lambda(\Sr', \blambda')^\perp$.
 \end{enumerate}
\end{lemma}

\begin{proof}
 (a) Consider the crossing 
  \[\mt{d}:\crosswn{\gamma}{\gamma_1 q^{-}}{\omega}{r \gamma_2}{\gamma'}{\gamma'_1 q'}{r'^{-} \gamma'_2} \]
  as given by the assumption.
  Using the first row, we have $X(\st \omega) \in \Tr_\Lambda(\Sr, \blambda)$.  

  Suppose on the contrary that there is a non-zero morphism from a module of $\Tr_\Lambda(\Sr', \blambda')$ to $X(\st \omega)$, then applying Lemma \ref{homfromtor} (i) yields an almost positive crossing of the form
  \[\mt{d'}: \crosswn{\delta}{\delta_1 s'^-}{\omega'}{t' \delta_2}{\omega}{[\omega_1 s]}{[t^{-} \omega_2]}\]
  with $\st\delta\in \Sr'$.
  Now, combining with the previous crossing $\mt{d}$, we obtain a positive crossing
  \[\crosswn{\delta}{\delta_1 s'^{-}}{\omega'}{t' \delta_2}{\gamma'}{\gamma'_1 q' [\omega_1 s]}{[t^{-} \omega_2] r'^{-} \gamma'_2}\]
  between strings in $\Sr'$.  This contradicts the assumption that $\Sr'$ is a non-crossing set of strings.

 (b) It is immediate that $X_M(\st \gamma) \in \Tr_\Lambda(\Sr, \blambda)$.  Suppose on the contrary that there is a non-zero morphism from a module in $\Tr_\Lambda(\Sr', \blambda')$ to $X_M(\st \gamma)$, then it follows from Lemma \ref{homfromtor} that there is a positive crossing from some $\delta'\in \Sr'$ to $\st\gamma$.  But this contradicts the assumption that $\Sr \geq \Sr'$. 
\end{proof}

\subsection{The proof}

Now we have everything we need to prove the classification of $\tors(\fl\Lambda)$ by non-crossing sets.

\begin{proof}[Proof of Theorem \ref{classtors}] \eqloc{classtors}
 As mentioned before (see Remark \ref{rem:lattice}), it suffices to show that $\Tr_\Lambda:\maxstr[k]\to \tors(\fl\Lambda)$ is (1) a morphism of complete join-semilattices and also (2) a bijective map.

 For (1), this follows immediately from Proposition \ref{stjoin}.

 For (2), we first explain why $\Tr_\Lambda$ is surjective.  By Proposition \ref{brickposs}, any torsion class of the form $\Tors(S)$ is in the image $\Tr_\Lambda$.  But such form of torsion classes are precisely the completely join-irreducible elements of $\tors(\fl\Lambda)$ by Proposition \ref{torclass1} (ii).  Hence, surjectivity follows from Lemma \ref{bricksurj}.

 Finally, to show that $\Tr_\Lambda$ is injective, start by considering $(\Sr, \blambda) > (\Sr', \blambda')$ in $\maxstr[k]$.
 Then $\Tr_\Lambda (\Sr, \blambda) \supseteq \Tr_\Lambda (\Sr', \blambda')$ as we have shown that $\Tr_\Lambda$ is order-preserving. By Lemma \ref{separate}, there is a module in $\Tr_\Lambda(\Sr, \blambda) \cap \Tr_\Lambda(\Sr', \blambda')^\perp$.  So it follows from the discussion preceding Lemma \ref{separate} that $\Tr_\Lambda (\Sr, \blambda) \supsetneq \Tr_\Lambda (\Sr', \blambda')$.  Now injectivity follows from Lemma \ref{injmor}.
\end{proof}

\section{Application: Torsion classes of Brauer graph algebras}\label{sec:BGA}

We give an elementary proof for a generalisation of \cite[Corollary 5.20]{DIRRT}. 
\begin{proposition}\label{torsMorph}
Suppose $A$ is an algebra and $I$ is an ideal of $A$.
If $SI=0$ for every brick $S\in \brick(\fl A)$ and $I=\langle r_1, r_2, \ldots, r_n\rangle$ for some $r_1, \ldots, r_n\in A$, then there is an isomorphism of complete lattice
\[
\xymatrix@R=8pt{
\tors(\fl A) \ar[r]^{\sim\;\;} &  \tors(\fl A/I)\\
 \Tr \ar@{|->}[r]& \Tr\otimes_A A/I \\
 {\Ur}^\wedge & \Ur \ar@{|->}[l],
 }
\]
where ${\Ur}^\wedge:=\{X\in \fl A\mid X/XI\in \Ur\}$.
\end{proposition}
\begin{proof}
Let us first explain how the maps are well-defined (c.f. \cite[Prop 5.1]{DIRRT}).  We regard $\fl A/I$ as a full subcategory of $\fl A$ (as we have a canonical embedding  $\Mod A/I \to \Mod A$ induced by the algebra morphism $A\to A/I$).  Then we can write $\Tr\otimes_A A/I$ as $\Tr\cap \fl A/I$.  It is easy to see that $\Tr\cap\fl A/I\in \tors(\fl A/I)$ and this construction preserves the inclusion order.

To show that ${\Ur}^\wedge$ is quotient-closed for $\Ur\in \tors(\fl A)$, consider an epimorphism $X\twoheadrightarrow Y$ with $X/XI= X\otimes_A A/I \in \Ur$.  Since $-\otimes_A A/I$ is right exact, we have an epimorphism $X/XI \twoheadrightarrow Y/YI$ in $\fl A/I$, and so $Y/YI \in \Ur$, meaning that $Y\in {\Ur}^\wedge$.

To see that ${\Ur}^\wedge$ is extension-closed, consider a short exact sequence $0\to X \xrightarrow{f} M \to Y \to 0$ in $\fl A$ with $X/XI, Y/YI\in \Ur$.  Applying $-\otimes_A A/I$ yields a short exact sequence
\[
0\to \im(f\otimes_AA/I) \to M/MI \to Y/YI \to 0.
\]
Since $\im(f\otimes_AA/I)$ is a quotient of $X/XI$, it is in $\Ur$, and so $M/MI\in \Ur$ as required.

Let us now show that ${\Ur}^\wedge\cap \fl A/I = {\Ur}^\wedge$.  Indeed, by the assumption of $SI=0$ for all $S\in \brick(\fl A)$, $-\otimes_A A/I$ induces a bijection between $\brick(\fl A)$ and $\brick(\fl A/I)$ (as any $a\in I$ induces an endomorphism $-\cdot a:X\to X$ on any $A$-module $X$).  Hence, we have \[
\brick({\Ur}^\wedge\cap \fl A/I) = \brick({\Ur}^\wedge)\cap \fl A/I = \brick(\Ur)\cap \fl A/I = \brick(\Ur).
\]
This implies that
\[
\Ur=\Filt(\brick(\Ur))=\Filt(\brick({\Ur\cap \fl A/I})) = (\Ur\cap \fl A/I)^\wedge.
\]
In particular, $\Ur \mapsto {\Ur}^\wedge$ is also a morphism of lattice.

It now remains to show that $(\Tr\otimes_A A/I)^\wedge=\Tr$.  Indeed, it follows immediate from the definitions that $(\Tr\otimes_A A/I)^\wedge \supseteq \Tr$.  For the other direction, consider $X\in (\Tr\otimes_A A/I)^\wedge$, we have $X/XI\in \Tr\otimes_A A/I$ and so $X/XI \in \Tr$.

By the assumption of $I=\langle r_1, r_2, \ldots, r_n\rangle$, we have an epimorphism $(-\cdot r_1, -\cdot r_2, \cdots, -\cdot r_n): X^{\oplus n} \twoheadrightarrow XI$.  Hence, for any integer $a\geq 0$, we have
\[
(XI^a/XI^{a+1})^{\oplus n} = \frac{(XI^a)^{\oplus n}}{(XI^{a+1})^{\oplus n}} \twoheadrightarrow \frac{XI^{a+1}}{XI^{a+2}}. 
\]
As $X/XI\in \Tr$, these epimorphisms imply that $XI^a/XI^{a+1}\in \Tr$ for all $a\geq 0$.  In particular, as $X$ is has a (finite) filtration $(XI^a)_{a\geq 0}$, it is also in $\Tr$.
\end{proof}

Let us now recall the definition of Brauer graph algebras.

\begin{definition}
A \defn{ribbon graph order} is an infinite-dimensional gentle algebra associated $(Q,R)$ so that any blossoming of $(Q,R)$ is $(Q,R)$ itself.
\end{definition}

For a ribbon graph order associated to $(Q, R)$ and $i\in Q_0$, we have exactly two incoming arrows $q,q'$ and two outgoing arrows $r,r'$ with $qr\neq 0\neq q'r'$.   This gives rise to two cycles (walks) $\gamma, \gamma'$ starting with the letters $r, r'$ respectively so that $\st{{}^\infty \gamma^\infty}, \st{{}^\infty \gamma'^\infty}\in \nstr$.  We call $\gamma, \gamma'$ the \defn{Brauer cycles at $i$}.

A \defn{multiplicity function} on a ribbon graph order is a map $m:\nstr \to \Z_{>0}$.  A \defn{Brauer commutation relation} with respect to a multiplicity function $m$ on a ribbon graph order $\Lambda$ is an element of the form 
\[
\gamma^{m(\gamma)}-\gamma'^{m(\gamma)}
\]
with $\gamma,\gamma'$ being the two distinct Brauer cycles at some $i\in Q_0$.

\begin{definition}
A finite-dimensional algebra $A$ is a \defn{Brauer graph algebra} if there is a multiplicity function $m$ such that it is Morita equivalent to the quotient of a ribbon graph order by the ideal generated by all Brauer commutation relations with respect to $m$.
\end{definition}
\begin{remark}
Usually in the literature, Brauer graph algebras are defined using a ribbon graph equipped with a multiplicity function defined on the vertices of the ribbon graph; ours are slightly different but it is easy to see the definitions are equivalent as vertices of the associated ribbon graph correspond bijectively with the null strings.
\end{remark}

Now we can obtain classification of torsion classes in the category $\mod A=\fl A$ of finite-dimensional modules over a Brauer graph algebra $A$.

\begin{corollary}
Let $A \cong \Lambda/I$ be a Brauer graph algebra arising from a ribbon graph order $\Lambda$ associated to $(Q,R)$.
Then there is a complete lattice isomorphism
 \begin{align*}
\Tr_\Lambda:\maxstr[k] & \xrightarrow{\sim} \tors(\mod A) \\
(\Sr, \blambda) &\mapsto \Tr_\Lambda(\Sr,\blambda)\cap \mod A.
 \end{align*}
\end{corollary}
\begin{proof}
Combine Theorem \ref{classtors} with Proposition \ref{torsMorph}.
\end{proof}

\bibliographystyle{alphanum}
\bibliography{gentle}

\end{document}